\documentclass[a4paper,11pt]{article}

\usepackage{natbib}                                 
\usepackage{fullpage}
\usepackage{setspace}
\usepackage{palatino}
\usepackage[usenames, dvipsnames]{color}                        
\usepackage[pdftex]{graphicx}
\usepackage{amsfonts, amsmath, amssymb}
\usepackage{pjthm}
\usepackage{pjnotation}
\usepackage{booktabs}
\usepackage{float}
\usepackage{subcaption}

\usepackage{rotating}
\usepackage{array}

\usepackage[many]{tcolorbox} 
\tcbset{colback=green!20,colframe=red,width=2em,height=12pt,halign=center,valign=bottom,before skip=0pt,after skip=0pt,left=0pt,right=0pt,top=0pt,bottom=2pt,boxsep=0pt,arc=0pt,outer arc=0pt,boxrule=0pt,tcbox raise=-2pt,nobeforeafter} 

\newcommand{\myboxx}[1]{\begin{tcolorbox}[colback=gray!60] {\bfseries #1} \end{tcolorbox}}
\newcommand{\myboxy}[1]{\begin{tcolorbox}[colback=gray!40] #1 \end{tcolorbox}}

\newcommand{\myboxyyy}[1]{\begin{tcolorbox}[colback=gray!20] #1 \end{tcolorbox}}

\newcommand{\myboxyyyyy}[1]{\begin{tcolorbox}[colback=gray!0] #1 \end{tcolorbox}}

\usepackage{hyperref}                                           
\hypersetup{
bookmarksnumbered,
pdfstartview={FitH},
linkcolor=blue,
citecolor=blue,
colorlinks=true,
pdfpagemode={UseNone},
plainpages=false
}%

\newtheorem{theorem}{Theorem}
\newtheorem{lemma}{Lemma}

\theoremstyle{definition}

\newtheoremstyle{rule}{}{}{\itshape}{}{\bfseries}{:}{.5em}{\thmnote{#3}}
\theoremstyle{rule}

\newcommand\totalNumber{\mathprobfont N}         
\newcommand\totalProportion{\mathprobfont P}         
\newcommand\totalRegret{\mathprobfont R}         
\newcommand\successes{s} 
\newcommand\failures{f} 

\newcommand\design[1]{\textsc{\sffamily\mdseries #1}}

\hypersetup{
pdftitle={Title},
pdfauthor={Author}, 
pdfkeywords={}
}%

\begin{document}

\author{Peter Jacko \\ Department of Management Science \\ Lancaster University, UK}

\title{The Finite-Horizon Two-Armed Bandit Problem \\ with Binary Responses
\\ \ \\
\Large A Multidisciplinary Survey of the History, State of the Art, and Myths
}

\date{June 18, 2019}

\maketitle

\begin{abstract}
In this paper we consider the two-armed bandit problem, which often naturally appears per se or as a subproblem in some multi-armed generalizations, and
serves as a starting point for introducing additional problem features. The consideration of binary responses is motivated by its widespread
applicability and by being one of the most studied settings. We focus on the undiscounted finite-horizon objective, which is the most relevant in many
applications. We make an attempt to unify the terminology as this is different across disciplines that have considered this problem, and present a
unified model cast in the Markov decision process framework, with subject responses modelled using the Bernoulli distribution, and the corresponding Beta
distribution for Bayesian updating. We give an extensive account of the history and state of the art of approaches from several disciplines, including
design of experiments, Bayesian decision theory, naive designs, reinforcement learning, biostatistics, and combination designs. We evaluate these
designs, together with a few newly proposed, accurately computationally (using a newly written package in Julia programming language by the author) in
order to compare their performance. We show that conclusions are different for moderate horizons (typical in practice) than for small horizons (typical
in academic literature reporting computational results). We further list and clarify a number of myths about this problem, e.g., we show that,
computationally, much larger problems can be designed to Bayes-optimality than what is commonly believed.
\end{abstract}

\newpage
\section{Introduction}

Statistical testing based on randomized equal allocation is a widespread state-of-the-art approach in the design of experiments for around $ 100 $ years,
known today as the \emph{randomized controlled trial} in biostatistics, the \emph{between-group design} in social sciences, and the \emph{A/B testing} in
digital marketing. Already \citet{Thompson1933biometrika}, a biostatistician from Yale University, proposed a data-driven approach which would in
expectation lead to a higher reward from an experiment, using the following words:
\begin{quote}
``...there can be no objection to the use of data, however meagre, as a guide to action required before more can be collected ... Indeed, the fact that
such objection can never be eliminated entirely---no matter how great the number of observations---suggested the possible value of seeking other modes of
operation than that of taking a large number of observations before analysis or any attempt to direct our course...''
\end{quote}
\citet{Robbins1952}, a prominent mathematician and statistician, emphasized that this problem is of a much wider importance:
\begin{quote}
``In fact, the problem represents in a simplified way the general question of how we learn---or should learn---from past experience.''
\end{quote}

A formulation of the problem using the Bayesian decision-theoretic framework allows for Bayes-optimality. Practical application of this Bayesian approach
has however been long hindered by its computational complexity, since the optimal solution is known in analytical form only for infinite horizon
\citep{Kelly1981}. A variety of practical approximations and heuristics have been developed and studied across several disciplines in order to overcome
this issue, but their analysis failed to give exact results or bounds sufficiently close to Bayes-optimality for finite horizon problems, which are the
problems most relevant to many situations in practice.

\subsection{Paper Structure and Contributions}

In this paper we thus focus on the finite-horizon setting. We also restrict the discussion to two arms, which often naturally appears per se or as a
subproblem in some multi-armed generalizations (e.g. if new arms appear over time), and serves as a starting point for introducing additional problem
features. The consideration of binary responses is motivated by its widespread applicability and by being one of the most studied settings.

Our main objective is to give an account of modelling and solution approaches arising in different disciplines, in a unified framework and using a unified
terminology. The problem description, origins and terminology is given in \autoref{section:problem}. Our unified model is described in
\autoref{section:model}, cast in the Markov decision process framework, with subject responses modelled using the Bernoulli distribution, and the
corresponding Beta distribution for Bayesian updating. Different problem settings, assumptions and objectives are summarized in \autoref{section:settings}.
\autoref{section:approaches} gives an account of the history and state of the art of approaches from several disciplines. In \autoref{section:performance}
we evaluate these designs, together with a few newly proposed, accurately computationally (using a newly written package in Julia programming language by
the author) in order to compare their performance, showing that conclusions are different for moderate horizons (typical in practice) than for small
horizons (typical in academic literature reporting computational results). We further list and clarify a number of myths about this problem in
\autoref{section:myths}, e.g., we show that, computationally, much larger problems can be designed to Bayes-optimality than what is commonly believed.
\autoref{section:conclusion} concludes.

\section{Problem}
\label{section:problem}

We consider the problem with two \emph{arms} (or, \emph{interventions}), called $ C $ (mnemonically for ``control'' or ``comparator'' or standard of
``care'') and $ D $ (for ``discovery'' or ``development'').\footnote{In the existing literature, it is common to denote the two arms as $ 1 , 2 $ (but we
prefer to keep these names for actions defined below) or $ A , B $ (but we prefer to keep $ A $ for the action process and $ B $ for the Beta function).} $
T $ \emph{subjects} become available one by one, and each subject must be allocated to exactly one of the arms. Upon allocation of a subject to arm $ C $
($ D $), subject's \emph{response} is observed, which is binary (success/failure), where the success probability is $ \theta_{ C } $ ($ \theta_{ D } $) and
the failure probability is $ 1 - \theta_{ C } $ ($ 1 - \theta_{ D } $). The primary objective is to find a \emph{design}, i.e. a strategy composed of
\emph{randomized actions} of allocating the subjects to arms, which, in expectation, achieves the highest number of observed successes from the $ T $
subjects, assuming that the success probabilities are unknown. A formal model is given in \autoref{section:model}.

\subsection{Problem Origins}

The first statement of the two-armed bandit problem is in \citet{Thompson1933biometrika}, extended in \citet{Thompson1935ajm} to multiple arms, in a
Bayesian setting. Apparently unaware of Thompson's works, \citet{Robbins1952} formulated the two-armed bandit problem in a frequentist setting. Neither
\citet{Thompson1933biometrika, Thompson1935ajm} nor \citet{Robbins1952} used the terms ``arm'' or ``bandit''. The term \emph{two-armed bandit} problem
first appeared in \citet{BradtEtal1956}, referring to the setting with binary responses in which one knows the set $ \{ \theta_{ C } , \theta_{ D } \} $,
but does not know which arm is which. \citet{BradtEtal1956} also proposed a generalization of that problem, in which $ \{ \theta_{ C } , \theta_{ D } \} $
is unknown, which is the two-armed bandit problem as known today and as considered in this paper. \citet{Bellman1956} referred to the latter problem as the
\emph{two-machine problem}.

P. Whittle stated on several occasions that researchers were aware of this type of problem since the 1940s and considered it an important but very hard
open problem.\footnote{``...it was formulated during the war, and efforts to solve it so sapped the energies and minds of Allied analysts that the
suggestion was made that the problem be dropped over Germany, as the ultimate instrument of intellectual sabotage'' \citet{Whittle1979discussion};
``...propounded during the Second World War, and soon recognized as so difficult that it quickly became a classic, and a by-word for intransigence.''
\citet{Whittle1989foreword}; ``...had resisted analysis, however, to the point of being regarded by some as intrinsically insoluble.''
\citet{Whittle2002}.} That could be attributed to the absence of a suitable mathematical framework and theory, as the progress on the problem occurred
slowly alongside the emergence and the development of areas such as sequential analysis, Bayesian statistics, decision theory, dynamic programming,
stochastic processes, and concentration inequalities.

Indeed, early papers describing theoretical solutions on the bandit problem were often among the pioneers in these areas, introducing novel terminology and
notation, not all of which has been adopted more generally, and might thus be hard to read for today's researchers. Drawing on the early research in 1950s
and 1960s, three dominant ``schools'' have emerged:
\begin{itemize}

\item the \emph{Berry's school}: starting with \citet{Berry1972}, and rewritten and further developed in \citet{BerryFristedt1985}, focussing
    predominantly on the finite-horizon setting;

\item the \emph{Gittins' school}: starting with \citet{GittinsJones1974}, and rewritten and further developed in \citet{Gittins1979, Gittins1989,
    GittinsEtal2011}, focussing predominantly on the discounted infinite-horizon setting;

\item the \emph{Robbins' school}: starting with \citet{LaiRobbins1985aam}, and rewritten and further developed in \citet{Agrawal1995aap,
    KatehakisRobbins1995pnas, BurnetasKatehakis1996}, focussing predominantly on the time-average infinite-horizon setting.

\end{itemize}
The author's suggestion is that this pioneering literature should be on the must-read list of researchers on bandit problems, regardless of their
discipline.

While the Robbins' school makes complete learning (i.e. identification of the better arm in infinite time almost surely) lexicographically more important
than the way of attaining it, both the Berry's and Gittins' schools replace the lexicographic ordering by resolving the trade-off between complete
\emph{learning and earning} (of rewards), with the relative weights implicitly given by the horizon and the discount factor, respectively. In
\autoref{section:approaches} we describe these ``schools'' and their relationships in more detail.

Of course, there are several other fascinating variants of the bandit problem, with different objective (e.g., risk-averse, adversarial, final-period-only,
etc.), different control (e.g., randomized, multi-mode, multi-resource, duelling, multi-player, etc.), and/or different dynamics (e.g., non-binary
responses, delayed responses, partial observability, arriving arms, covariates, correlation, restlessness, non-stationarity, non-Markovian, etc.); all
these are unfortunately beyond the scope of this paper.

\subsection{Applications}

The problem has been formulated, addressed or applied in a number of disciplines, each developing its own terminology, see \autoref{table:terminology}. In
this paper we use the terminology which we believe is a reasonable compromise and should not cause confusion for researches and practitioners from all the
disciplines. The author wishes to encourage researchers from all the disciplines to follow this terminology as closely as possible to facilitate for
researchers and practitioners from other disciplines to learn about their work.

According to \citet{Scott2010asmbi}: ``Multi-armed bandits have an important role to play in modern production systems that emphasize \emph{continuous
improvement}, where products remain in a perpetual state of feature testing even after they have been launched.'' The most commonly listed applications,
often requiring to adapt the generic multi-armed bandit problem to specific features, are as follows:
\begin{itemize}

  \item \emph{Digital marketing}: In the digital world it is relatively easy to introduce and quick to get feedback on new document variants, and so
      bandit problems have been proposed for social media advertising, personalized websites and user interfaces, email campaigns, influence
      maximization, etc; see, e.g., \citet{LiberaliEtal2017handbook}. Bandit problems can also be used to address the problem of dynamic pricing with
      demand uncertainty, which requires to solve a trade-off of learning (of the demand curve) and earning (the highest revenue); for a survey, see,
      e.g., \citet{denBoer2015sorms}.

  \item \emph{Clinical trials}: \citet{Thompson1933biometrika} pointed out that his bandit problem ``...would be important in cases where either the rate
      of accumulation of data is slow or the individuals treated are valuable, or both.'' \citet{Gluss1962iac} further explained the motivation primarily
      by \emph{rare diseases}. Following the focus on rare and/or life-threatening diseases, a few novel bandit-based designs have been developed and
      proposed recently, and are being implemented in a growing number of trials, mainly in several types of cancer, where patients are stratified into
      smaller groups using genetic biomarkers. Discussions about the advantages and disadvantages of bandit-based designs are ongoing, e.g.,
      \citet{BerryEsserman2016nejm} argue that, in certain clinical trials, data-driven approaches make great sense ethically, statistically,
      economically, scientifically, and logistically. For a survey on real adaptive trials, see e.g., \citet{BothwellEtal2018bmj}.

  \item \emph{Search}: Bandit designs have been proposed for recommender systems in which new items and users appear frequently in order to assure
      sufficient exploration, see, e.g., \citet{Aggarwal2016book}. Although digital search is typically considered by recommender systems, many
      non-digital search applications exist, e.g. search for natural resources, search and rescue, surveillance and monitoring. Related to this category
      are also the so-called best-arm identification problem and problems appearing in ranking and selection.

\end{itemize}

\begin{table}[tbp]
\setlength{\tabcolsep}{5pt}
\centering
\begin{tabular}{lcccc}
	\toprule
	Anecdotic & strategy & choice & pull & arms \\
    \midrule
	Operations \& Management & policy & allocation & resource & projects \\
	Reinforcement learning & algorithm & decision & time step & actions \\
	Biometrics \& Biostatistics & design & randomization & patient & treatments \\
	Ranking \& selection & policy & spread over & measurement & alternatives \\
	Economics & strategy & choice & resource & experiments \\
	Computing \& Telecom. & scheduler & allocation & server & jobs \\
	Marketing & policy & allocation & impression & advertisements \\
	Transportation & driver & selection & vehicle & roads \\
    \midrule
	This paper & \emph{design} & \emph{randomized action} & \emph{subject} & \emph{interventions/arms} \\
    \bottomrule
\end{tabular}
\caption{An illustration of typical terminology for bandit problems across disciplines}\label{table:terminology}
\end{table}

\section{Model}
\label{section:model}

In this section we formulate a general two-armed problem with binary responses as a Markov decision process, which provides sufficient generality to
accommodate all the solution approaches discussed in \autoref{section:approaches}.

\paragraph{Interventions.}
We consider arms (or, \emph{interventions}) labelled by $ k \in \setK := \{ C , D \} $. A subject must be allocated to exactly one intervention, and such
allocation yields a \emph{binary} response from that intervention: $ 0 $ (failure) or $ 1 $ (success). The response set is denoted by $ \setO := \{ 0 , 1
\} $. Subject responses are uncertain, i.e., modelled as Bernoulli-distributed with parameter $ 0 \le \theta_{ k } \le 1 $, the \emph{success probability},
independent across arms. The responses are \emph{immediate}, meaning that the response of an allocated subject is observed before the next decision needs
to be done.

\paragraph{Timing.}
Subjects arrive (i.e., are recruited) sequentially (i.e., one by one) at random moments in continuous time. Since we do not discount the future, we can
without loss of generality focus only on the moments of subjects' arrivals, which we call discrete \emph{time epochs} and see as regularly spaced. That is,
equivalently, we can consider that subjects arrive at time epochs $ t \in \setT := \{ 0 , 1 , 2 , \dots , T - 1 \} $, where $ T \le +\infty $ is the number
of subjects in the trial, i.e., the \emph{trial size}, or the \emph{time horizon}. To clarify, the $ ( t + 1 ) $-st subject arrives at time epoch $ t $.
Note that $ t = T $ is the time epoch denoting the end of the trial, when the response of the last subject is observed and no subject arrives.

\paragraph{States.}
At any moment in continuous time, the \emph{physical state} is represented by the numbers of observed successes and failures on each arm, the number of
allocated subjects without an observed response on each arm, the number of arrived subjects without being allocated, and the number of remaining subjects
to arrive. This is a vector with $ 8 $ elements summing up to $ T $ at any moment (all are non-negative integers). At time epochs, this can be simplified
without loss of generality to a vector with $ 5 $ elements, with the numbers of observed successes and failures on arm $ C $ denoted by $ s_{ C } $ and $
f_{ C } $, respectively, the numbers of observed successes and failures on arm $ D $ denoted by $ s_{ D } $ and $ f_{ D } $, respectively, and the number
of remaining subjects to be allocated (exactly one of which has arrived), $ n $. Since at time epochs $ s_{ C } + f_{ C } + s_{ D } + f_{ D } + n = T $, it
is sufficient to keep track of any four of these five numbers, leading to a state as vector with $ 4 $ elements, which we choose to be $ \vecx := ( s_{ C }
, f_{ C } , s_{ D } , f_{ D } ) $. Note that at time epoch $ t $, $ s_{ C } + f_{ C } + s_{ D } + f_{ D } = t $.

In addition to the physical state, there is an \emph{information state}, which at any moment in continuous time captures all the information that could
possibly affect the decisions. This may include real-world evidence and/or modelling assumptions. The real world evidence may be available before the start
and/or it can arrive anytime during the trial. The modelling assumptions typically refer to the parameters of the prior distributions (built on historical
data or expert opinions) for the success probability of each arm (whose weight may change over time, and can be either informative or non-informative), but
may also include other parameters such as the probability of dropouts, the probability of errors in recording the observations and/or the subject
allocations, the probability of mistakes in the statistical analysis and/or in the administration process, the timing of planned interim analyses, the
probability and/or timing of unplanned stopping of the trial due to safety concerns, the estimate of the size of the subject population after the end of
the trial, etc. For full generality, we consider the information state $ \veci $ (potentially dependent on the current physical state $ \vecx $ and/or
otherwise changing during the trial), and thus the state is $ ( \vecx , \veci ) $.

\paragraph{Actions.}
At every time epoch $ t \in \setT $ the design must prescribe how the arrived subject should be randomized (i.e., randomly allocated) to interventions.
While there are only two possible allocations, in every state we consider a possibly infinite \emph{action set} $ \setA_{ ( \vecx , \veci ) } $ of
randomized actions $ a $ identified by probabilities $ ( p_{ C }^{ a } , p_{ D }^{ a } ) $, meaning that the subject is allocated to intervention $ C $ ($
B $) with probability $ p_{ C }^{ a } $ ($ p_{ D }^{ a } $). Formally, $ \setA_{ ( \vecx , \veci ) } \subseteq \{ a : p_{ C }^{ a } \ge 0 , p_{ D }^{ a }
\ge 0 , p_{ C }^{ a } + p_{ D }^{ a } = 1 ) \}. $ Since from the theory of Markov decision processes it follows that an action which is a randomized
combination of other two actions is optimal only if all three are optimal, it is sufficient to consider only an action set of two \emph{pure randomized
actions}, which we call action $ 1 $ (identified by $ ( p_{ C }^{ 1 } , p_{ D }^{ 1 } ) $) and action $ 2 $ (identified by $ ( p_{ C }^{ 2 } , p_{ D }^{ 2
} ) $). For convenience in situations when both actions are optimal, we also consider an \emph{equally-weighted mixed randomized action}, which we call
action $ 3 $ (identified by $ ( p_{ C }^{ 3 } , p_{ D }^{ 3 } ) := ( ( p_{ C }^{ 1 } + p_{ C }^{ 2 } ) / 2 , ( p_{ D }^{ 1 } + p_{ D }^{ 2 } ) / 2 ) $),
which is a combination of the two pure randomized actions with equal weights. Formally, $ \setA_{ ( \vecx , \veci ) } = \{ 1 , 2 , 3 \} $, and without loss
of generality we assume $ p_{ C }^{ 1 } \ge p_{ C }^{ 2 } $. In some approaches discussed in \autoref{section:approaches}, there is no choice of actions,
meaning that the cardinality of $ \setA_{ ( \vecx , \veci ) } $ is one, which can be obtained by setting $ 1 \equiv 2 \equiv 3 $, effectively reducing the
Markov decision process to a Markov reward process.

In some approaches, the action set depends on the observations only via their sum $ t = s_{ C } + f_{ C } + s_{ D } + f_{ D } $, thus can be written as $
\setA ( t ) $. Finally, the simplest case is the one in which the action set is constant, which we write as $ \setA $.

\paragraph{Transition Probabilities.}
Denote by $ q_{ k , ( \vecx , \veci ) , o } $ the probability of observing response $ o \in \setO $ for the current subject if it is allocated to arm $ k
\in \setK $ in state $ ( \vecx , \veci ) $. We assume that $ \sum_{ o \in \setO } q_{ k , ( \vecx , \veci ) , o } = 1 $ for all $ k $, but this can be
relaxed in some models, e.g. if allowing for dropouts (i.e., missing responses).

If the information state $ \veci $ does not change during the trial, then the transition probabilities of moving from state $ ( \vecx , \veci ) $ to state
$ ( \vecx^{ \prime } , \veci ) $ under action $ a $ are
\begin{align*}
h_{ ( \vecx , \veci ) , ( \vecx^{ \prime } , \veci ) }^{ a } =
\begin{cases}
p_{ C }^{ a } q_{ C , ( \vecx , \veci ) , 1 } & \text{ if } \vecx^{ \prime } = \vecx + \vece_{ 1 } \\
p_{ C }^{ a } q_{ C , ( \vecx , \veci ) , 0 } & \text{ if } \vecx^{ \prime } = \vecx + \vece_{ 2 } \\
p_{ D }^{ a } q_{ D , ( \vecx , \veci ) , 1 } & \text{ if } \vecx^{ \prime } = \vecx + \vece_{ 3 } \\
p_{ D }^{ a } q_{ D , ( \vecx , \veci ) , 0 } & \text{ if } \vecx^{ \prime } = \vecx + \vece_{ 4 } \\
\end{cases}
\end{align*}
where $ \vece_{ j } $ is the standard basis vector.
If the information state changes during the trial, then these transition probabilities need to be amended to reflect its dynamics.

\paragraph{Expected One-Period Rewards.}
The expected one-period reward $ r_{ ( \vecx , \veci ) }^{ a } $ for all states $ ( \vecx , \veci ) $ and all actions $ a $ needs to be defined. If the
information state $ \veci $ does not change during the trial, then it is as follows: $ r_{ ( \vecx , \veci ) }^{ a } = 0 $ for all states such that $ s_{ C
} + f_{ C } + s_{ D } + f_{ D } \le T - 1 $ and $ r_{ ( \vecx , \veci ) }^{ a } = s_{ C } + s_{ D } $ for all states such that $ s_{ C } + f_{ C } + s_{ D
} + f_{ D } = T $ (i.e., the reward is the number of observed successes in all states in which the trial can eventually end).

The above definition of the reward is novel. The conventional one is to set the reward to the expected value of observing one success in a given state
under a given action, i.e., $ r_{ ( \vecx , \veci ) }^{ a } = p_{ C }^{ a } q_{ C , ( \vecx , \veci ) , 1 } + p_{ D }^{ a } q_{ D , ( \vecx , \veci ) , 1 }
$ at all time epochs $ t = 0 , 1 , \dots , T - 1 $ and $ r_{ ( \vecx , \veci ) }^{ a } = 0 $ at time epoch $ t = T $. In
\autoref{section:reward_equivalence} we prove the following theorem.

\begin{theorem}\label{theorem:1}
The two reward definitions give the same expected total reward for any fixed design.
\end{theorem}

\paragraph{State and Action Processes.}
The evolution of a Markov decision process is captured by the state process, which in full generality is 2-dimensional in order to keep the physical and
information states separately, $ ( \vecX ( \cdot ) , \vecI ( \cdot ) ) $, and the action process which depends on the state process, but can be briefly
written as $ A ( \cdot ) $, where $ A_{ ( \vecX ( t ) , \vecI ( t ) ) } \in \setA_{ ( \vecX ( t ) , \vecI ( t ) ) } $.

\section{Assumptions, Settings and Objectives}
\label{section:settings}

\subsection{Usage}

There are three principal types of usage of the model described above.

\paragraph{Evaluation by Simulation.}

Computer simulation is now a commonly used evaluation tool as it is relatively straightforward and the accuracy vs runtime trade-off can be addressed by
adjusting the number of simulation runs. But we believe that it has the \emph{law-of-the-hammer} syndrome of all simple universal tools: ``if the only tool
you have is a hammer, to treat everything as if it were a nail.''

\paragraph{Evaluation by Backward Recursion.}

In this paper, we give evidence that it is possible and preferable to use backward recursion instead of simulation for evaluation. This yields a perfectly
accurate evaluation (subject to computational accuracy of the chosen numerical type). We discuss its runtime in \autoref{section:myths}.

\paragraph{Optimization.}

When the action set is not singular in all states, there is room for choosing one of the actions for every state according to an objective of maximizing
some function. This does not necessarily need to be done by backward recursion; we describe several approaches in \autoref{section:approaches}. In this
case we assume that the success probabilities are unknown as otherwise it is trivial to optimize.

\subsection{Knowledge Assumptions}

While the success probabilities are assumed unknown for optimization, they may not necessarily be so for evaluation. There are two principal ways of
specifying the probabilities $ q_{ k , ( \vecx , \veci ) , o } $ of observing response $ o $, which depend on the knowledge assumption about the success
probabilities.

\paragraph{Known Success Probabilities.}
Evaluation of all the approaches described in \autoref{section:approaches} can be done by assuming that success probabilities $ \theta_{ k } $ are known,
and part of $ \vecI ( t ) $ for all $ t $. In that case the transition probabilities are independent of the physical state, so we can write $ q_{ k , \veci
, o } $. If the information state is $ \veci = ( \theta_{ C } , \theta_{ D } ) $ during the whole trial, then
\begin{align*}
q_{ k , \veci , o } =
\begin{cases}
\theta_{ k } & \text{ if } o = 1 \\
1 - \theta_{ k } & \text{ if } o = 0 \\
\end{cases}
\end{align*}

\paragraph{Unknown Success Probabilities.}
Approaches that allow for optimization assume that the success probabilities are unknown (otherwise the decision with the objective of maximising the
expected number of successes is trivial), and so require estimates of these, which can be obtained using Bayesian updating. Following the existing
literature, we use the Bayesian Beta-Bernoulli model for each arm $ k \in \setK $, in which $ \theta_{ k } $ is assumed to be a random variable drawn from
Beta distribution dependent on the state. At the initial time epoch $ t = 0 $, i.e., in physical state $ ( 0 , 0 , 0 , 0 ) $, each arm $ k $ is given a
prior Beta distribution with parameters $ \widetilde{ \successes }_{ k } ( 0 ) , \widetilde{ \failures }_{ k } ( 0 ) $. These parameters can be interpreted
as the numbers of pseudo-observations of successes and failures before the start of the trial. They are thus part of the information state $ \vecI ( t ) $
for all $ t $ and do not change over time. At every time epoch $ t $, in physical state $ \vecx = ( \successes_{ C } , \failures_{ C } , \successes_{ D } ,
\failures_{ D } ) $, each arm $ k $ has the posterior distribution given, because of conjugacy, by the Beta distribution with parameters $ \widetilde{
\successes }_{ k } , \widetilde{ \failures }_{ k } $, briefly $ \text{Beta} ( \widetilde{ \successes }_{ k } , \widetilde{ \failures }_{ k } ) $, where $
\widetilde{ \successes }_{ k } = \widetilde{ \successes }_{ k } ( 0 ) + \successes_{ k } $ and $ \widetilde{ \failures }_{ k } = \widetilde{ \failures }_{
k } ( 0 ) + \failures_{ k } $.

If the information state is $ \veci = ( \widetilde{ \successes }_{ C } ( 0 ) , \widetilde{ \failures }_{ C } ( 0 ) , \widetilde{ \successes }_{ D } ( 0 ) ,
\widetilde{ \failures }_{ D } ( 0 ) ) $ during the whole trial, then
\begin{align*}
q_{ k , ( \vecx , \veci ) , o } =
\begin{cases}
\frac{ \widetilde{ \successes }_{ k } }{ \widetilde{ \successes }_{ k } + \widetilde{ \failures }_{ k } } & \text{ if } o = 1 \\
\frac{ \widetilde{ \failures }_{ k } }{ \widetilde{ \successes }_{ k } + \widetilde{ \failures }_{ k } } & \text{ if } o = 0 \\
\end{cases}
\end{align*}

The conventional assumption is to take the uniform distribution as a prior distribution on each arm $ k $, i.e., $ ( \widetilde{ \successes }_{ k } ( 0 ) ,
\widetilde{ \failures }_{ k } ( 0 ) ) = ( 1 , 1 ) $; for a discussion on choosing a different prior Beta distribution, see \autoref{section:priors}.

\subsection{Performance Measures and Objectives}

In this paper we focus on the \emph{number of successes} as the principal performance measure (a.k.a. operating characteristic in clinical trials
literature). But instead of the number of successes, we report two equivalent measures: the \emph{proportion of successes} and the \emph{regret number of
successes}, as these provide complementary interpretation and insights. Yet another equivalent measure (not reported in this paper) is the \emph{fraction
of subjects allocated to the better arm}. Many other additive measures, e.g. monetary cost typical in health economics and health technology assessment,
can be defined analogously, but are not discussed in this paper.

A particular design $ \pi $ prescribes the action process $ A ( \cdot ) $. Let $ \Pi $ be the set of designs that are non-anticipating\footnote{A
non-anticipating design is a design which cannot see into the future; i.e., an action prescribed by the design at a given time epoch does not require the
knowledge of states which have not been observed yet.} and satisfy the above constraints on $ A ( \cdot ) $.

Let us denote by $ \Expectation^{ \pi }_{ t } [ \cdot ] $ the expectation under design $ \pi \in \Pi $ conditioned on information available at time epoch $
t \in \setT $. The mean number of successes is
\begin{align}
\totalNumber^{ \pi }_{ t } \left( \vecx , \veci \right) := \Expectation^{ \pi }_{ t } &\left[ \left. \sum_{ u = t }^{ T } r_{ ( \vecX ( u ) , \vecI ( u ) ) }^{ A_{ ( \vecX ( u ) , \vecI ( u ) ) } } \right| \left( \vecX ( t ) , \vecI ( t ) \right) = \left( \vecx , \veci \right) \right] .
\end{align}
and the mean proportion of successes is
\begin{align}
\totalProportion^{ \pi }_{ t } \left( \vecx , \veci \right) := \frac{ 1 }{ T - t } \totalNumber^{ \pi }_{ t } \left( \vecx , \veci \right) .
\end{align}
These two are measures of \emph{subject benefit}, while the former is on the absolute scale, the latter yields the average per-subject probability of
observed success, i.e., the subject benefit on the percentage scale.
We further define the mean regret number of successes,
\begin{align}
\totalRegret^{ \pi }_{ t } \left( \vecx , \veci \right) := \Expectation \left[ \left. \max\{ \theta_{ C } , \theta_{ D } \} \right| \vecI ( t ) = \veci \right] - \totalNumber^{ \pi }_{ t } \left( \vecx , \veci \right) ,
\end{align}
which is a measure of \emph{subject loss}. Note that all the three measures depend on parameter $ T $, although we have suppressed the explicit notation.

The objective is to find an optimal design $ \pi^{ * } $ that maximises the mean number of successes as evaluated at time epoch $ t = 0 $ when there are no
observations ($ \vecx = \veczero $), i.e.,
\begin{align}
\pi^{ * } := \argmax_{ \pi \in \Pi } \totalNumber^{ \pi }_{ 0 } \left( \veczero , \veci \right)
\end{align}
or, equivalently, maximizes the mean proportion of successes, or, equivalently, minimizes the mean regret number of successes.

Following the two knowledge assumptions above, we have two approaches to performance evaluation.

\paragraph{Known Success Probabilities.}
When the success probabilities are assumed to be known, we call the above measures the \emph{frequentist number of successes}, the \emph{frequentist
proportion of successes}, and the \emph{frequentist regret number of successes}, respectively. Due to symmetry, we can assume that $ \theta_{ C } \le
\theta_{ D } $ without loss of generality, and thus
\begin{align}
\totalRegret^{ \pi }_{ t } \left( \vecx , \veci \right) = \theta_{ D } - \totalNumber^{ \pi }_{ t } \left( \vecx , \veci \right) .
\end{align}

\paragraph{Unknown Success Probabilities.}
When the success probabilities are assumed to be unknown, we evaluate performance in terms of the quantities known in the literature as the Bayes return,
Bayes worth, or Bayes risk. In particular, we call the above measures the \emph{Bayes number of successes}, the \emph{Bayes proportion of successes}, and
the \emph{Bayes regret number of successes}, respectively.

For a problem with uniform distribution as a prior distribution on each arm, as considered in this paper, $ \Expectation \left[ \left. \max\{ \theta_{ C }
, \theta_{ D } \} \right| \vecI ( t ) = \veci \right] = 2 / 3 $ \citep{Berry1978jasa}, thus
\begin{align}
\totalRegret^{ \pi }_{ t } \left( \vecx , \veci \right) = \frac{ 2 }{ 3 } - \totalNumber^{ \pi }_{ t } \left( \vecx , \veci \right) ,
\end{align}

Note that our definition is different from the so-called Bayesian regret \citep[cf.][Section 34.6]{LattimoreSzepesvari2019book}, which is the average of
the frequentist regret with respect to the prior distribution, i.e. it is Bayesian only in the initial time epoch.

%
%
%
%
%
%

\section{Approaches}
\label{section:approaches}

In this section we describe a number of common approaches to the problem described in the previous section. We start with Equal Randomized Allocation,
which was originally developed for static setting, in which all the subjects arrive at the same time. All the remaining approaches are suitable only in
dynamic setting, i.e., are data-driven.

\subsection{Design of Experiments}

Statistical testing based on \emph{equal (i.e., \design{1:1}) randomized allocation} (a.k.a. random assignment) is a widespread state-of-the-art approach
in the design of experiments today, known as \emph{randomized controlled trial} in medicine, \emph{between-group design} in social sciences, and \emph{A/B
testing} in digital marketing. I theory, it allows the greatest reliability and validity of statistical estimate of the intervention effect (i.e., the
difference between the two success probabilities) under an initial equipoise assumption. Originally developed, advocated and popularized by the founders of
statistics such as C. S. Peirce in the fields of psychology and education in the late 19th century, and J. Neyman and R. A. Fisher in agriculture and other
fields in the early 20th century. In the middle of the 20th century, A. B. Hill popularized the method in the field of medicine and it is now the preferred
approach by regulatory agencies for assessing efficacy when deciding about marketing authorisation of new medicinal interventions in many countries. The
approach can be evaluated using routine (frequentist) statistical methods.

In this approach the transition probabilities are frequentist, and
\begin{itemize}

\item the information state $ \veci $ is ignored;

\item the action set is constant, $ | \setA | = 1 $ and $ p_{ C }^{ a } = p_{ D }^{ a } = 1 / 2 $;

\end{itemize}

If the success probabilities are known, then its performance can be evaluated directly: the proportion of successes with mean $ ( \theta_{ C } +
\theta_{ D } ) / 2 $ and standard deviation $ \sqrt{ ( \theta_{ C } + \theta_{ D } ) ( 2 - \theta_{ C } - \theta_{ D } ) / 4 T }$; the regret
number of successes with mean $ T \max \{ ( \theta_{ C } - \theta_{ D } ) / 2 , ( \theta_{ D } - \theta_{ C } ) / 2 \} $ and standard deviation
$ \sqrt{ T ( \theta_{ C } + \theta_{ D } ) ( 2 - \theta_{ C } - \theta_{ D } ) / 4 }$.

The approach is believed to be well understood. However, there are several potential concerns, e.g., (i) it was developed under the assumption of an
infinite population of subjects to which the best arm will be applied, but it is not clear how valid it is when the population is finite or when new,
potentially better arms become available in future; (ii) it was developed under the assumption of parallel allocation of subjects, but in dynamic (i.e.
sequential) setting, in which subjects arrive one by one, there is a risk of introducing accrual and/or allocation bias in the estimation of the success
probabilities if the person (or machine) delivering an intervention is aware of the characteristics and/or responses of the previous subjects; (iii) it was
developed under the assumption of the simple random sampling method, which is not true if the subjects need to provide a consent (which is a legal
requirement, for instance, in clinical trials); (iv) it was developed under the assumption of the use of random numbers, which not true for computer
generated randomization (pseudo-random numbers) provided by error-prone software and procedures of external companies.

\subsection{Bayesian Decision Theory}

Bayesian Decision Theory emerged together with game theory and mathematical programming in the middle of the 20th century, building on models and
techniques from Bayesian statistics, applied mathematics, and economics. A number of prominent researchers contributed to the development of the field,
including A. Wald, J. Wolfowitz, L. J. Savage, K. J. Arrow, D. Blackwell, H. Raiffa, R. Schlaifer, R. Bellman, etc.

The two-armed problem with binary responses was first formulated in this framework in \citet{Bellman1956}, with a known success probability of one arm,
over an infinite horizon $ T = + \infty $, and with geometric discounting of future rewards. He proposed to employ backward recursion based on the Bellman
equation from dynamic programming he had developed. \citet{Gluss1962iac} extended the model to two unknown arms using the same approach and discussed
the theoretical memory requirement relevant for online optimization assuming a finite-horizon truncation, after which the better arm is used forever.
\citet{Steck1964moc} programmed the backward recursion on one of the scientific computers of those times (Univac 1105) and listed the optimal allocations
with the truncation at $ T = 25 $ (and the discount factor $ 0.95 $).

\begin{figure}[tbp]
\centering
    \begin{subfigure}[b]{0.48\textwidth}
        \includegraphics[trim=0pt 0pt 0pt 0pt, clip=true, width=\textwidth]{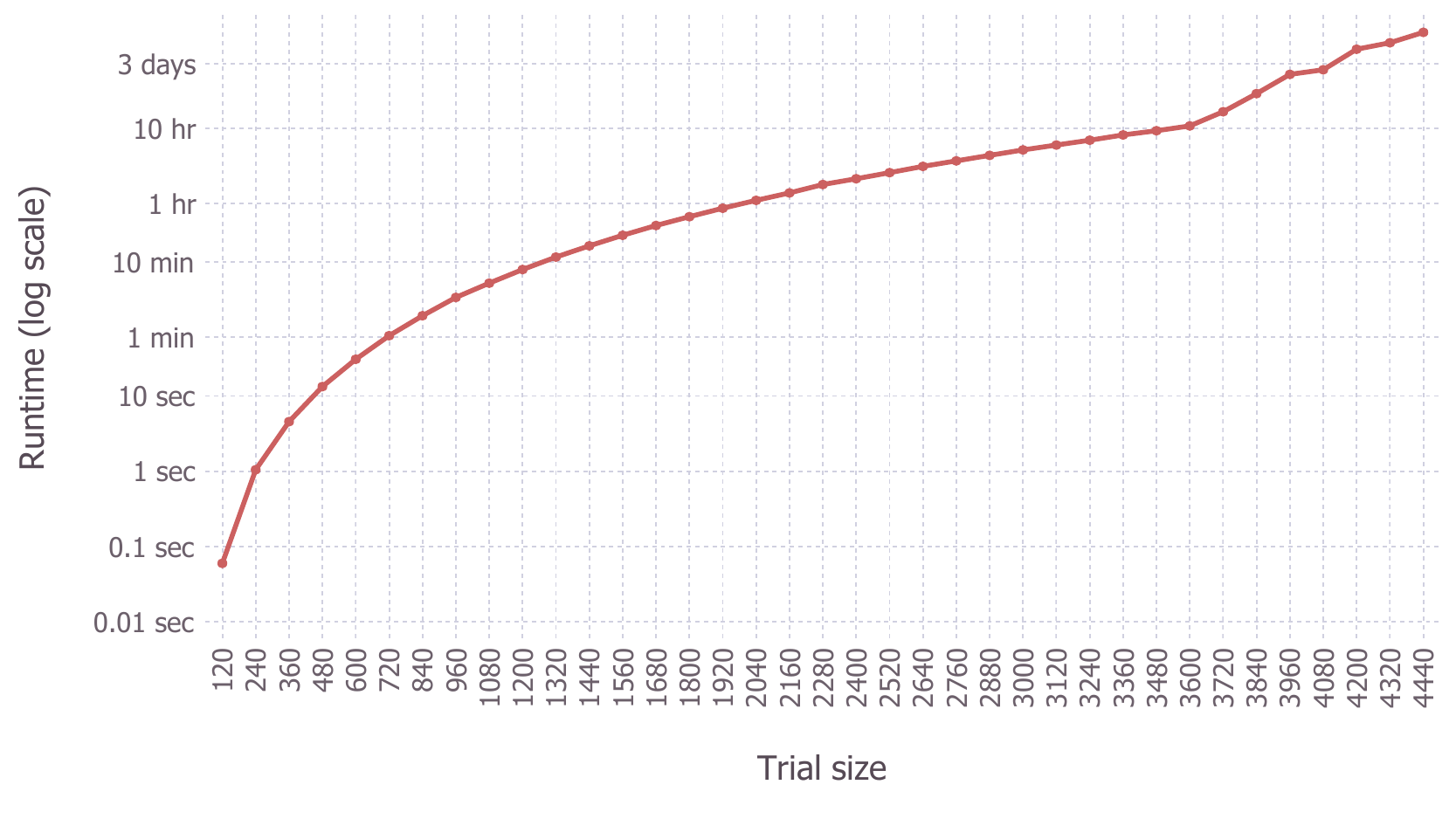}
    \end{subfigure}\hfill%
    \begin{subfigure}[b]{0.48\textwidth}
        \includegraphics[trim=0pt 0pt 0pt 0pt, clip=true, width=\textwidth]{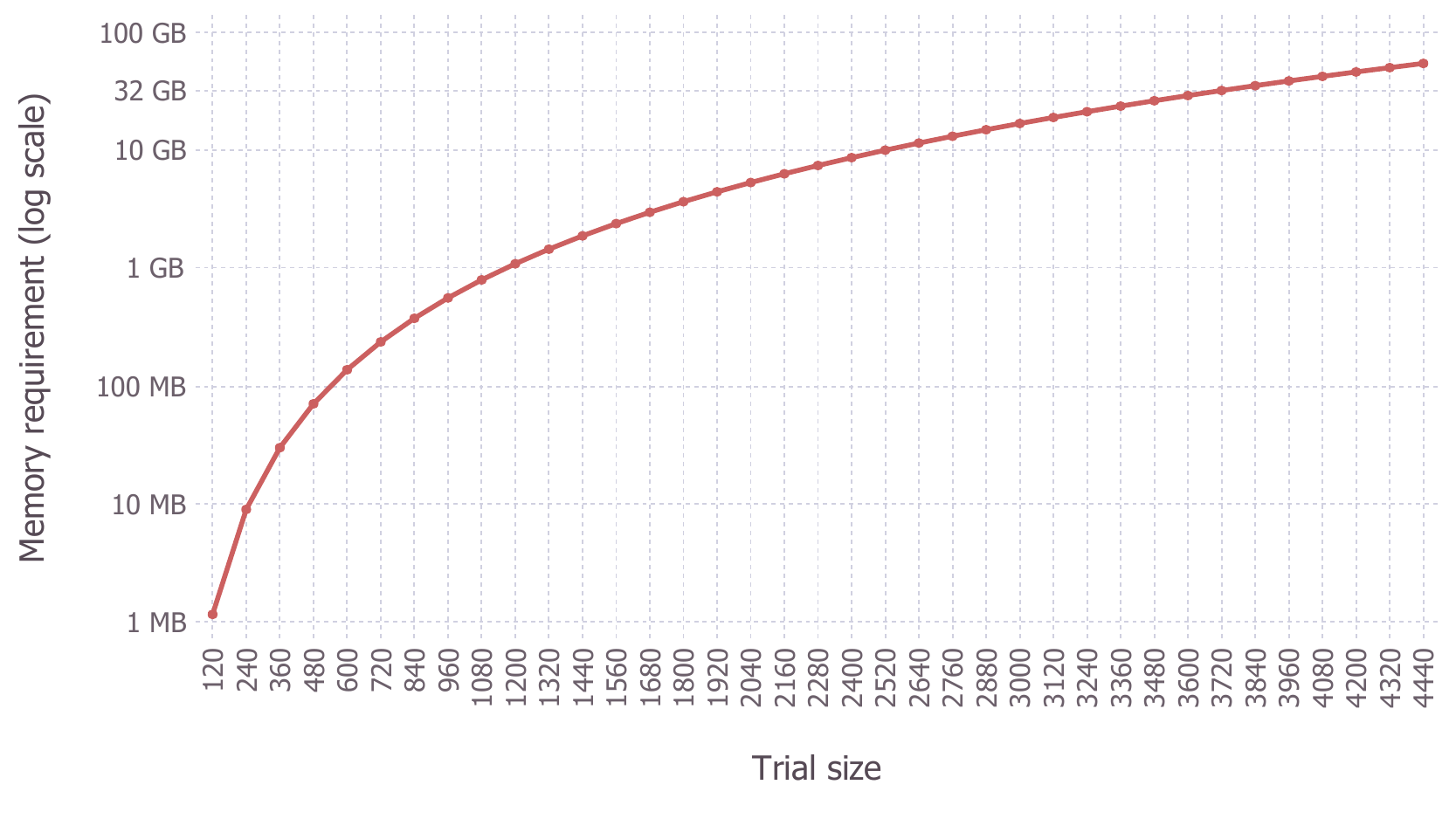}
    \end{subfigure}%
\caption{An illustration of computational complexity of online calculation of the deterministic Bayesian decision-theoretic design over a range of trial sizes.}\label{fig:1}
\end{figure}

For finding the best action for each state via this approach, the transition probabilities are Bayesian, and
\begin{itemize}

\item the conventional assumption is to take as information state $ \veci $ (that does not change over time) the prior distribution on each arm $ k $;

\item the action set is constant, $ | \setA | = 3 $;

\end{itemize}

Since \autoref{theorem:1} is true for any design, it is also true for the optimal design, which is derived using the Bellman equation, which is an
optimization equivalent of the Poisson equation.

\begin{theorem}
The two reward definitions give the same optimal design and to the same optimal expected total reward.
\end{theorem}


The conventional assumption is that the two pure actions are deterministic, i.e., $ p_{ C }^{ 1 } = p_{ D }^{ 2 } = 1 , p_{ C }^{ 2 } = p_{ D }^{ 1 } = 0
$, which we will refer to as the \emph{deterministic Bayesian decision-theoretic design}, despite the fact that randomization is allowed when both actions
are optimal.

A more general setting, in which the pure actions are allowed to be randomized instead of deterministic, was proposed in \citet{ChengBerry2007biometrika}
and further developed in \citet{WilliamsonEtal2017csda}. In fact, such a setting provides a continuum of designs, ranging from the deterministic Bayesian
decision-theoretic design to the equal randomized allocation design, recovered by setting $ p_{ C }^{ 1 } = p_{ C }^{ 2 } = p_{ D }^{ 1 } = p_{ D }^{ 2 } =
1 / 2 $.

\subsubsection{Optimal --- Dynamic Programming}

The Bayesian decision-theoretic model can be solved by (stochastic) dynamic programming (\design{DP}), which comprises of a calculation starting by
enumerating all the possible states in the final time period, continuing backwards in time while employing the Bellman equation in every state.
Unfortunately, complete structure of the optimal design is unknown, thus numerical computation is the only way of obtaining it.

\autoref{fig:1} illustrates the computational complexity of \emph{online calculation} (i.e., outputting the optimal action of the initial state) of this
design, using a state-of-the-art package \emph{BinaryBandit} in Julia programming language. In the number of elements (i.e., without multiplying by the
memory required for each element), the memory requirement presented in \autoref{fig:1}(right) is $ 2 $ times larger than that of \citet{Gluss1962iac}
because of his elimination of symmetric states, which can be done if the prior Beta distributions are the same for the two arms.

The computational complexity of \emph{offline calculation} (i.e., outputting the optimal actions of all possible states) of this design is very similar to
online calculation in terms of the runtime, but radically different in terms of memory requirement. Using the \emph{BinaryBandit} Julia package, a computer
with 32GB RAM is able to solve the problem of up to trial size $ T = 1,440 $, keeping all the optimal actions in RAM during the calculations. For practical
purposes, however, it may be possible to store parts of the solution on hard disk and thus relieve RAM memory to allow calculation for larger trial sizes.

A number of author's colleagues and PhD students who had programmed this design for their needshave kindly provided runtimes of their code implementations,
which are presented in \autoref{table:code_runtime} for comparison, indicating that the BinaryBandit package is two orders of magnitude faster than ad hoc
codes and is able to solve a few times larger problems.

\begin{table}[tbp]
\centering
\begin{tabular}{lrrrrrrr}
\toprule
Software & RAM & $ T = 60 $ & $ T = 120 $ & $ T = 180 $ & $ T = 240 $ & $ T = 300 $ & $ T^{ \max } $ \\
\midrule
Julia 0.6.2 \& ad hoc & 12 GB & 2sec & 22sec & 108sec & 331sec & 789sec & 420 \\
Julia 1.0.1 \& ad hoc & 12 GB & 1sec & 17sec & 82sec & 262sec & 643sec & 420 \\
R \& ad hoc & 16 GB & 1sec & 12sec & 59sec & 191sec & N/A & 240 \\
Julia 1.0.1 \& BB & 31 GB & 0.0036sec & 0.046sec & 0.23sec & 0.73sec & 1.6sec & 1440 \\
\midrule
R \& ad hoc & 5 GB & 1sec & 6sec & 26sec & 84sec & 209sec & 420 \\
Julia 1.0.1 \& BB & 31 GB & 0.0040sec & 0.056sec & 0.27sec & 0.91sec & 2.8sec & 4440 \\
\bottomrule
\end{tabular}
\caption{A comparison of runtime of the Bayesian decision-theoretic design for $ T = 60 : 60 : 300 $ and the largest horizon $ T^{ \max } $ (as a multiple of $ 60 $) which does not give an out-of-memory error using code implementations by author's colleagues and students. The top four are for offline calculation, the bottom two are for online calculation. BB refers to the use of the BinaryBandit package.}\label{table:code_runtime}
\end{table}

\subsubsection{Asymptotically Optimal --- Gittins and Whittle Index Rules}
\label{section:Gittins}

The structure of the optimal Bayesian decision-theoretic designs with deterministic pure actions is however known when considering the rewards over an
infinite horizon $ T \to + \infty $ and discounted with a geometrically-distributed discount factor $ 0 < \gamma < 1 $. \citet{GittinsJones1974} discovered
that optimal allocations can be characterized by an \emph{index rule}, which allows (Gittins) index values to be calculated for every arm separately, and
which at every time epoch allocates a subject to the arm with the highest Gittins index value (breaking the ties arbitrarily). See, e.g., \citet{Gittins1979,
Gittins1989, GittinsEtal2011} for general theory on the Gittins index.

Although theoretically appealing and useful in many other problems, in the setting of the Bayesian decision-theoretic design it does not provide an
ultimate solution, because the states are time-dependent and the horizon $ T $ is finite. Still, the Gittins index rule can be used as an approximation to
the optimal design, being asymptotically optimal as $ T \to +\infty $. It decreases the computational complexity of the offline calculation notably because
of a decreased size of the problem (one arm, i.e. two dimensions), but is still computed by dynamic programming, which is computationally demanding. In its
calculation, however, it is necessary either to use a discount factor and/or a horizon truncation, or to compute it for every horizon $ T $ considered by
adding the remaining number of subject allocations as a third dimension. In the former case, the resulting index rule is only an approximation and tends to
focus on learning less than the optimal design, see e.g. \citet{VillarEtal2015survey}. In the latter case, \citet[Section 6.2]{Nino2011} provides a
comparison of an approximate, so-called calibration, algorithm (using a grid of values to desired accuracy) and an exact algorithm for the calculation of
the Whittle index values, showing that already $ T \le 100 $ requires several GB of RAM, which suggests that the calibration algorithm with a grid of not
more than $ 4 $ significant digits is the only practical method for larger horizons.


Besides the requirement of infinite horizon, the theory of the Gittins index only applies when the pure actions are deterministic. When these are
randomized, the problem becomes so-called \emph{restless}, meaning that more than one arm can change its state in every period. Also, even when the pure
actions are deterministic, but the horizon is finite, the problem can be seen as restless, by adding the remaining number of subject allocations to the
state of each arm. \citet{Whittle1988} proposed to solve the restless problem also by an index rule, acknowledging that such a rule would not necessarily
be optimal, but conjecturing that it would admit a form of asymptotic optimality as both the number of arms and the number of allocated arms in each period
grow to infinity at a fixed proportion, which was eventually proved in \citet{WeberWeiss1990} under some technical assumptions. Whittle defined an index,
which reduces to the Gittins index in the non-restless setting, which became known as the \emph{Whittle index}. The above discussion suggests that the
Whittle index rule is conceptually more appropriate and more accurate than the Gittins index rule, and this was confirmed for the Bayesian
decision-theoretic design numerically, see, e.g., \citet{VillarEtal2015survey, Villar2018peis}. The computational complexity of the approximate Whittle
index values calculated by the calibration method for every horizon $ T $ considered (using a grid of values to desired accuracy) is similar to that of the
approximate Gittins index values using the same method.

Note that both the Gittins and Whittle index rule require to keep $ 3 $ elements in the one-arm state, so the reduction from the optimal two-armed problem
is only by one dimension. Moreover, in offline calculation, the index values that need to be stored are non-integer, thus require $ 64 $ bits per one-arm
state, while the offline calculation of the solution to the two-armed problem stores directly optimal actions, which require $ 2 $ bits per two-arm state.
Thus, the index rules (calculated to a few significant digits) are typically preferable to dynamic programming in calculation for horizons around $ T \ge
1000 $ and only if planned to be used for evaluation by simulation or for implementation in practice. However, \citet{Kaufmann2018aos} reports that she was
able to compute the Gittins index values only up to $ T = 1000 $.

Of course, advantages of index rules become important in problems with more than two arms, in which dynamic programming suffers from the \emph{curse of
dimensionality}. A discussion of such problems is however beyond the scope of this paper.

\subsubsection{Approximately Optimal --- Approximate Dynamic Programming}

Several general approaches have been proposed to deal with the curse of dimensionality of stochastic optimization problems, which are collectively known as
the approximate dynamic programming. There is a number of approximation techniques, but broadly focus on problem size reduction (e.g., the state space is
approximated by a grid for which optimal actions are computed, and interpolated on non-included states) and/or on simplification of function computations
(e.g. the value function is approximated by looking at decisions only a few periods ahead). See, e.g., \citet[Section 6.6]{PowellRyzhov2018book},
\citet{AhujaBirge2019report}.

We will describe one approach, which leads to a Bayesian design known as the \emph{knowledge gradient} (\design{BKG}). The fundamental idea to reduce the
amount of information and computation required for a decision in a given state at a given time epoch is to assume that this decision is the last one, and
in the next period we will identify the better arm which will whence be allocated to all the remaining subjects. \citet[Section 7.2]{FrazierEtal2008sjco}
showed that it is optimal for a search variant (i.e., maximizing the final-period expected reward) of the two-armed bandit problem with continuous
responses. A specific variant of this approach for our setting was presented in \citet[Section 4.7.1]{PowellRyzhov2012book} and further studied and
improved in \citet{EdwardsEtal2017peis}.

Let us denote by $ \mu_{ k } ( \widetilde{ \successes }_{ k } ( t ) , \widetilde{ \failures }_{ k } ( t ) ) $, or briefly $ \widetilde{ \mu }_{ k } $, the
\emph{belief}, i.e., the mean of the posterior Beta distribution of arm $ k $ calculated using posterior observations (i.e., both the observations and
prior pseudo-observations), i.e.,
\begin{align}
\widetilde{ \mu }_{ k } := \frac{ \widetilde{ \successes }_{ k } ( t ) }{ \widetilde{ \successes }_{ k } ( t ) + \widetilde{ \failures }_{ k } ( t ) } .
\end{align}
We will further denote the belief conditional on an additional (not necessarily binary) observation $ \delta $, respectively, by
\begin{align}
\widetilde{ \mu }_{ k }^{ +\delta } := \frac{ \widetilde{ \successes }_{ k } ( t ) + \delta }{ \widetilde{ \successes }_{ k } ( t ) + \widetilde{ \failures }_{ k } ( t ) + 1 } .
\end{align}

Under independent prior Beta distributions, the allocation at epoch $ t $ by this design is to the arm $ k $ with currently the largest value of
the following score
\begin{align}
\begin{cases}
\widetilde{ \mu }_{ k } + ( T - t - 1 ) \widetilde{ \mu }_{ \ell } & \text{ if } \widetilde{ \mu }_{ \ell } \ge \widetilde{ \mu }_{ k }^{ +1 } \\
\widetilde{ \mu }_{ k } + ( T - t - 1 ) \left[ ( 1 - \widetilde{ \mu }_{ k } ) \widetilde{ \mu }_{ \ell } + \widetilde{ \mu }_{ k } \widetilde{ \mu }_{ k }^{ +1 } \right] & \text{ if } \widetilde{ \mu }_{ k }^{ +1 } \ge \widetilde{ \mu }_{ \ell } \ge \widetilde{ \mu }_{ k }^{ +0 } \\
\widetilde{ \mu }_{ k } + ( T - t - 1 ) \widetilde{ \mu }_{ k }^{ +\widetilde{ \mu }_{ k } } & \text{ if } \widetilde{ \mu }_{ k }^{ +0 } \ge \widetilde{ \mu }_{ \ell }
\end{cases}
\end{align}
where $ \ell $ refers to the other arm, and ties are broken randomly. Although this score can be interpreted as capturing the allocation priority similarly
to the Gittins and Whittle indices, it depends on the state of the other arm, thus, strictly speaking, it is not an index.

\subsection{Na\"{\i}ve Designs}

While in the previous subsection we described designs which are horizon-dependent and forward-looking in the way of making allocation decisions, we now
turn our attention to horizon-independent designs and present a number of different na\"{\i}ve approaches in this subsection. Such designs are
computationally simple and easy to interpret, hence sometimes more appealing for practical use. It turns out that such deterministic designs are either
related to \emph{myopia} ($ T = 1 $) or to \emph{utopia} ($ T = +\infty $).

We will say that a design leads to \emph{complete learning}, if in the problem considered over an infinite horizon (i.e., $ T = +\infty $) it identifies
the better arm almost surely. It is easy to see that to achieve complete learning it is necessary to allocate each arm infinitely often.

\subsubsection{Myopic Index Rules}

\citet{BradtEtal1956} considered the setting in which one knows the two-point set $ \{ \theta_{ C } , \theta_{ D } \} $, but does not know which arm is
which, and proved that if $ \theta_{ C } + \theta_{ D } = 1 $, then it is optimal, under any prior distributions, to deterministically allocate every
subject to the arm with currently largest Bayesian expected one-period reward. The also proved the same in the case $ \theta_{ C } + \theta_{ D } = 1 $
when the set $ \{ \theta_{ C } , \theta_{ D } \} $ is unknown. For the case $ \theta_{ C } + \theta_{ D } \neq 1 $ with known set $ \{ \theta_{ C } ,
\theta_{ D } \} $ they proved that such a design leads to complete learning and conjectured that it is optimal in terms of the Bayesian expected number of
successes. They also discussed that such a design is not optimal if both of these assumptions are dropped.

Note that under independent Beta prior distributions, the allocation at epoch $ t $ by this design is to the arm $ k $ with currently the largest mean
calculated using posterior observations (i.e., both the observations and prior pseudo-observations),
i.e.,
\begin{align}
\frac{ \widetilde{ \successes }_{ k } ( t ) }{ \widetilde{ \successes }_{ k } ( t ) + \widetilde{ \failures }_{ k } ( t ) }
\end{align}
where ties are broken randomly. We call this the \emph{Bayesian myopic} (\design{BM}) design, because it takes an action which is best for the next subject
only and ignores the future.

\citet{Feldman1962ams} also considered the setting with known two-point set $ \{ \theta_{ C } , \theta_{ D } \} $ and further generalized the optimality of
the Bayesian myopic design. \citet[Section 8]{Berry1972} realized that the setting with known two-point set $ \{ \theta_{ C } , \theta_{ D } \} $ (i.e.,
with strong between-arm dependence) is equivalent to the two-armed problem with independent arms with two-point prior distributions. \citet{Kelley1974as}
further generalized the setting with dependent arms and identified conditions for optimality of the Bayesian myopic design.

We also consider the \emph{frequentist myopic} (\design{FM}) design, which allocates each arm once in the first two periods, and then deterministically
allocates every subject to the arm with currently the largest mean calculated using the observations only (breaking the ties randomly), i.e.,
\begin{align}
\frac{ \successes_{ k } ( t ) }{ \successes_{ k } ( t ) + \failures_{ k } ( t ) }
\end{align}
Such a design was mentioned in \citet[Eq. (1.1)]{Bather1981jrssb} and called ``play the favourite''. \footnote{In the machine learning literature since the
2000s, this design is also known as ``follow the leader'' (as a simpler variant of ``follow the perturbed leader''). The term has however been used also
for other designs: \citet{SobelWeiss1972} are the first to use the term ``follow the leader'', but they mean a variant of ``stay-with-a-winner \&
switch-on-a-loser'', with a difference of randomizing (rather than switching) when a failure is observed and the number of failures on both arms is equal.
Many papers studying continuous bandit problems use the term ``follow the leader'' for (Gittins) index rule.}
Note that \design{BM} and \design{FM} are equivalent if and only if $ \widetilde{ \successes }_{ C } ( 0 ) = \widetilde{ \failures }_{ C } ( 0 ) =
\widetilde{ \successes }_{ D } ( 0 ) = \widetilde{ \failures }_{ D } ( 0 ) = 0 $.

\citet[Section 3]{Berry1978jasa} established that, in the setting with known two-point set $ \{ \theta_{ C } , \theta_{ D } \} $, in which the \design{BM}
design is optimal, at any epoch $ t $ at which the posterior probability that arm $ C $ is better than arm $ D $ is $ 1 / 2 $, it is optimal to allocate to
the arm with currently the greatest difference between the number of successes and the number of failures. Several tie-breaking rules can be used but they
only affect the performance marginally, so these are not discussed here in detail. We refer to this as the \emph{Bayesian greatest difference first}
(\design{BGDF}) design. The Bayesian performance of this design was illustrated numerically to be near optimal \citep{Berry1978jasa} and to outperform the
\design{BM} design \citep{Villar2018peis}. The mentioned condition applies when the one-arm prior distributions satisfy $ \widetilde{ s }_{ C } ( 0 ) =
\widetilde{ f }_{ C } ( 0 ) = \widetilde{ s }_{ D } ( 0 ) = \widetilde{ f }_{ D } ( 0 ) $, as in this paper, and note that in that case the frequentist
analogue, \design{FGDF}, yields an equivalent design.

\subsubsection{Utopic Index Rules}
\label{section:utopic}

In the frequentist setting, \citet{Robbins1952} introduced the deterministic ``stay-with-a-winner \& switch-on-a-loser'' design, in which the first
allocation is made randomly, and then the same arm is allocated whenever the observed response is a success, while the other arm is allocated whenever the
observed response is a failure. He realized that it leads to complete learning, but fails to achieve the maximum expected proportion of successes, because
it allocates the worse arm at a regular rate.

Interestingly, many of the deterministic designs actually have the ``stay-with-a-winner'' property. It is easy to see that the \design{BM}, \design{FM},
\design{BGDF} and \design{FGDF} all have it. In the Bayesian setting, \citet[Theorem 6.2]{Berry1972} proved the ``stay-with-a-winner'' property for the
deterministic \design{DP} design. Formally, the ``stay-with-a-winner'' property in this approach is: if arm $ k $ is optimal in period $ t \le T - 2 $, the
subject is allocated to it and its response is a success, then arm $ k $ is uniquely optimal in period $ t + 1 $. Contrary to the above, the
``switch-on-a-loser'' property is not, in general, satisfied by these designs. When one arm looks significantly better than the other one, the designs
would not switch after observing a failure on the former arm. See, e.g., \citet[p. 79]{BerryFristedt1985}.



\citet[Conjecture A, p. 892]{Berry1972} conjectured that for a large number of remaining allocations (i.e. at epochs $ t \ll T $), the only criterion for
Bayesian optimality in the \design{DP} design is the difference between the posterior number of failures on the two arms. (He did not specify what to do if
there is the same number of failures on the two arms.)

\citet{Kelly1981} studied the Gittins index rule in the infinite-horizon setting (in which it is optimal) as the discount factor approaches one (i.e. the
undiscounted setting) and established, under a technical condition on the prior distributions, that the Gittins index rule reduces to the following rule
which we call the \emph{frequenist least failures first} (\design{FLFF}) design: at every time epoch, allocate the subject to the arm with least observed
number of failures, breaking the ties in favor of any arm with greatest observed number of successes (breaking the double ties arbitrarily).

Although the similarity to Berry's conjecture is remarkable, \citet[Remark 4.15]{Kelly1981} stated that there seems to be no immediate link between these
two statements, and that Berry's conjecture may require a technical condition on the prior distributions. \citet[Remark 4.7]{Kelly1981} further elucidated
that \design{FLFF} is a slight variation of the ``stay-with-a-winner \& switch-on-a-loser'' design in \citet{Robbins1952}. Indeed, it is easy to see that
\design{FLFF} has the ``stay-with-a-winner'' property. After observing a failure, it switches if the other arm has lower number of failures, but it stays
or switches accordingly to the greater number of successes if the number of failures on both arms is the same.

Analogously, we call the \emph{Bayesian least failures first} (\design{BLFF}) design the one based on posterior numbers of successes and failures rather
than on the observed ones. If the prior distributions on the two arms are the same, then this design is equivalent to the \design{FLFF} design, but if they
are not, then the arm with lower number of posterior failures will be allocated without switching (at least) until the number of failures on the two arms
becomes the same.

We would like to highlight that the three radically different ``schools'' of the bandit problem, namely Robbins', Berry's and Gittins', all led to the
\design{FLFF} design with only slight variations. The ``stay-with-a-winner \& switch-on-a-loser'' design may be useful in some practical situations because
it only depends on the last observation, not on the numbers of successes and failures over the whole history.

\subsection{Reinforcement Learning --- UCB Index Rules}

\citet{LaiRobbins1985aam} laid out the theory of asymptotically optimal allocation and were the first to actually use the term \emph{``upper confidence
bound''} (UCB). They wrote it in quotes, as the quantity it refers to depends on time period $ t $, and is thus not the conventional upper bound of
confidence intervals, but can be interpreted as the upper confidence bound with significance level $ 1 / t $. \citet[Eq. (2.6)]{Lai1987as} introduced a UCB
index rule using the Kullback-Leiber divergence and \citet[Example 5.7]{Agrawal1995aap} developed UCB index values which are inflations of the observed
mean proportion of successes. The theory was further extended in \citet{BurnetasKatehakis1996} to multivariate and non-parametric distributions. The first
simple UCB index rules with finite-time theoretical guaranties were developed in \citet{AuerEtal2002}. See, e.g.,
\citet{BubeckCesa-Bianchi2012ftml,KaufmannGarivier2017esaim,LattimoreSzepesvari2019book} for accounts of subsequent developments.

The use of the observed mean proportion of successes in defining an index rule is attractive mainly because it is the maximum likelihood estimator of the
success probability in the static setting of the design of experiments. Index rules that use time-dependent inflations of the observed mean proportion of
successes were proposed and investigated even before \citet{LaiRobbins1985aam}, e.g., \citet{Bather1980aap, Bather1981jrssb, AbdelHamid1981thesis} in
frequentist setting and \citet{GittinsJones1974, Glazebrook1980jrssb, GittinsWang1992aos} in Bayesian setting. 

Following \citet[Section 2]{BubeckCesa-Bianchi2012ftml}, we consider the popular \design{$ \alpha $UCB} design which allocates each arm once in the first
two periods, and then deterministically allocates every subject to the arm with currently the largest index (breaking ties randomly) of the form
\begin{align}
\frac{ \successes_{ k } ( t ) }{ \successes_{ k } ( t ) + \failures_{ k } ( t ) } + \sqrt{ \frac{ \alpha \cdot \ln( t + 1 ) }{ \successes_{ k } ( t ) + \failures_{ k } ( t ) } }
\end{align}
where $ \alpha > 0 $. The original design introduced in \citet{AuerEtal2002} used $ \alpha = 2 $. Theoretical upper bounds currently exist for $ \alpha > 1
$, but researchers have noticed empirically that lower values of $ \alpha $ typically lead to better performance and some used $ \alpha = 1 $, see, e.g.,
\citet{CsernaEtal2017uai}. In our numerical experiments (not reported here) we found that approximately the best performance is achieved with $ \alpha =
0.18 $.


Many other types of UCB designs have been developed, see e.g., \citet[Figure 2]{Kaufmann2018aos} for comparison of some of them. Besides designs based on
idea of UCB, there is a number of popular designs with randomized actions, e.g., epsilon-greedy, Boltzmann exploration, Thompson sampling, etc.

\subsection{Biostatistics}

Blinding of patients and personnel to the allocated treatment is an important desideratum in many types of clinical trials to mitigate a variety of biases
such as performance bias, detection bias and attrition bias \citep[Section 8.4]{Cochrane2011}. If allocation is deterministic, the patients and personnel
can, if they know the state of the trial, identify the allocated treatment with certainty or high probability. Thus, some amount of randomness in the
allocation decision is desirable. Moreover, the theory of design of experiments explains that randomization is important for mitigation of the selection
bias, for ensuring similarity in the treatment groups and for providing a basis for inference, which are essential for making valid conclusions at the end
of a trial \citep{RosenbergerEtal2019sim}. Thus, in clinical trials theory and practice, a lot of attention is paid to designs with randomized actions
\citep{RosenbergerLachin2015book}, although it should also be noted that not all types of clinical trials require or allow to implement blinding and
randomization. There are three major approaches to adaptive randomization: (i) \emph{Bayesian Response-Adaptive Randomization}, starting with
\citet{Thompson1933biometrika}, generalized in \citet{ThallWathen2007ejoc} (see also \citet[p. 156]{BerryEtal2011book}), and recently studied also in the
reinforcement learning literature (see, e.g., \citet{AgrawalGoyal2012colt}); (ii) \emph{Frequentist P\'{o}lya Urn Randomization}, which is a randomized
version of the ``stay-with-a-winner \& switch-on-a-loser'' design, starting with \citet{WeiDurham1978jasa}; and (iii) \emph{Randomized Index Rules} such as
those proposed by \citet{Bather1980aap, Bather1981jrssb, AbdelHamid1981thesis} in frequentist setting and by \citet{Glazebrook1980jrssb} in Bayesian
setting.

Note that the term ``bandit'' usually does not appear in the relevant biostatistics literature.

\subsection{Combination Designs}

Many researchers have suggested to use one design for the first $ N $ subjects and use another one for the remaining $ ( T - N ) $ subjects. This idea
appeared implicitly in early papers on sequential design of experiments, considering the \design{1:1} design initially (on a sample of size $ N $), and
then sticking to the arm with higher mean (concluded at a given significance level) forever, since that approach assumes $ T = +\infty $.
\citet{ChengEtal2003biometrika} found that when $ T $ is finite, for this combination design it is optimal for $ N $ to be of the order of $ \sqrt{ T } $,
and, in the particular case of prior Beta distribution with parameters $ ( 1 , 1 ) $ on each arm, the optimal $ N \to \sqrt{ 2 T } $ asymptotically as $ T
\to +\infty $ (and is slightly lower for finite horizons).

\citet{HoelEtal1972biometrika} considered a setting in which $ T $ is not fixed in advance, but the allocation is stopped if a predefined difference in the
number of successes is reached. They proposed to start with the \emph{arm alternating} design (i.e., a deterministic version of \design{1:1}), and then to
switch to the ``stay-with-a-winner \& switch-on-a-loser'' design if the estimate of $ \max \{ \theta_{ C } , \theta_{ D } \}$ is greater than $ 0.6 $ and
continue with arm alternating otherwise. This paper is just one example from the work in the area of \emph{ranking and selection}.

\citet[Section 4]{Zelen1969jasa} proposed to initially use the ``stay-with-a-winner \& switch-on-a-loser'' design, and then to stick to the arm with higher
mean. He found that this is better than the classic combination if $ \theta_{ C } + \theta_{ D } \ge 1 $ and the value of $ N \approx T / 3 $ yields a
nearly-optimal number of successes.

\citet[Remark 4.10]{Kelly1981} elucidated that the optimal Bayesian decision-theoretical design in the discounted setting can be regarded as having three
stages: first, \design{BLFF}, which he refers to as ``information gathering'', then it moves further and further away from that design, and finally it
always allocates to the same arm. \citet[Section 3.4]{Villar2018peis} provided some illustrative numerical evaluation of these stages for the Gittins and
Whittle index rules.

Motivated by the above and by the discussion in \autoref{section:utopic}, in this paper we consider three combination designs in which the designs used are
\design{BLFF} in the first stage, followed by \design{BM}, \design{$ \alpha $UCB} and \design{BMSF}, respectively, and one combination rule in which the
designs used are \design{1:1} in the first stage, followed by \design{BMSF}:
\begin{itemize}

\item \design{BLFF+BM} with $ N = \sqrt{ 4 T } $

\item \design{BLFF+$ \alpha $UCB} with $ N = \sqrt{ 4 T } $

\item \design{BLFF+BMSF} with $ N = \sqrt{ 9 T } $

\item \design{1:1+BMSF} with $ N = \sqrt{ 2 T } $

\end{itemize}

The lengths of the first stage using \design{BLFF} have been obtained heuristically based only on a small numerical study and so they are not necessarily
overall optimal, and surely not optimal for each particular scenario. To the best of our knowledge, these three designs have not been studied previously
and their theoretical analysis is an open problem. Note that using \design{BMSF} in the second stage represents the situation in which the arm with higher
mean (most successes is approximately equivalent to highest mean since \design{BLFF} keeps the number of failures balanced up to a difference of $ 1 $)
after the first stage is allocated to all the subjects in the second stage since only the allocated arm can collect additional successes. This is inspired
by but not fully equivalent to the design studied by \citet[Section 4]{Zelen1969jasa}, because \design{BLFF} is not fully equivalent to
``stay-with-a-winner \& switch-on-a-loser''.

Using \design{BM} in the second stage is similar to \design{BMSF}, but allows for additional learning during the second stage: the mean of the arm with
higher mean at the end of the first stage may eventually decrease below the mean of the other arm, at which moment the allocation switches to that arm
(there may be more than one such ``correction''). The combination design that uses \design{$ \alpha $UCB} in the second stage is included because it has
interesting performance, to be discussed in \autoref{section:performance}.

Design \design{1:1+BMSF} is inspired by \citet{ChengEtal2003biometrika}, although using \design{BMSF} in the second stage is not fully equivalent to
sticking to the arm with higher mean at the end of the first stage, because \design{1:1} may lead to unbalanced allocation due to sampling variability.

\section{Designs Performance}
\label{section:performance}

For some of the response-adaptive designs, theoretical values of or bounds on their performance exist. However, these are either asymptotic as $ T \to
+\infty $, or up to an additive and/or multiplicative constants, which are often large or unknown. For small ($ T \approx 1 : 10^{ 2 } $) and moderate ($ T
\approx 10^{ 2 } : 10^{ 4 } $) horizons, computational evaluation is the most appropriate to illustrate designs performance.

In this section we report computational experiments in which every design is evaluated by backward recursion, i.e. at full computational accuracy (of
Float64, which is of the order of $ 10^{ \log( T ) - 16 } $). For fairness, we only include deterministic designs and \design{1:1} as a benchmark (i.e. we
exclude the biostatistics ones).

To the best of our knowledge, such accurate evaluation has only been reported for small horizons in the existing literature. For moderate horizons, which
are often the most relevant in practice, for instance in clinical trials, designs are usually evaluated by simulation, which is remarkably less accurate.

We present the performance in terms of both the proportion of successes and (equivalently) the regret number of successes, as these provide complementary
insights. See \autoref{section:performance_continued} for small horizons, in which the performance of some of the designs is \emph{fundamentally
different}.

\subsection{Bayesian Performance --- Moderate Trial Sizes}

\autoref{fig:2} illustrates the Bayesian performance of the Bayesian decision-theoretic design for $ T = 120 : 120 : 4440 $. Although it is an interesting
way of summarizing the performance, it has three major drawbacks: (i) it is a simple average over all the possible pairs of parameters $ ( \theta_{ C } ,
\theta_{ D } ) $, while in practice one would be interested in a particular subset of the whole parameter space, possibly weighted in a particular way;
(ii) it presents the Bayesian value in which the observations happen according to the belief at every time epoch rather than being fixed over the whole
horizon, which blurs its interpretation because the beliefs are biased and which is less desirable since in practice one would be interested in the value
under the true (unknown) success probabilities; (iii) it depends on the prior distributions, so a choice needs to be made or it needs to be computed for a
number of different options.

Nevertheless, \autoref{fig:2} is instrumental in giving an idea about the order of magnitude of the objective, and in particular it is interesting to
observe that the regret number of successes is increasing and concave, taking value of around $ 5 $ for the trial size of $ 3000 $ and, by extrapolation,
it is likely to be below one per mille of the trial sizes beyond $ 10^{ 4 } $.

\begin{figure}[tbp]
\centering
    \begin{subfigure}[b]{0.48\textwidth}
        \includegraphics[trim=0pt 0pt 0pt 0pt, clip=true, width=\textwidth]{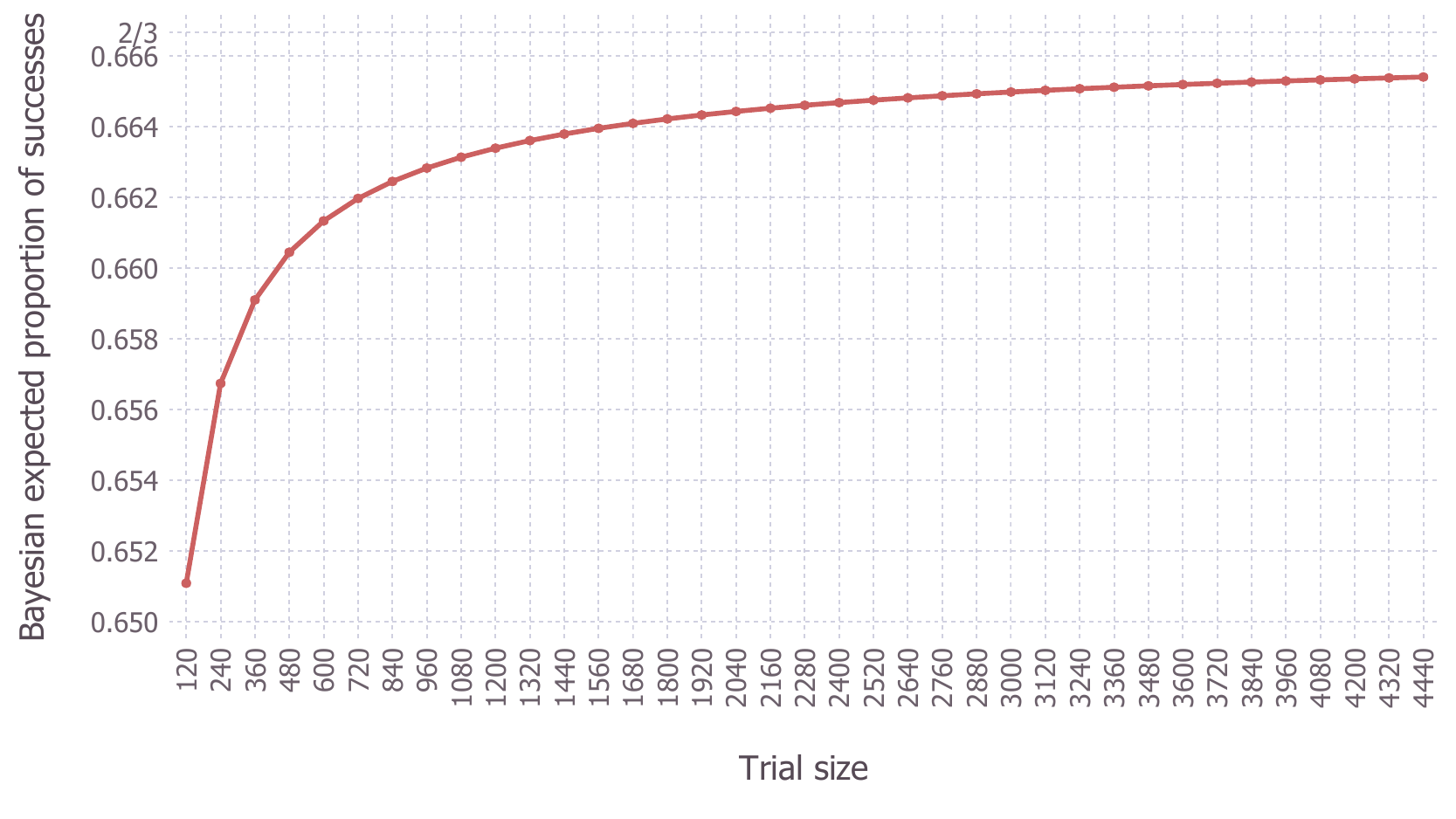}
    \end{subfigure}\hfill%
    \begin{subfigure}[b]{0.48\textwidth}
        \includegraphics[trim=0pt 0pt 0pt 0pt, clip=true, width=\textwidth]{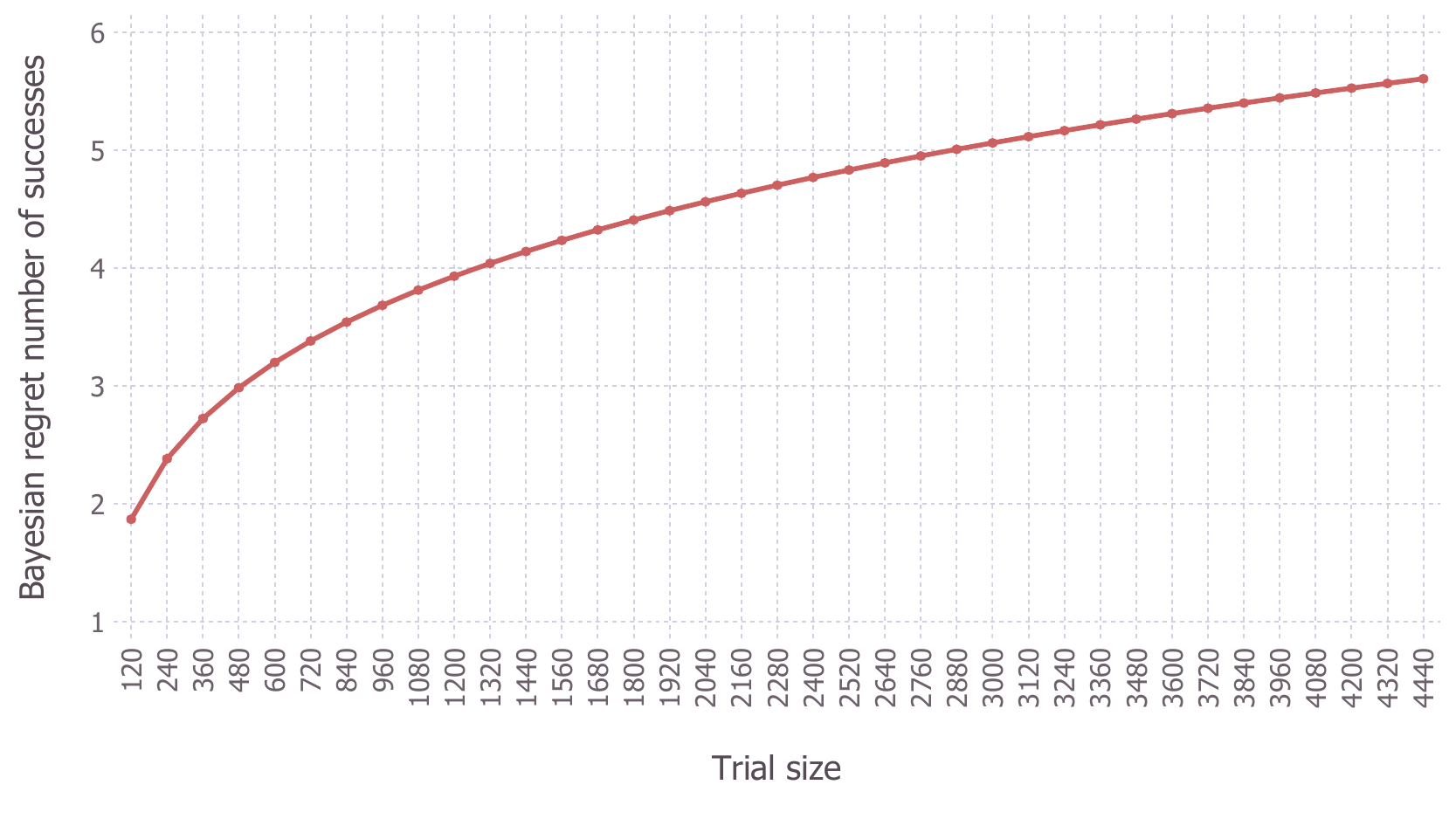}
    \end{subfigure}%
\caption{An illustration of two (equivalent) Bayesian performance measures evaluated for the deterministic Bayesian decision-theoretic design over a range of moderate trial sizes.}\label{fig:2}
\end{figure}

\subsection{Frequentist Performance --- Moderate Trial Sizes}

In this subsection we evaluate and compare the frequentist performance of the above designs for for $ T = 120 : 60 : 1200 \cup 1200 : 300 : 2400 $. We do
so in four scenarios: $ ( \theta_{ C } , \theta_{ D } ) \in \{ ( 0.1 , 0.3 ) , ( 0.3 , 0.5 ) , ( 0.5 , 0.7 ) , ( 0.7 , 0.9 ) \} $. These scenarios are a
natural choice, and have been previously considered, e.g., $ ( 0.3 , 0.5 ) $ in \citet[Table 3]{Lai1987as}, \citet[Table 2]{BrezziLai2002jedc},
\citet[Table 5]{VillarEtal2015survey} and \citet[Table 6]{Villar2018peis}, $ ( 0.7 , 0.9 ) $ in \citet[Table 3]{Lai1987as}, and all of them in
\citet{HoelEtal1972biometrika}. Although conclusions from a number scenarios do not guarantee their validity in other scenarios, we have found these four
scenarios to be illustrative enough to provide negative conclusions for many designs.

We report both the proportion of successes, in \autoref{fig:PS_2400}, and the regret number of successes, in \autoref{fig:regret_2400}. Besides the mean,
for each measure we report also the standard deviation, which can be considered as a secondary criterion providing additional insights into the performance
of the designs.

\design{1:1} is the worst design in terms of the mean and often the best in terms of the SD but, surprisingly, it is not always the best, meaning that in
some scenarios there are designs that are better under both measures (all of these are combination designs, to be discussed below). \design{BLFF} notably
improves the mean, increasingly so in the scenarios with higher success probabilities, while only marginally deteriorates the SD. They both over-explore.

The curves of \design{BMSF}, \design{FM}, \design{BM}, \design{BGDF} and \design{BKG} look approximately linear in both the mean and SD of the regret
number of successes (approximately constant in both the mean and SD of the proportion of successes). The performance of \design{BGDF} and \design{BKG} is
quite bad in general and heavily scenario dependent. \design{BGDF}'s mean deteriorates with more extreme success probabilities, but its SD improves with
lower success probabilities. \design{BKG}'s mean and SD both improve with lower success probabilities, remarkably so in scenario $ ( 0.1 , 0.3 ) $, in
which the mean regret is the best of all designs and seems to be constant, while all the other designs' mean regret is increasing.

There are three designs that perform particularly badly in all four scenarios: \design{BMSF}, \design{FM} and \design{BM} are dominated by almost all the
other designs. High SD indicates that these three designs are extremely \emph{under-exploring}: too aggressively sticking to one of the arms, and
relatively often choosing the worse one. Among these three, \design{BMSF} is worse than the other two in both the mean and SD in all four scenarios. The
performance order of \design{FM} and \design{BM} depends on scenario; \design{BM}'s mean remarkably improves with lower success probabilities, while
\design{FM}'s mean improves with more extreme success probabilities. For these three designs it also holds that the better the mean, the better the SD, so
in every scenario the order is identical under both measures.

The curves for the mean of all the remaining designs (\design{DP}, \design{$ \alpha $UCB}, and the combination designs) look approximately concave, leading
to significant improvement in the mean. Quantitatively, all three versions of \design{$ \alpha $UCB} are almost identical across the four scenarios,
indicating that its performance depends on the difference $ \theta_{ C } - \theta_{ D } $ rather than on their respective values. Interestingly, the mean
regret number of successes of \design{2UCB} is more than $ 50\% $ higher than that of \design{1UCB}, which is in turn around $ 3 $ times higher than that
of \design{0.18UCB}, which is in turn still distinctively higher (between $ 20\% - 100\% $, depending on scenario) than that of \design{DP} which is the
best performing design except when dominated by \design{BKG} in scenario $ ( 0.1 , 0.3 ) $. The mean regret of \design{DP} in scenario $ ( 0.7 , 0.9 ) $
seems to be constant, while all the other designs' mean regret is increasing. The SD of \design{2UCB} and \design{1UCB} is practically undistinguishable,
while the SD of \design{0.18UCB} is notably higher in the scenarios with higher success probabilities, with the SD of \design{DP} being in between.

The combination designs bring some surprising results. First, \design{BLFF+2UCB} is practically identical to \design{2UCB} and \design{BLFF+1UCB} to
\design{1UCB} in all the scenarios and trial sizes, indicating that \design{$ \alpha $UCB} initially behaves essentially as \design{BLFF}. That is however
not true for \design{BLFF+0.18UCB}, which is similar but not identical to \design{0.18UCB}: in terms of the mean being better in scenarios with higher
success probabilities while being worse in those with lower success probabilities. The SD of \design{BLFF+0.18UCB} is always lower than that of
\design{0.18UCB}, and in particular it is the lowest of all designs except when dominated by \design{1:1} and \design{BLFF} in scenarios with low success
probabilities.

Somewhat surprisingly, \design{BLFF+BM} performs quite well, roughly similarly to \design{BLFF+0.18UCB} in terms of the mean, although its SD is sometimes
higher. \design{BLFF+BMSF} performs worse than \design{BLFF+BM} but still significantly better than \design{1UCB} in terms of the mean, but its SD is
notably higher. \design{1:1+BMSF} is in most instances significantly worse than \design{BLFF+BMSF} both in terms of the mean and SD.

We summarize the above discussion as follows:
\begin{itemize}

  \item \design{DP} may be the preferred design in terms of the mean in scenarios with high and moderate success probabilities;

  \item \design{BKG} may be the preferred design in terms of the mean in scenarios with low success probabilities;

  \item \design{BLFF+0.18UCB} may be the overall preferred design in terms of both the mean (higher weight) and the SD (lower weight);

  \item \design{BLFF} may be the overall preferred design in terms of both the mean (lower weight) and the SD (higher weight);

  \item \design{1:1} may be the preferred design in terms of the SD in scenarios with low and moderate success probabilities;

  \item \design{BLFF+0.18UCB} may be the preferred design in terms of the SD in scenarios with high success probabilities;


\end{itemize}

\begin{figure}[tbp]
\centering
    \begin{subfigure}[b]{0.48\textwidth}
        \includegraphics[trim=0pt 0pt 0pt 0pt, clip=true, width=\textwidth]{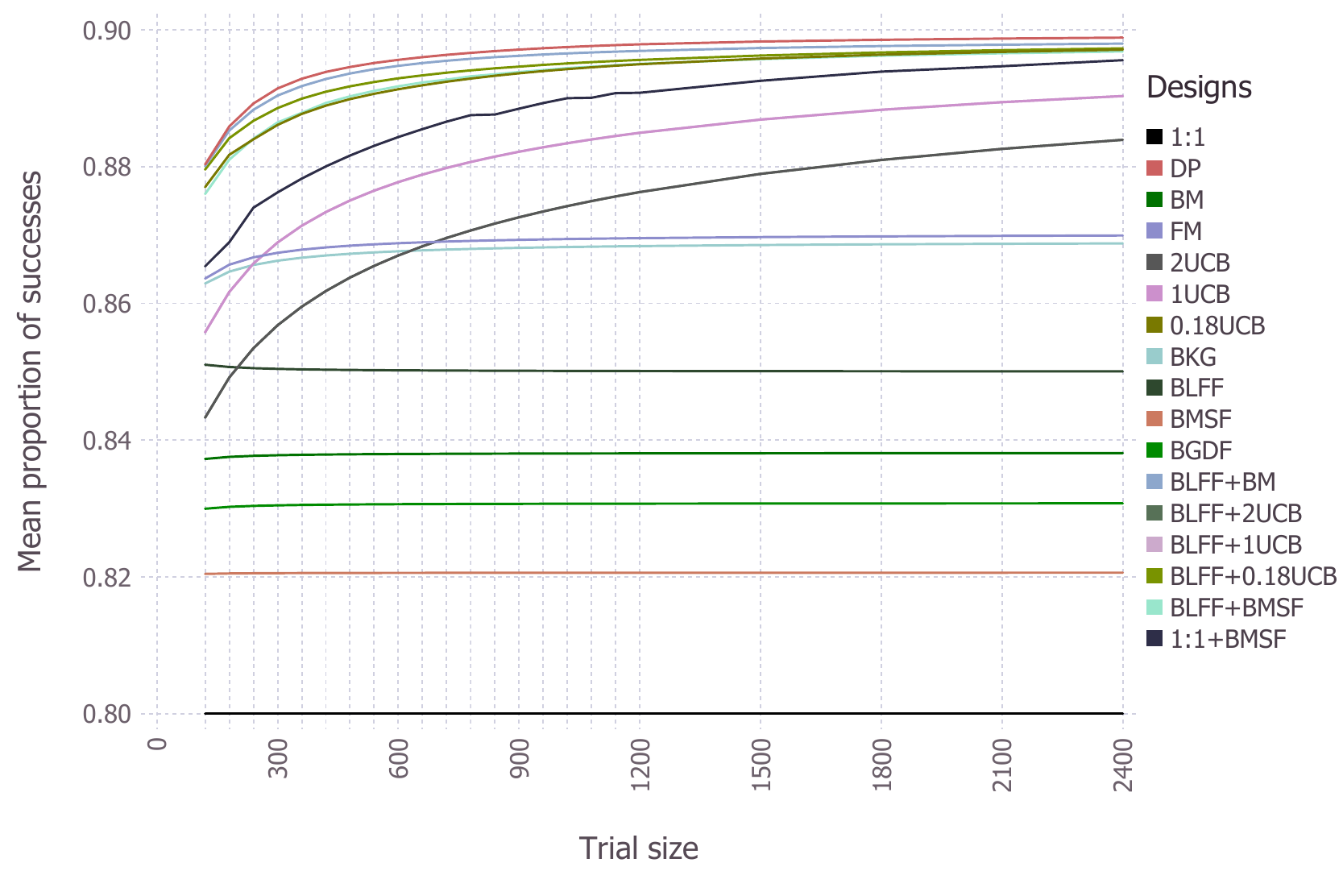}
    \end{subfigure}\hfill%
    \begin{subfigure}[b]{0.48\textwidth}
        \includegraphics[trim=0pt 0pt 0pt 0pt, clip=true, width=\textwidth]{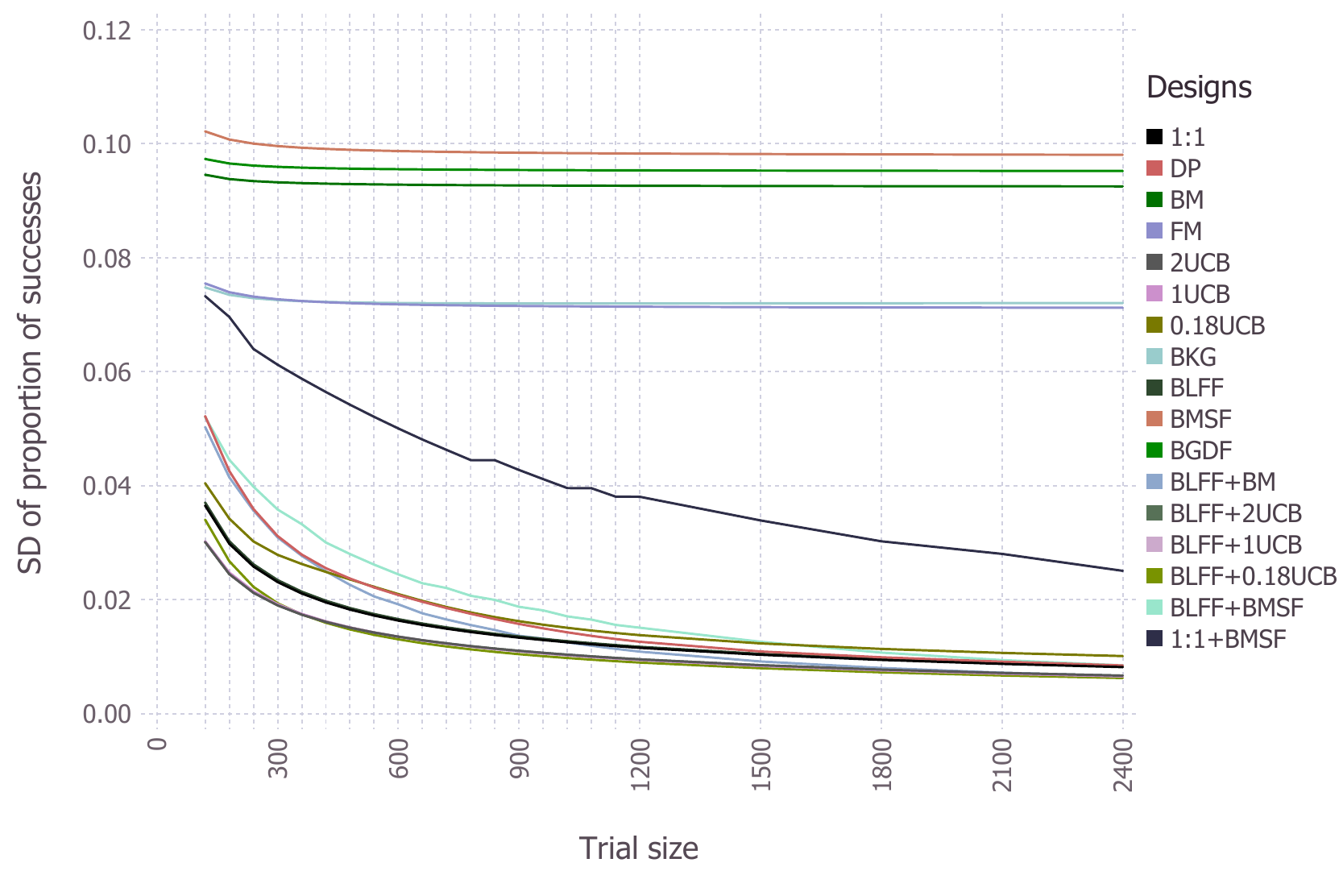}
    \end{subfigure}%

    \begin{subfigure}[b]{0.48\textwidth}
        \includegraphics[trim=0pt 0pt 0pt 0pt, clip=true, width=\textwidth]{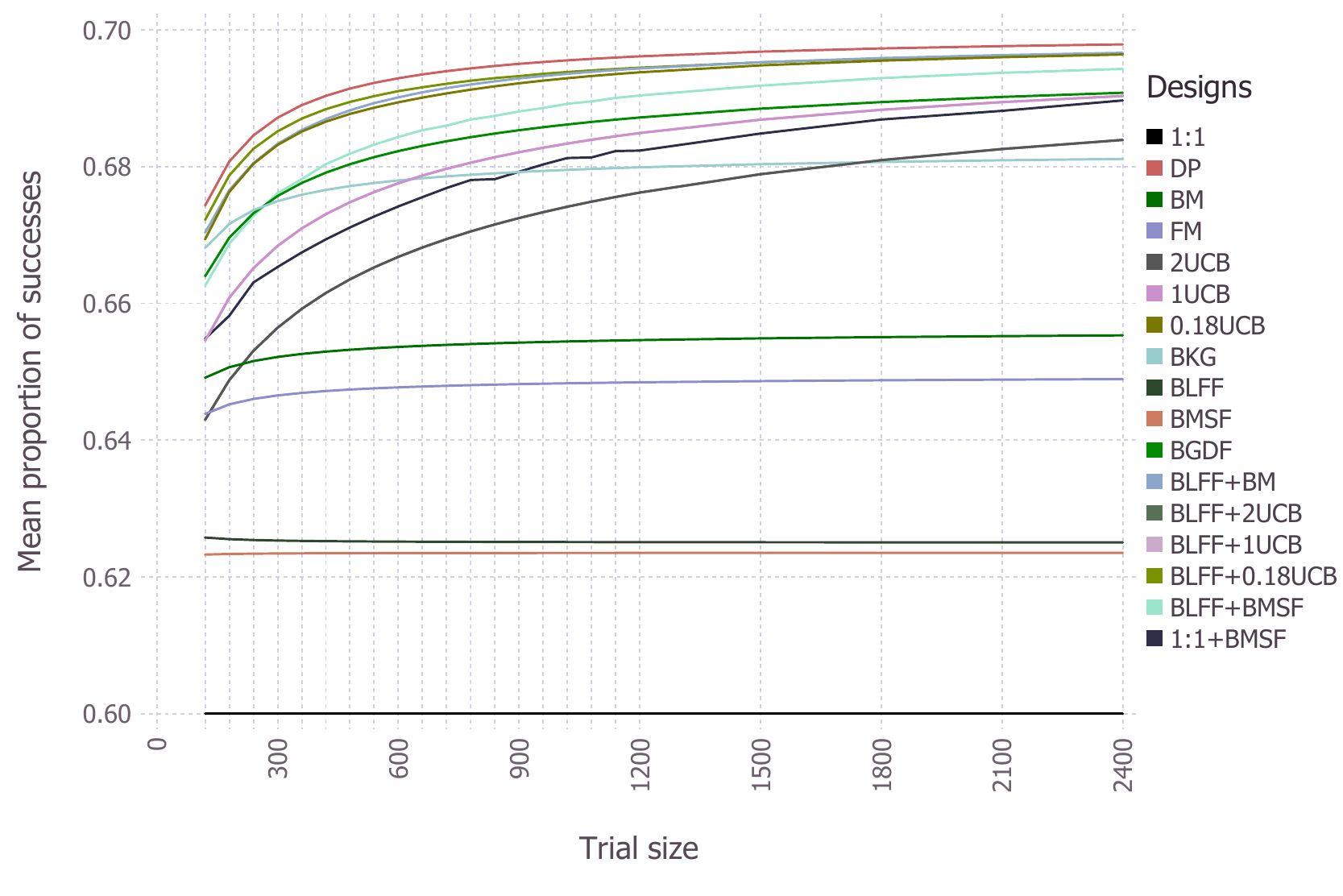}
    \end{subfigure}\hfill%
    \begin{subfigure}[b]{0.48\textwidth}
        \includegraphics[trim=0pt 0pt 0pt 0pt, clip=true, width=\textwidth]{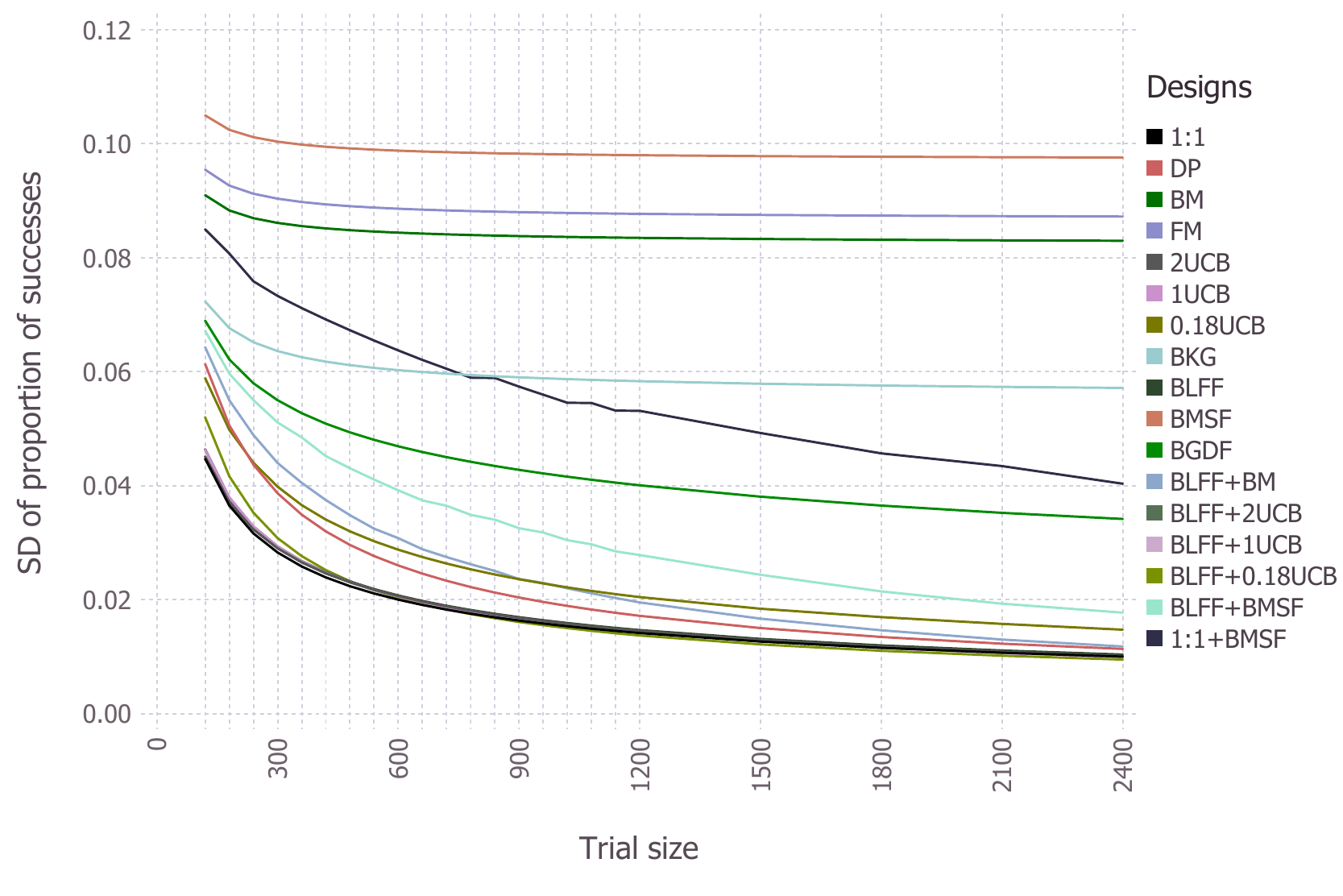}
    \end{subfigure}%

    \begin{subfigure}[b]{0.48\textwidth}
        \includegraphics[trim=0pt 0pt 0pt 0pt, clip=true, width=\textwidth]{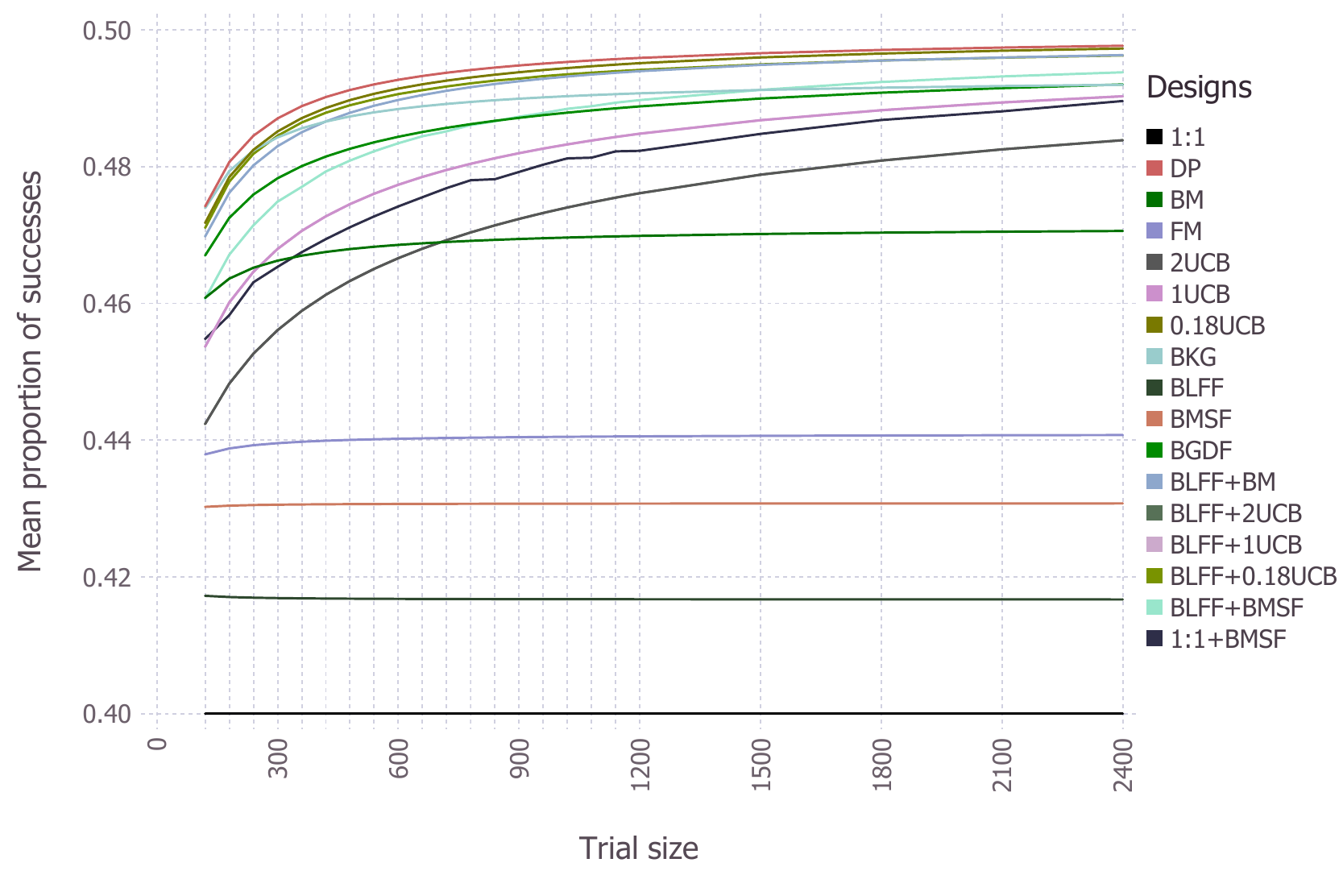}
    \end{subfigure}\hfill%
    \begin{subfigure}[b]{0.48\textwidth}
        \includegraphics[trim=0pt 0pt 0pt 0pt, clip=true, width=\textwidth]{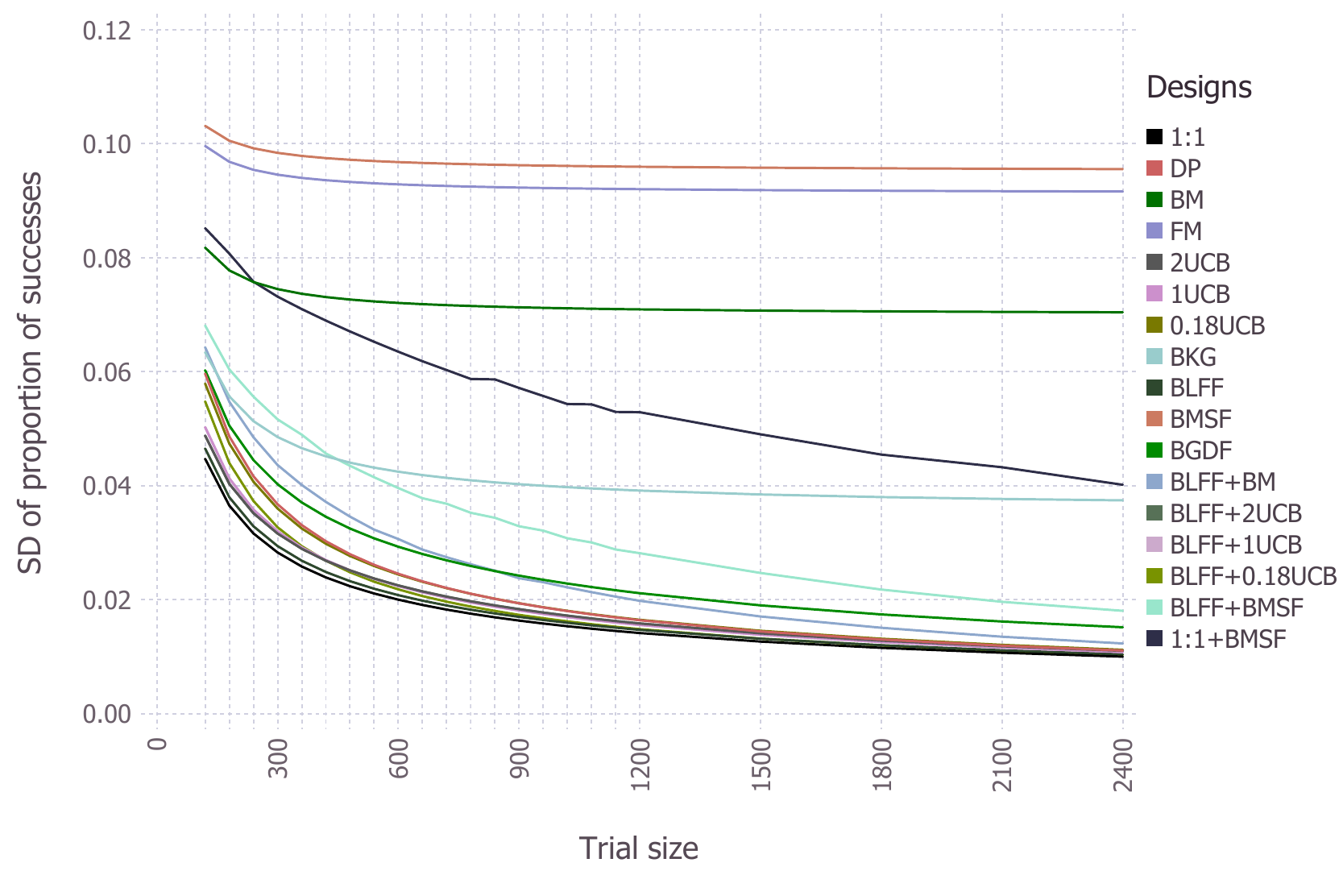}
    \end{subfigure}%

    \begin{subfigure}[b]{0.48\textwidth}
        \includegraphics[trim=0pt 0pt 0pt 0pt, clip=true, width=\textwidth]{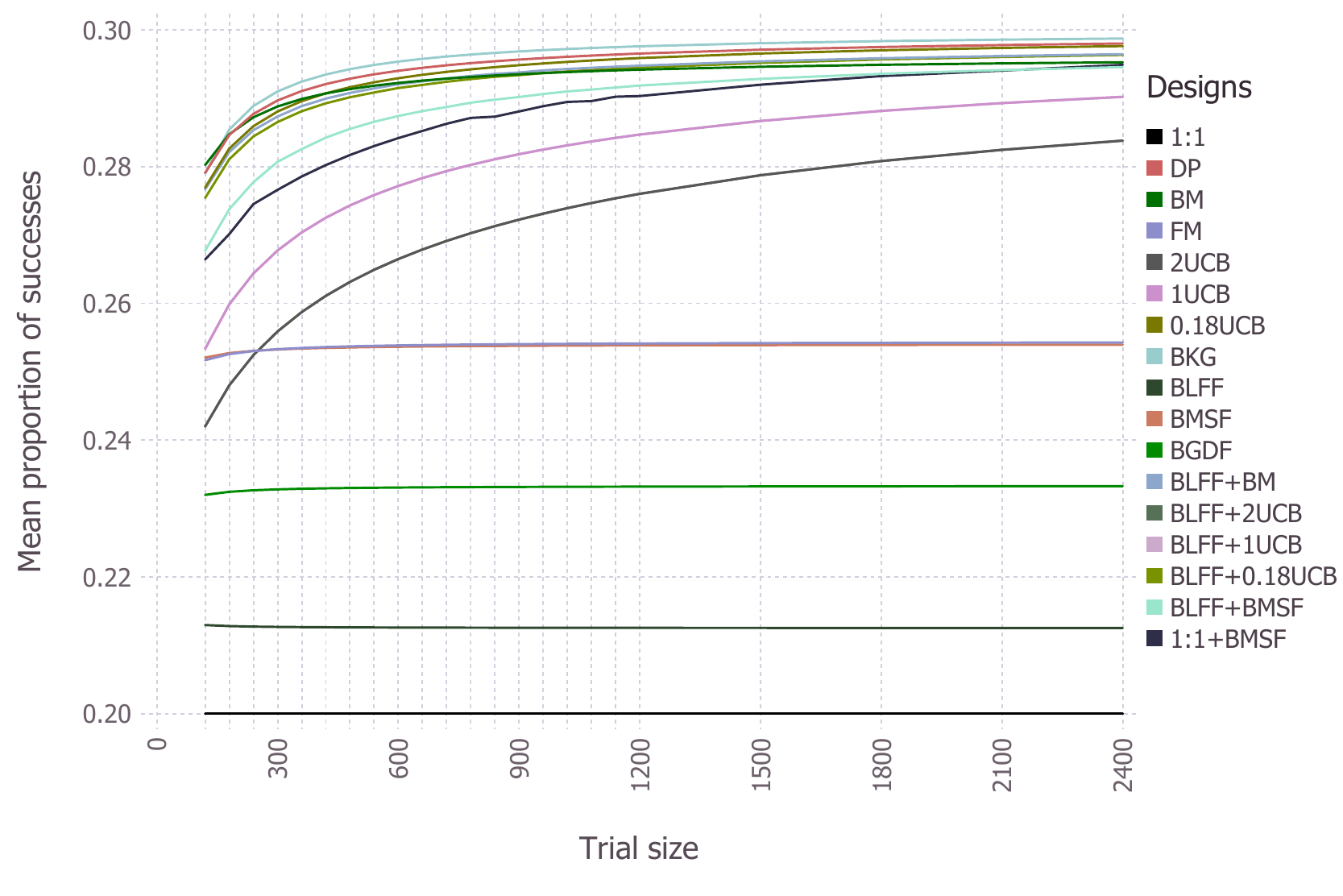}
    \end{subfigure}\hfill%
    \begin{subfigure}[b]{0.48\textwidth}
        \includegraphics[trim=0pt 0pt 0pt 0pt, clip=true, width=\textwidth]{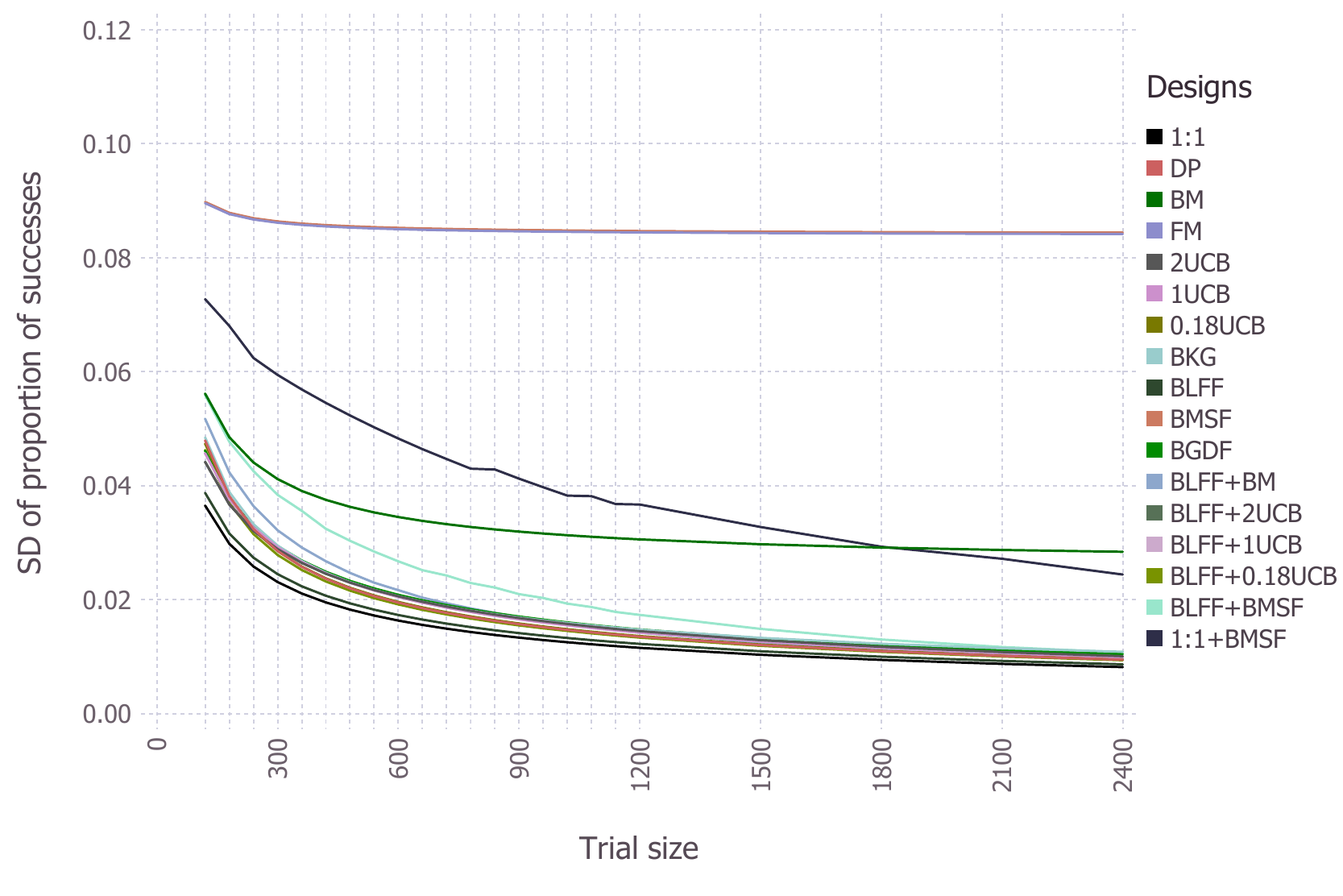}
    \end{subfigure}%
\caption{An illustration of performance (mean on the left, standard deviation on the right) in terms of the expected proportion of successes evaluated for
deterministic designs over a range of moderate trial sizes, for $ ( \theta_{ C } , \theta_{ D } ) = ( 0.7 , 0.9 ) $ in the first row, $ ( 0.5 , 0.7 ) $ in
the second row, $ ( 0.3 , 0.5 ) $ in the third row, $ ( 0.1 , 0.3 ) $ in the fourth row.}\label{fig:PS_2400}
\end{figure}

\begin{figure}[tbp]
\centering
    \begin{subfigure}[b]{0.48\textwidth}
        \includegraphics[trim=0pt 0pt 0pt 0pt, clip=true, width=\textwidth]{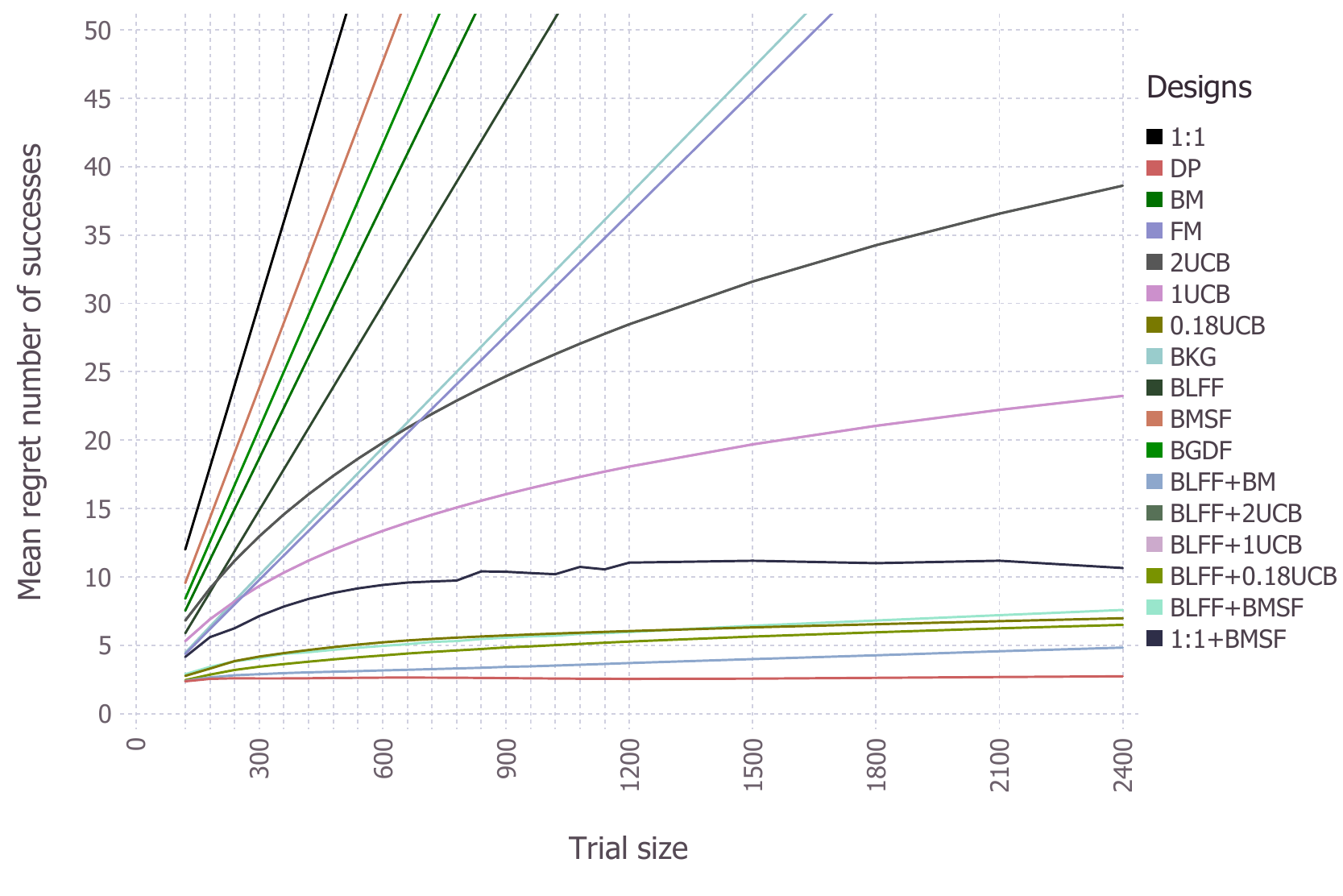}
    \end{subfigure}\hfill%
    \begin{subfigure}[b]{0.48\textwidth}
        \includegraphics[trim=0pt 0pt 0pt 0pt, clip=true, width=\textwidth]{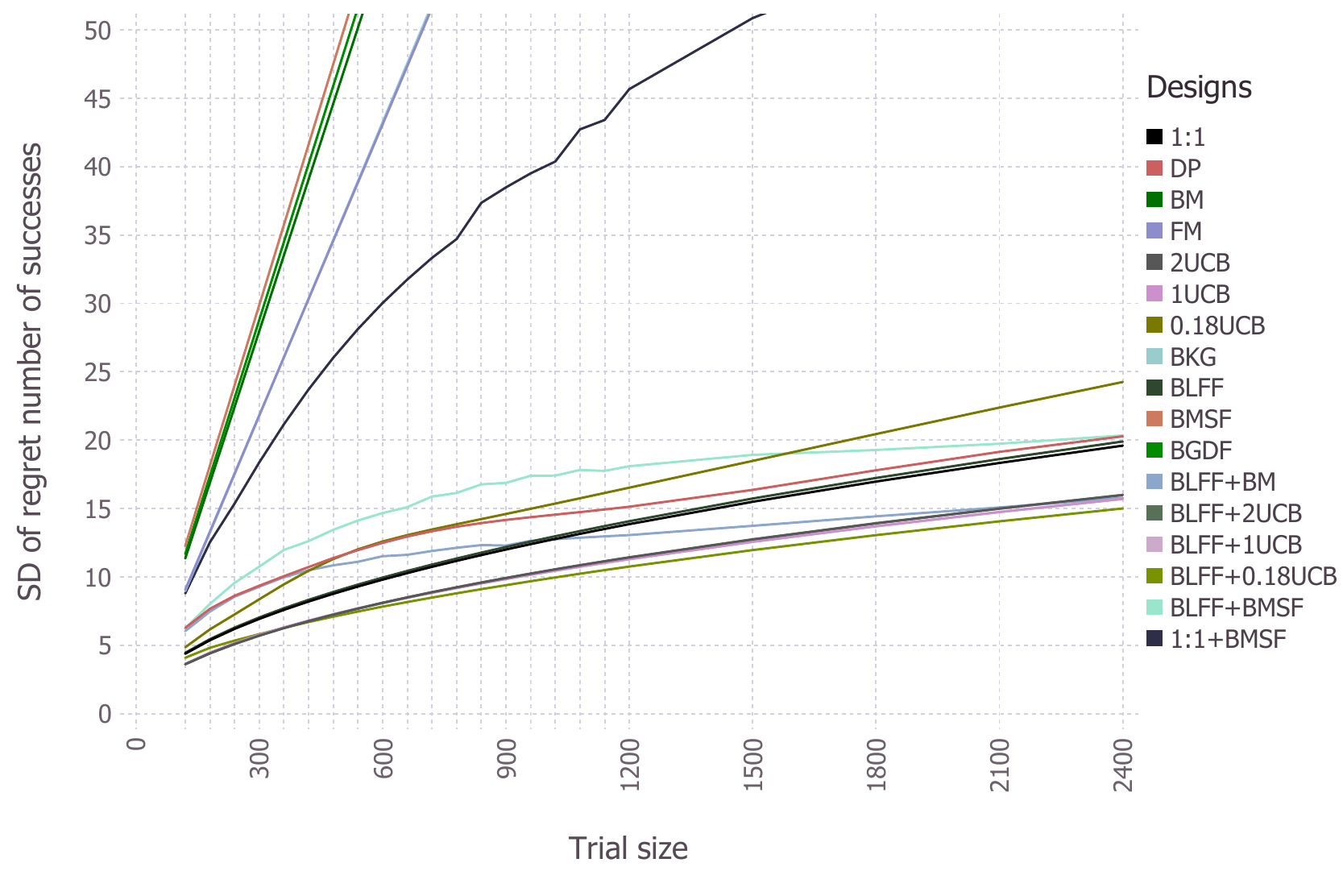}
    \end{subfigure}%

    \begin{subfigure}[b]{0.48\textwidth}
        \includegraphics[trim=0pt 0pt 0pt 0pt, clip=true, width=\textwidth]{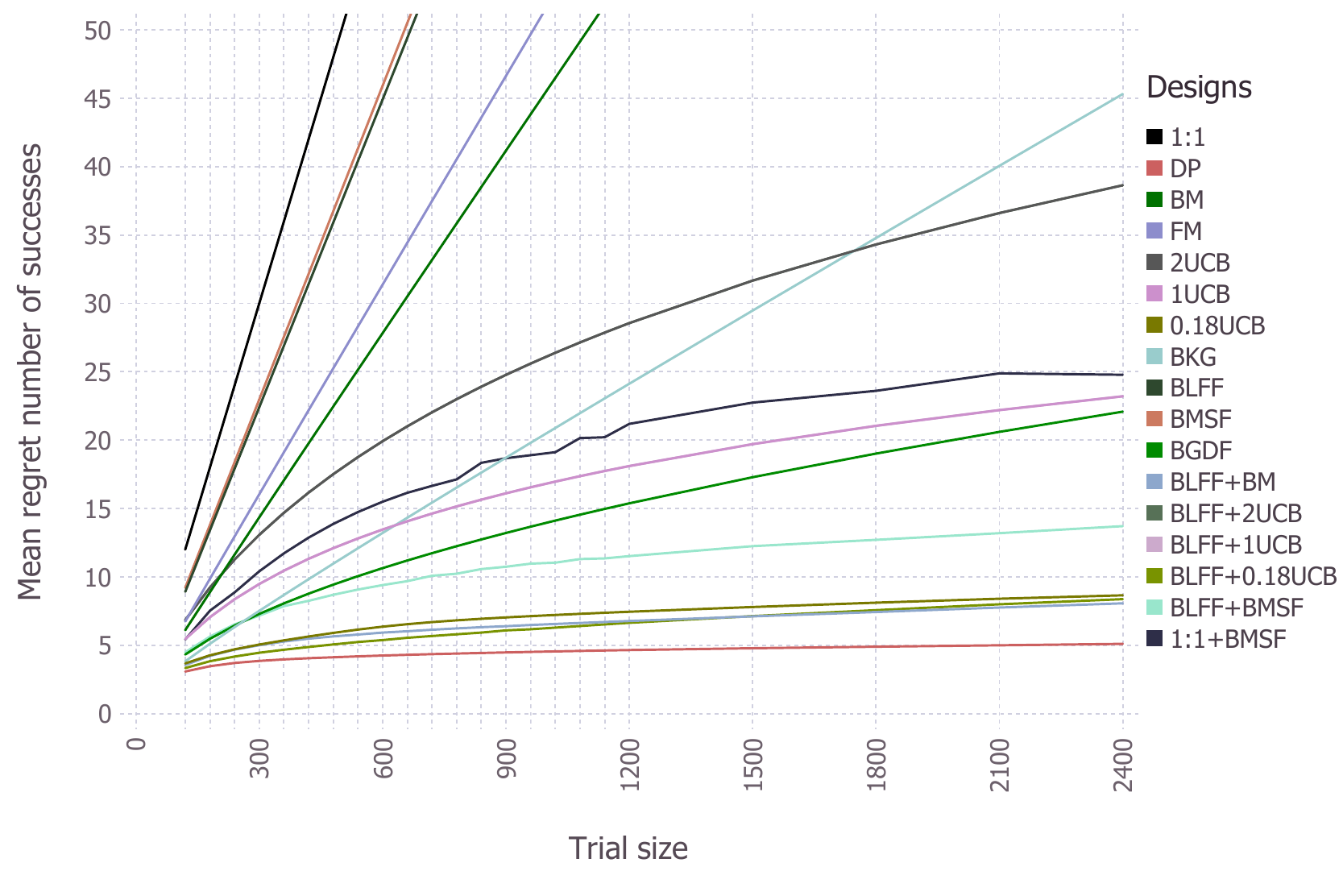}
    \end{subfigure}\hfill%
    \begin{subfigure}[b]{0.48\textwidth}
        \includegraphics[trim=0pt 0pt 0pt 0pt, clip=true, width=\textwidth]{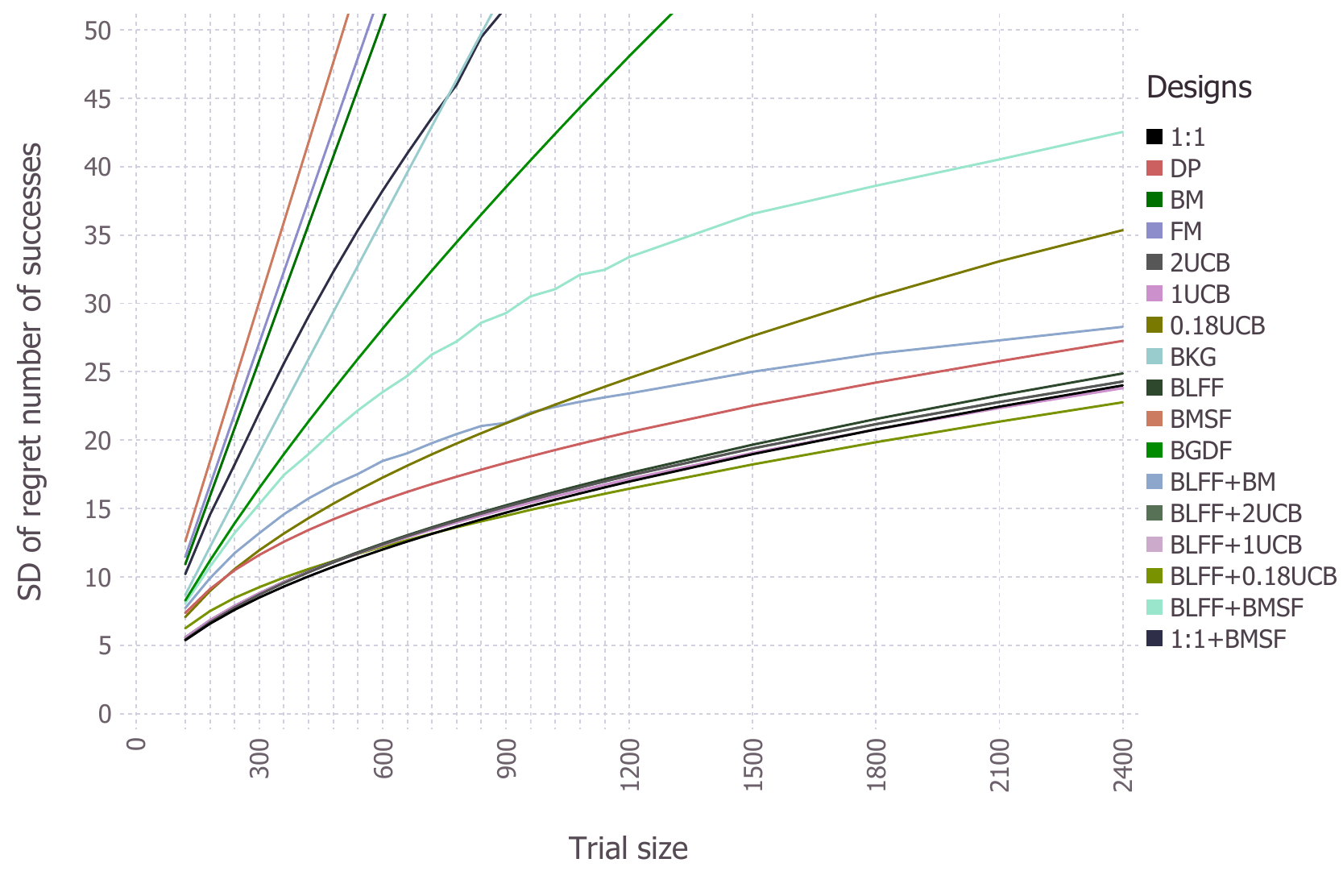}
    \end{subfigure}%

    \begin{subfigure}[b]{0.48\textwidth}
        \includegraphics[trim=0pt 0pt 0pt 0pt, clip=true, width=\textwidth]{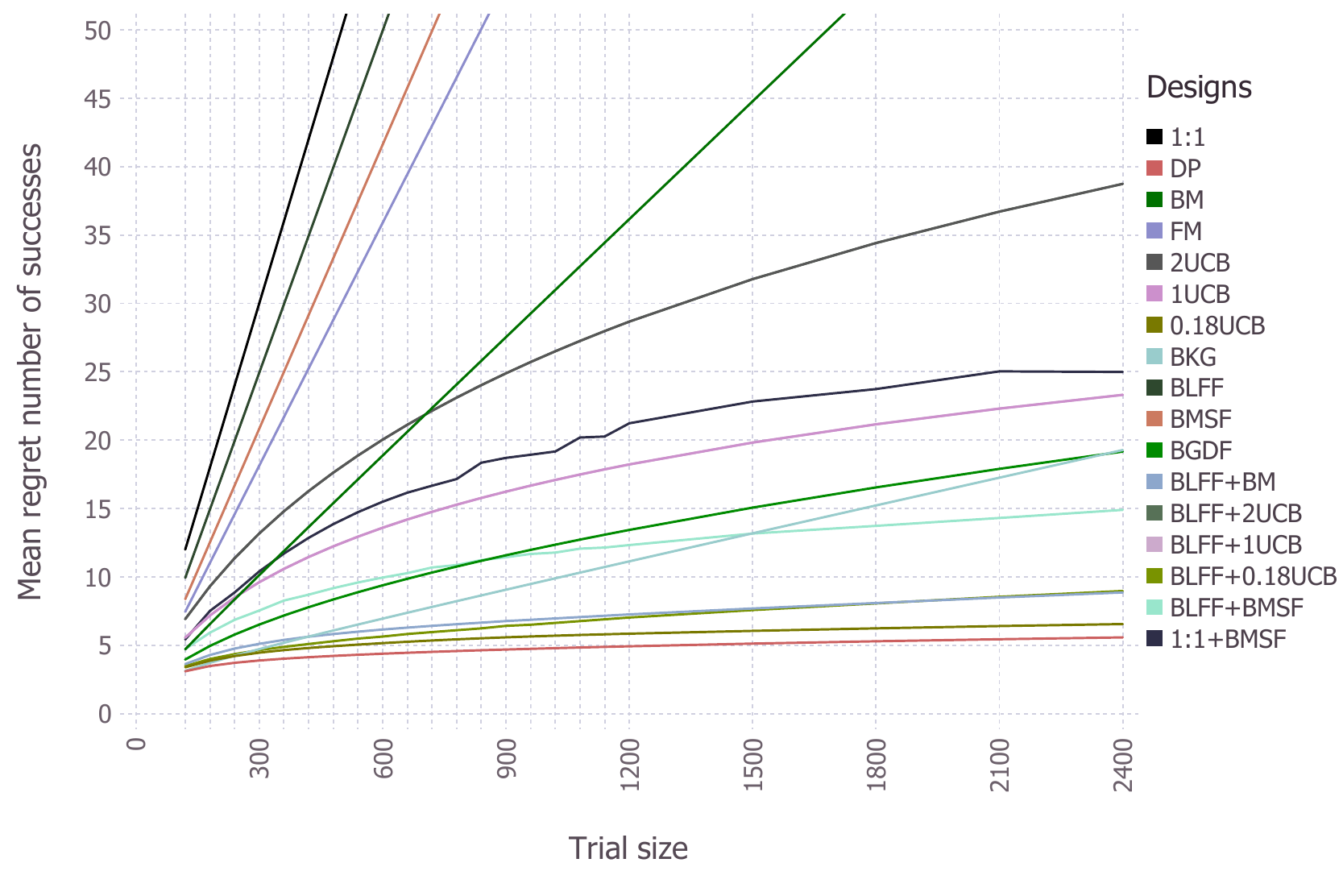}
    \end{subfigure}\hfill%
    \begin{subfigure}[b]{0.48\textwidth}
        \includegraphics[trim=0pt 0pt 0pt 0pt, clip=true, width=\textwidth]{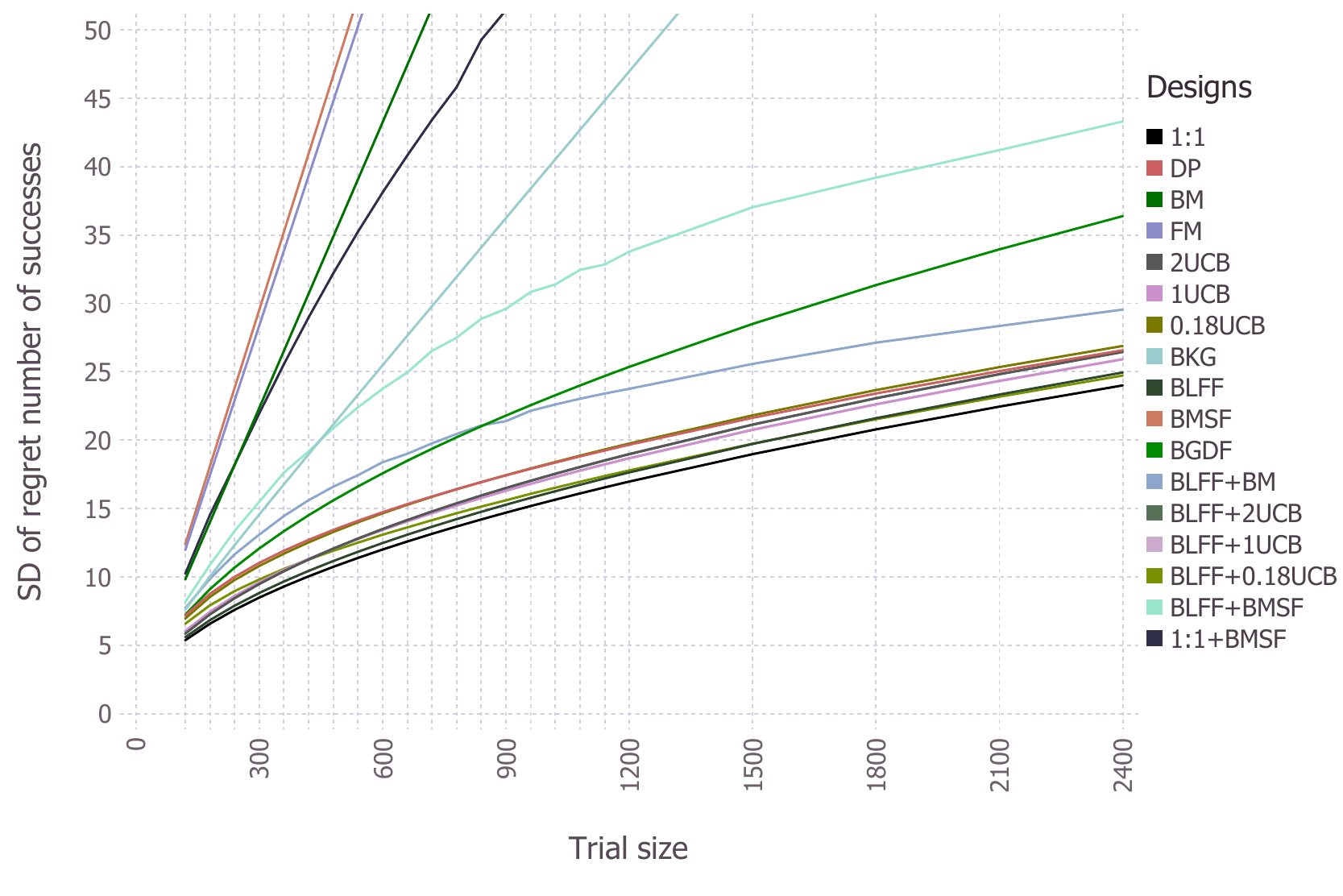}
    \end{subfigure}%

    \begin{subfigure}[b]{0.48\textwidth}
        \includegraphics[trim=0pt 0pt 0pt 0pt, clip=true, width=\textwidth]{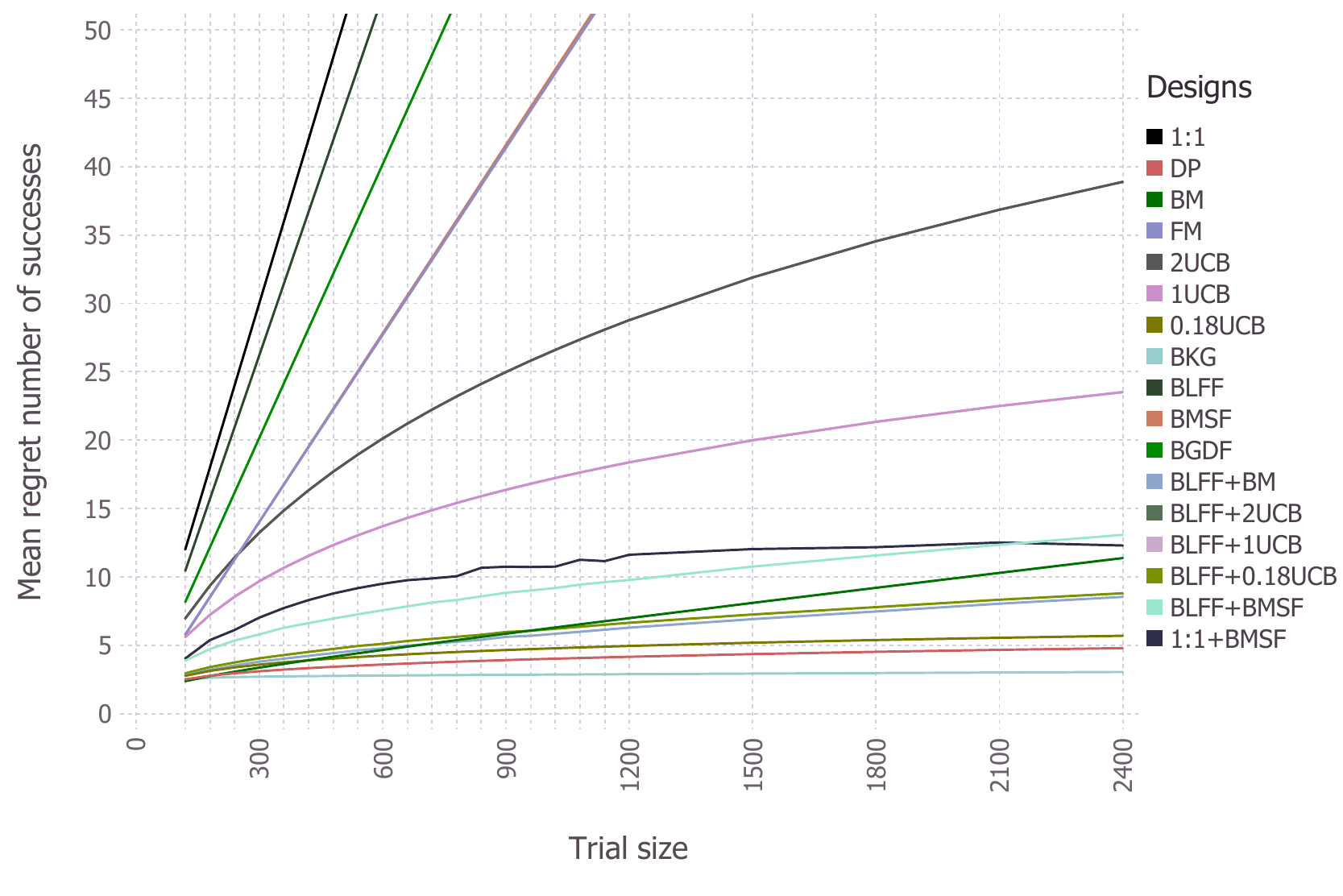}
    \end{subfigure}\hfill%
    \begin{subfigure}[b]{0.48\textwidth}
        \includegraphics[trim=0pt 0pt 0pt 0pt, clip=true, width=\textwidth]{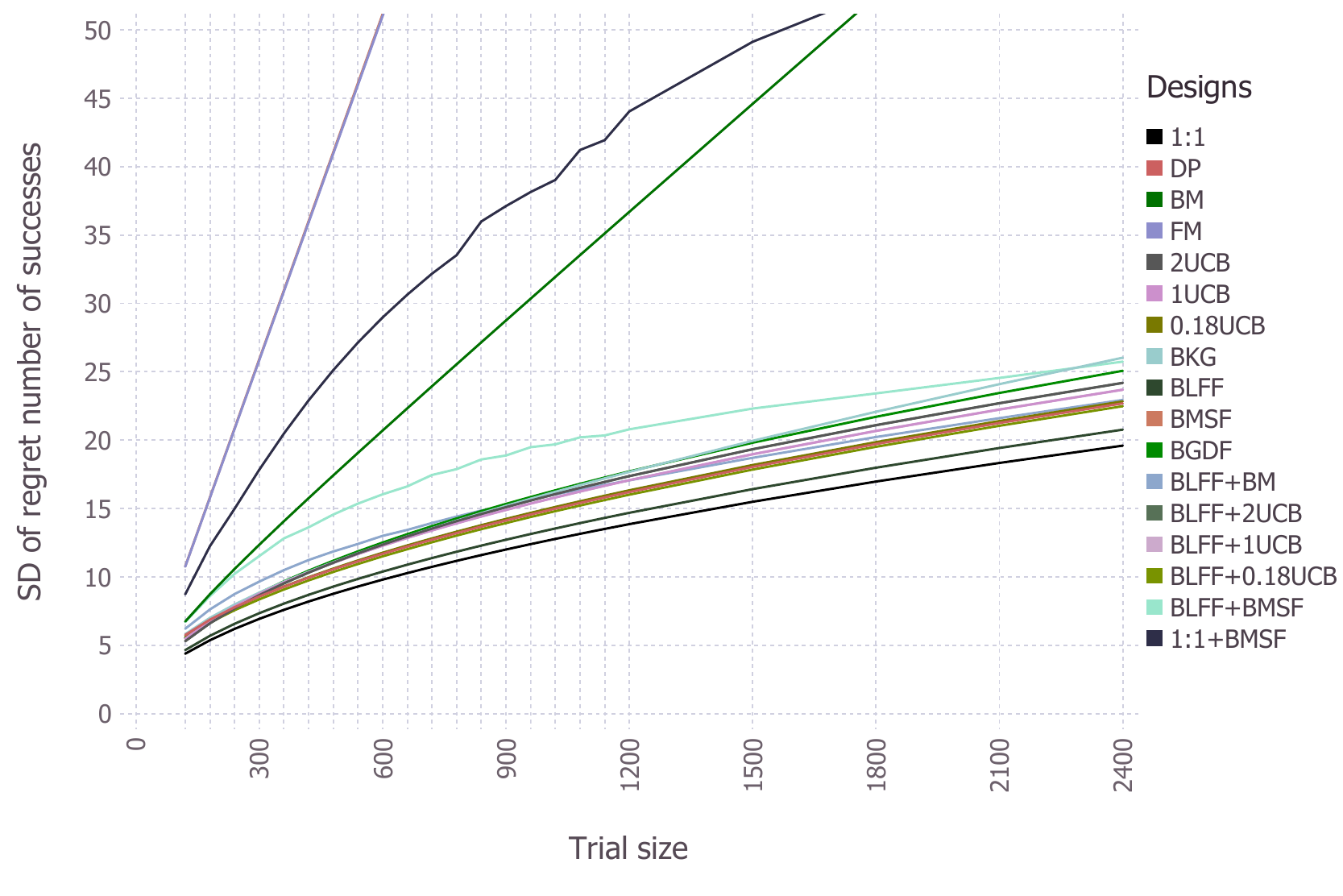}
    \end{subfigure}%
\caption{An illustration of performance (mean on the left, standard deviation on the right) in terms of the regret number of successes evaluated for
deterministic designs over a range of moderate trial sizes, for $ ( \theta_{ C } , \theta_{ D } ) = ( 0.7 , 0.9 ) $ in the first row, $ ( 0.5 , 0.7 ) $ in
the second row, $ ( 0.3 , 0.5 ) $ in the third row, $ ( 0.1 , 0.3 ) $ in the fourth row.}\label{fig:regret_2400}
\end{figure}

\section{Myths}
\label{section:myths}

The hardness of the problem and the diversity of approaches mean that it is difficult for a single person to have complete information and full
understanding of the many details. This paper is, to the best of our knowledge, first attempt to survey key ideas from multiple disciplines. Author's
engagement with researchers and practitioners from a number of fields has given him an impression of common beliefs that may not be fully correct, listed
next.

\subsection{Myth \#1: The Bayesian Decision-Theoretic Design is Intractable}

This myth seems to be widespread across disciplines, e.g., ``The curse of dimensionality renders exact approaches impractical.''
\citep{AhujaBirge2019report}; ``Even in relatively benign setups the computation of the Bayesian optimal policy appears hopelessly intractable.''
\citep[Section 35.1]{LattimoreSzepesvari2019book}; ``Trials of a size up to 3500 patients would be feasible with today’s number \#1 supercomputer (with 1.3
PB of RAM).'' \citep{WilliamsonEtal2017csda}.

\autoref{table:DP_horizons} illustrates the evolution of what was reported as computationally tractable in the literature. Indeed, besides the two results
in the 1960s that pre-date the era of personal computers, the existing literature gives an impression that solving the problem beyond $ T = 100 $ is
impractical or unfeasible. A closer look however reveals that no improvement seem to have been achieved since \citet{Berry1978jasa} which appeared more
than 40 years ago, despite the theoretical progress in computer science and technological progress in personal computers. In this paper we use a
state-of-the-art package in Julia programming language written by the author, and show that on a standard laptop or desktop computer, the \design{DP}
design can be computed for offline use or evaluated in online fashion in a few minutes ($ T \approx 1,000 $), in a few hours ($ T \approx 2,000 $), or in a
few days ($ T \approx 4,000 $). 32GB  of RAM allows storing (e.g. for offline use) of the whole design up to trial size around $ 1440 $; when its storing
is not needed (e.g. for Bayesian evaluation or for calculation of the initial action) it allows up to trial size around $ 4440 $. \footnote{The first
version of the package is planned to be released to public in mid-2019. The package will be described in a separate paper, whose abstract has been accepted
by the editors of the special issue on Bayesian Statistics of the Journal of Statistical Software, and will be submitted by its deadline 30 June 2019.}

\begin{table}[tbp]
\centering
\begin{tabular}{llll}
\toprule
Publication                     & $ T $     & $ T^{ \max } $   & SW, HW, RAM \\
\midrule
\citet{Steck1964moc}            & $ 25 $    & N/A       & N/A, UNIVAC 1105, 54 kB \\ 
\citet{Yakowitz1969book}        & $ 5 $     & N/A       & Fortran, N/A, N/A \\
\citet{Berry1978jasa}           & $ 100 $   & N/A       & Basic (?), Atari (?), N/A \\
\citet{GinebraClayton1999jspi}  & $ 150 $   & $ 180 $   & N/A, N/A, N/A \\
\citet{HardwickEtal2006}        & $ 100 $   & $ 200 $   & N/A, N/A, N/A \\ 
\citet{AhujaBirge2016ejor}      & $ 96 $    & $ 240 $   & N/A, Mac 4GB \\ 
\citet{WilliamsonEtal2017csda}  & $ 100 $   & $ 215 $   & R, PC, 16GB \\
\citet{Villar2018peis}          & $ 100 $   & N/A       & Matlab, PC, N/A \\
\citet{Kaufmann2018aos}         & $ 70 $    & N/A       & N/A, N/A, N/A \\
This paper                      & $ 4440 $  & $ 4440 $  & Julia 1.0.1 \& BB, PC, 32GB \\
\bottomrule
\end{tabular}
\caption{Horizons $ T $ with reported results and $ T^{ \max } $ reported as the largest computationally tractable on a standard computer by backward recursion in the literature on two-armed problem with binary responses.}\label{table:DP_horizons}
\end{table}

\subsection{Myth \#2: The Bayesian Decision-Theoretic Design is Optimal}

The \design{DP} is Bayes-optimal, but that does not necessarily mean that it is optimal in the non-Bayesian objective, i.e. that it achieves the largest
possible mean number of successes. \autoref{fig:regret_2400} and \autoref{fig:regret_240} clearly show that other designs perform better in some scenarios.

\subsection{Myth \#3: The Gittins Index Rule Leads to Incomplete Learning}

According to \citet{BrezziLai2000econometrika}: ``...we give in Section 3 a simple proof of the incompleteness of optimal learning from endogenous data in
the discounted multi-armed bandit problem... This generalizes \citet{Rothschild1974jet}'s result for Bernoulli two-armed bandits, and also the result of
\citet{BanksSundaram1992econometrika} who show that there is positive probability of incomplete learning in multi-armed bandits with general distributions
of rewards if the priors have finite support.'' This is however true only in the discounted setting; \citet{Kelly1981} proved that the structure of the
Gittins index rule in the undiscounted setting (see \autoref{section:utopic}) leads to complete learning.

\subsection{Myth \#4: The Gittins Index is Computationally Simpler than Dynamic Programming for the Two-Armed Problem}

As we explain in \autoref{section:Gittins}, there is a trade-off between accuracy and computational complexity of the Gittins index. For computing the
Gittins index, all the algorithms use a truncation of the horizon, and due to numerical instability most of the numerical algorithms use a discount factor
strictly lower than one although the appropriate one would be equal to one. So, there are three levels of approximation that the Gittins index requires,
and such an approximate Gittins index rule may in the end not even perform better than the Myopic index rule (see e.g. \citet{AhujaBirge2016ejor}).

\subsection{Myth \#5: The UCB Index Rule is Near-Optimal for the Finite-Horizon Problem}

\autoref{section:performance} illustrates that the \design{2UCB} and \design{1UCB} designs are significantly suboptimal, yielding five to tenfold mean
regret number of successes than the optimal design. Tuning the coefficient even beyond the intervals required by theoretical analysis can significantly
improve the performance, e.g. \design{0.18UCB} yields around twofold mean regret number of successes than the optimal design, which is still too large to
be considered near-optimal.

\subsection{Myth \#6: The Frequentist Mean Regret Number of Successes is Increasing}

\citet{LaiRobbins1985aam} presented a lower bound, valid under certain technical assumptions, indicating that the frequentist mean regret number of
successes increases proportionally to $ \log T $ as $ T \to \infty $. Our computational results indicate that this is not necessarily true for all designs
over finite horizons. For instance, the mean regret of \design{DP} in scenario $ \theta_{ C } = 0.7 , \theta_{ D } = 0.9 $ (see
\autoref{table:performance79}) at $ T = 240 $ is lower than that at $ T = 300 $, and it is non-increasing over $ T = 660 : 60 : 1200 $.

\section{Conclusion}
\label{section:conclusion}

The aim of this paper to provide a unified account of approaches from various disciplines to the two-armed bandit problem. We have proposed problem
terminology that we believe should not create conflicts with other existing terminology in most of the disciplines, to facilitate mutual learning. We have
proposed to use backward recursion instead of simulation for more accurate evaluation. We have created a computational comparison of designs from different
disciplines in a standardized set of scenarios. The computational results have showed that some of the simple ones (e.g. \design{BLFF+BM} and
\design{BLFF+BMSF}) perform surprisingly well in our scenarios and outperform may of the more sophisticated and more studied ones. We have also showed that
\design{DP} is tractable for much larger horizons than it is commonly believed. This suggests that there is a case for these to be used among benchmark
designs when developing new designs.

We have given an account of approaches to the problem with the objective of maximizing the mean number of successes, which is linear across arms and over
time. Most of the above designs, especially those horizon-ignorant, crucially depend on that property, and it is not clear how they could be modified if
the objective changed to another one. The only exception is the \design{DP} and its variants, which are quite flexible to accommodate other finite-horizon
objectives.

A significant area of research left out of this paper deals with practicalities of implementation of the designs, especially in the context of clinical
trials, in which (i) the objective is different, because it focusses much more on estimation of the success probabilities, and (ii) there are additional
constraints, e.g. requirement of a certain degree of randomization. A fair comparison of such designs would actually require a series of comparisons fixing
the randomization degree and including only those designs that satisfy it. Such work is extensive and is left for a separate paper.

\bibliography{BB_v6_arxiv}

\appendix

\section{Choosing the Prior Beta Distribution Parameters}
\label{section:priors}

While in the paper we have used the conventional Bayes parameters $ \left( \widetilde{ \successes } ( 0 ) , \widetilde{ \failures } ( 0 ) \right) = ( 1 , 1
) $ for the prior Beta distribution of each arm, the trials in practice are often designed based on pilot trials, expert opinion, or real-world evidence,
which can be summarized by the mean and the variance of the estimate of the success probability.

For Beta distribution with parameters $ \widetilde{ \successes } > 0 , \widetilde{ \failures } > 0 $, the mean and the variance are
\begin{align}
\mu ( \widetilde{ \successes } , \widetilde{ \failures } ) &= \frac{ \widetilde{ \successes } }{ \widetilde{ \successes } + \widetilde{ \failures } } , &
\sigma^2 ( \widetilde{ \successes } , \widetilde{ \failures } ) &= \frac{ \mu ( \widetilde{ \successes } , \widetilde{ \failures } ) ( 1 - \mu ( \widetilde{ \successes } , \widetilde{ \failures } ) ) }{ \widetilde{ \successes } + \widetilde{ \failures } + 1 } .
\end{align}
From these expressions we can obtain that for given mean $ \mu $ and variance $ \sigma^2 $, the parameters of the Beta distribution are
\begin{align}
\widetilde{ \successes } ( \mu , \sigma^2 ) &= \mu \left[ \frac{ \mu ( 1 - \mu ) }{ \sigma^2 } - 1 \right] , &
\widetilde{ \failures } ( \mu , \sigma^2 ) &= ( 1 - \mu ) \left[ \frac{ \mu ( 1 - \mu ) }{ \sigma^2 } - 1 \right] .
\end{align}

If parameters are $ \widetilde{ \successes } = 0 , \widetilde{ \failures } > 0 $, the distribution is degenerate and the mean and the variance are
\begin{align}
\mu ( \widetilde{ \successes } , \widetilde{ \failures } ) &= 0 , &
\sigma^2 ( \widetilde{ \successes } , \widetilde{ \failures } ) &= 0 .
\end{align}
These are the only Beta distributions with mean equal to $ 0 $. Such mean thus requires variance equal to $ 0 $, and parameter $ \widetilde{ \successes } =
0 $ while $ \widetilde{ \failures } $ cannot be uniquely determined.

If parameters are $ \widetilde{ \successes } > 0 , \widetilde{ \failures } = 0 $, the distribution is degenerate and the mean and the variance are
\begin{align}
\mu ( \widetilde{ \successes } , \widetilde{ \failures } ) &= 1 , &
\sigma^2 ( \widetilde{ \successes } , \widetilde{ \failures } ) &= 0 .
\end{align}
These are the only Beta distributions with mean equal to $ 1 $. Such mean thus requires variance equal to $ 0 $, and parameter $ \widetilde{ \failures } = 0
$ while $ \widetilde{ \successes } $ cannot be uniquely determined.

If parameters are $ \widetilde{ \successes } = 0 , \widetilde{ \failures } = 0 $, the distribution is a symmetric two-point Bernoulli distribution with
support $ \{ 0 , 1 \} $, and the mean and the variance are
\begin{align}
\mu ( \widetilde{ \successes } , \widetilde{ \failures } ) &= \frac{ 1 }{ 2 } , &
\sigma^2 ( \widetilde{ \successes } , \widetilde{ \failures } ) &= \frac{ 1 }{ 4 } .
\end{align}
This is the only Beta distribution with the largest possible variance, $ 1 / 4 $. Such variance thus requires mean equal to $ 1 / 2 $, and parameters are $
\widetilde{ \failures } = 0 , \widetilde{ \successes } = 0 $.

\begin{sidewaystable}[tbp]
\centering
\setlength{\columnsep}{30pt}
\footnotesize
\begin{tabular}{>{\bfseries}ccccccccccc}
\toprule
& \bfseries 0.05 & \bfseries 0.10 & \bfseries 0.15 & \bfseries 0.20 & \bfseries 0.25 & \bfseries 0.30 & \bfseries 0.35 & \bfseries 0.40 & \bfseries 0.45 & \bfseries 0.50 \\
\midrule
1/140 & (0.28, 5.37)   & (1.16, 10.44)  & (2.53, 14.32)  & (4.28, 17.12)  & (6.31, 18.94)  & (8.52, 19.88)  & (10.80, 20.05) & (13.04, 19.56) & (15.14, 18.51) & (17.00, 17.00) \\
1/136 & (0.27, 5.19)   & (1.12, 10.12)  & (2.45, 13.89)  & (4.15, 16.61)  & (6.12, 18.38)  & (8.27, 19.29)  & (10.48, 19.46) & (12.66, 18.98) & (14.70, 17.96) & (16.50, 16.50) \\
1/132 & (0.26, 5.01)   & (1.09, 9.79)   & (2.37, 13.46)  & (4.02, 16.10)  & (5.94, 17.81)  & (8.02, 18.70)  & (10.16, 18.87) & (12.27, 18.41) & (14.25, 17.42) & (16.00, 16.00) \\
1/128 & (0.25, 4.83)   & (1.05, 9.47)   & (2.30, 13.02)  & (3.90, 15.58)  & (5.75, 17.25)  & (7.76, 18.12)  & (9.84, 18.28)  & (11.89, 17.83) & (13.81, 16.87) & (15.50, 15.50) \\
1/124 & (0.24, 4.65)   & (1.02, 9.14)   & (2.22, 12.59)  & (3.77, 15.07)  & (5.56, 16.69)  & (7.51, 17.53)  & (9.52, 17.69)  & (11.50, 17.26) & (13.36, 16.33) & (15.00, 15.00) \\
\midrule
1/120 & (0.24, 4.46)   & (0.98, 8.82)   & (2.14, 12.16)  & (3.64, 14.56)  & (5.38, 16.12)  & (7.26, 16.94)  & (9.20, 17.10)  & (11.12, 16.68) & (12.92, 15.79) & (14.50, 14.50) \\
1/116 & (0.23, 4.28)   & (0.94, 8.50)   & (2.07, 11.72)  & (3.51, 14.05)  & (5.19, 15.56)  & (7.01, 16.35)  & (8.89, 16.50)  & (10.74, 16.10) & (12.47, 15.24) & (14.00, 14.00) \\
1/112 & (0.22, 4.10)   & (0.91, 8.17)   & (1.99, 11.29)  & (3.38, 13.54)  & (5.00, 15.00)  & (6.76, 15.76)  & (8.57, 15.91)  & (10.35, 15.53) & (12.02, 14.70) & (13.50, 13.50) \\
1/108 & (0.21, 3.92)   & (0.87, 7.85)   & (1.92, 10.85)  & (3.26, 13.02)  & (4.81, 14.44)  & (6.50, 15.18)  & (8.25, 15.32)  & (9.97, 14.95)  & (11.58, 14.15) & (13.00, 13.00) \\
1/104 & (0.20, 3.74)   & (0.84, 7.52)   & (1.84, 10.42)  & (3.13, 12.51)  & (4.62, 13.88)  & (6.25, 14.59)  & (7.93, 14.73)  & (9.58, 14.38)  & (11.13, 13.61) & (12.50, 12.50) \\
\midrule
1/100 & (0.19, 3.56)   & (0.80, 7.20)   & (1.76, 9.99)   & (3.00, 12.00)  & (4.44, 13.31)  & (6.00, 14.00)  & (7.61, 14.14)  & (9.20, 13.80)  & (10.69, 13.06) & (12.00, 12.00) \\
1/96 & (0.18, 3.38)   & (0.76, 6.88)   & (1.69, 9.55)   & (2.87, 11.49)  & (4.25, 12.75)  & (5.75, 13.41)  & (7.29, 13.55)  & (8.82, 13.22)  & (10.24, 12.52) & (11.50, 11.50) \\
1/92 & (0.17, 3.20)   & (0.73, 6.55)   & (1.61, 9.12)   & (2.74, 10.98)  & (4.06, 12.19)  & (5.50, 12.82)  & (6.98, 12.95)  & (8.43, 12.65)  & (9.80, 11.97)  & (11.00, 11.00) \\
1/88 & (0.16, 3.02)   & (0.69, 6.23)   & (1.53, 8.69)   & (2.62, 10.46)  & (3.88, 11.62)  & (5.24, 12.24)  & (6.66, 12.36)  & (8.05, 12.07)  & (9.35, 11.43)  & (10.50, 10.50) \\
1/84 & (0.15, 2.84)   & (0.66, 5.90)   & (1.46, 8.25)   & (2.49, 9.95)   & (3.69, 11.06)  & (4.99, 11.65)  & (6.34, 11.77)  & (7.66, 11.50)  & (8.91, 10.88)  & (10.00, 10.00) \\
\midrule
1/80 & (0.14, 2.66)   & (0.62, 5.58)   & (1.38, 7.82)   & (2.36, 9.44)   & (3.50, 10.50)  & (4.74, 11.06)  & (6.02, 11.18)  & (7.28, 10.92)  & (8.46, 10.34)  & (9.50, 9.50)   \\
1/76 & (0.13, 2.48)   & (0.58, 5.26)   & (1.30, 7.39)   & (2.23, 8.93)   & (3.31, 9.94)   & (4.49, 10.47)  & (5.70, 10.59)  & (6.90, 10.34)  & (8.01, 9.80)   & (9.00, 9.00)   \\
1/72 & (0.12, 2.30)   & (0.55, 4.93)   & (1.23, 6.95)   & (2.10, 8.42)   & (3.12, 9.38)   & (4.24, 9.88)   & (5.38, 10.00)  & (6.51, 9.77)   & (7.57, 9.25)   & (8.50, 8.50)   \\
1/68 & (0.11, 2.12)   & (0.51, 4.61)   & (1.15, 6.52)   & (1.98, 7.90)   & (2.94, 8.81)   & (3.98, 9.30)   & (5.06, 9.41)   & (6.13, 9.19)   & (7.12, 8.71)   & (8.00, 8.00)   \\
1/64 & (0.10, 1.94)   & (0.48, 4.28)   & (1.07, 6.09)   & (1.85, 7.39)   & (2.75, 8.25)   & (3.73, 8.71)   & (4.75, 8.81)   & (5.74, 8.62)   & (6.68, 8.16)   & (7.50, 7.50)   \\
\midrule
1/60 & (0.09, 1.76)   & (0.44, 3.96)   & (1.00, 5.65)   & (1.72, 6.88)   & (2.56, 7.69)   & (3.48, 8.12)   & (4.43, 8.22)   & (5.36, 8.04)   & (6.23, 7.62)   & (7.00, 7.00)   \\
1/56 & (0.08, 1.58)   & (0.40, 3.64)   & (0.92, 5.22)   & (1.59, 6.37)   & (2.38, 7.12)   & (3.23, 7.53)   & (4.11, 7.63)   & (4.98, 7.46)   & (5.79, 7.07)   & (6.50, 6.50)   \\
1/52 & (0.07, 1.40)   & (0.37, 3.31)   & (0.84, 4.79)   & (1.46, 5.86)   & (2.19, 6.56)   & (2.98, 6.94)   & (3.79, 7.04)   & (4.59, 6.89)   & (5.34, 6.53)   & (6.00, 6.00)   \\
1/48 & (0.06, 1.22)   & (0.33, 2.99)   & (0.77, 4.35)   & (1.34, 5.34)   & (2.00, 6.00)   & (2.72, 6.36)   & (3.47, 6.45)   & (4.21, 6.31)   & (4.90, 5.98)   & (5.50, 5.50)   \\
1/44 & (0.05, 1.04)   & (0.30, 2.66)   & (0.69, 3.92)   & (1.21, 4.83)   & (1.81, 5.44)   & (2.47, 5.77)   & (3.15, 5.86)   & (3.82, 5.74)   & (4.45, 5.44)   & (5.00, 5.00)   \\
\midrule
1/40 & (0.04, 0.85)   & (0.26, 2.34)   & (0.61, 3.48)   & (1.08, 4.32)   & (1.62, 4.88)   & (2.22, 5.18)   & (2.83, 5.26)   & (3.44, 5.16)   & (4.00, 4.90)   & (4.50, 4.50)   \\
1/36 & (0.04, 0.67)   & (0.22, 2.02)   & (0.54, 3.05)   & (0.95, 3.81)   & (1.44, 4.31)   & (1.97, 4.59)   & (2.52, 4.67)   & (3.06, 4.58)   & (3.56, 4.35)   & (4.00, 4.00)   \\
1/32 & (0.03, 0.49)   & (0.19, 1.69)   & (0.46, 2.62)   & (0.82, 3.30)   & (1.25, 3.75)   & (1.72, 4.00)   & (2.20, 4.08)   & (2.67, 4.01)   & (3.11, 3.81)   & (3.50, 3.50)   \\
1/28 & (0.02, 0.31)   & (0.15, 1.37)   & (0.39, 2.18)   & (0.70, 2.78)   & (1.06, 3.19)   & (1.46, 3.42)   & (1.88, 3.49)   & (2.29, 3.43)   & (2.67, 3.26)   & (3.00, 3.00)   \\
1/24 & (0.01, 0.13)   & (0.12, 1.04)   & (0.31, 1.75)   & (0.57, 2.27)   & (0.88, 2.62)   & (1.21, 2.83)   & (1.56, 2.90)   & (1.90, 2.86)   & (2.22, 2.72)   & (2.50, 2.50)   \\
\midrule
1/20 & N/A            & (0.08, 0.72)   & (0.23, 1.32)   & (0.44, 1.76)   & (0.69, 2.06)   & (0.96, 2.24)   & (1.24, 2.31)   & (1.52, 2.28)   & (1.78, 2.17)   & (2.00, 2.00)   \\
1/16 & N/A            & (0.04, 0.40)   & (0.16, 0.88)   & (0.31, 1.25)   & (0.50, 1.50)   & (0.71, 1.65)   & (0.92, 1.72)   & (1.14, 1.70)   & (1.33, 1.63)   & (1.50, 1.50)   \\
1/12 & N/A            & (0.01, 0.07)   & (0.08, 0.45)   & (0.18, 0.74)   & (0.31, 0.94)   & (0.46, 1.06)   & (0.61, 1.12)   & (0.75, 1.13)   & (0.89, 1.08)   & \textbf{(1.00, 1.00)}   \\
1/8 & N/A            & N/A            & (0.00, 0.02)   & (0.06, 0.22)   & (0.12, 0.38)   & (0.20, 0.48)   & (0.29, 0.53)   & (0.37, 0.55)   & (0.44, 0.54)   & \textbf{(0.50, 0.50)}   \\
1/4 & N/A            & N/A            & N/A            & N/A            & N/A            & N/A            & N/A            & N/A            & N/A            & \textbf{(0.00, 0.00)}   \\
\bottomrule
\end{tabular}
\caption{Parameters $ ( \widetilde{ \successes } , \widetilde{ \failures } ) $, rounded to two decimal places, of the Beta distribution that achieve the
indicated values of mean (columns) and variance (rows). For mean greater than $ 0.5 $, one can use the property of the Beta distribution that swapping the
parameters gives the complementary mean.} \label{table:beta}
\end{sidewaystable}

\autoref{table:beta} illustrates values of the parameters of the Beta distribution that correspond to particular pairs of mean and variance. Three
well-known priors are highlighted in bold: the \emph{Bayes prior} $ \left( \widetilde{ \successes } , \widetilde{ \failures } \right) = ( 1 , 1 ) $  with variance $ 1 / 12 $, the
\emph{Jeffreys prior} $ \left( \widetilde{ \successes } , \widetilde{ \failures } \right) = ( 1 / 2 , 1 / 2 ) $  with variance $ 1 / 8 $, and the \emph{Haldane prior} $ \left(
\widetilde{ \successes } , \widetilde{ \failures } \right) = ( 0 , 0 ) $ with variance $ 1 / 4 $. All three are considered non-informative, although it may be argued that they
provide different amount of information, e.g. the Bayes prior provides information that it is possible to observe each response, while the Haldane prior
includes the setting that only one of the responses is possible.


\section{Equivalence of Reward Definitions}
\label{section:reward_equivalence}

Throughout this section, we assume that information state $ \veci $ is constant and drop the explicit dependence on it. A particular design prescribes $
a_{ \vecx } \in \setA := \{ 1 , 2 , 3 \} $ for every state $ \vecx = ( s_{ C } , f_{ C } , s_{ D } , f_{ D } ) $. We first consider that the design is
fixed.

Using the classic reward definition, the Poisson equation \citep[cf.][]{MakowskiShwartz2002handbook} is: for all states satisfying $ s_{ C } + f_{ C } + s_{
D } + f_{ D } = T $,
\begin{align}
F_{ T } ( s_{ C } , f_{ C } , s_{ D } , f_{ D } ) = 0 \label{F_T}
\end{align}
and for all $ t \in \setT $ and for all states satisfying $ s_{ C } + f_{ C } + s_{ D } + f_{ D } = t $,
\begin{align}
F_{ t } ( s_{ C } , f_{ C } , s_{ D } , f_{ D } ) = p_{ C }^{ a_{ ( s_{ C } , f_{ C } , s_{ D } , f_{ D } ) } } \left[
q_{ C , ( s_{ C } , f_{ C } , s_{ D } , f_{ D } ) , 1 } \left( 1 + F_{ t + 1 } ( s_{ C } + 1 , f_{ C } , s_{ D } , f_{ D } ) \right) \right. \notag \\
\left. + q_{ C , ( s_{ C } , f_{ C } , s_{ D } , f_{ D } ) , 0 } \left( F_{ t + 1 } ( s_{ C } , f_{ C } + 1 , s_{ D } , f_{ D } ) \right) \right] \notag \\
+ p_{ D }^{ a_{ ( s_{ C } , f_{ C } , s_{ D } , f_{ D } ) } } \left[
q_{ D , ( s_{ C } , f_{ C } , s_{ D } , f_{ D } ) , 1 } \left( 1 + F_{ t + 1 } ( s_{ C } , f_{ C } , s_{ D } + 1 , f_{ D } ) \right) \right. \notag \\
\left. + q_{ D , ( s_{ C } , f_{ C } , s_{ D } , f_{ D } ) , 0 } \left( F_{ t + 1 } ( s_{ C } , f_{ C } , s_{ D } , f_{ D } + 1 ) \right) \right] \label{F_t}
\end{align}

Using the new reward definition, the Poisson equation is: for all states satisfying $ s_{ C } + f_{ C } + s_{ D } + f_{ D } = T $,
\begin{align}
G_{ T } ( s_{ C } , f_{ C } , s_{ D } , f_{ D } ) = s_{ C } + s_{ D } \label{G_T}
\end{align}
and for all $ t \in \setT $ and for all states satisfying $ s_{ C } + f_{ C } + s_{ D } + f_{ D } = t $,
\begin{align}
G_{ t } ( s_{ C } , f_{ C } , s_{ D } , f_{ D } ) = p_{ C }^{ a_{ ( s_{ C } , f_{ C } , s_{ D } , f_{ D } ) } } \left[
q_{ C , ( s_{ C } , f_{ C } , s_{ D } , f_{ D } ) , 1 } \left( G_{ t + 1 } ( s_{ C } + 1 , f_{ C } , s_{ D } , f_{ D } ) \right) \right. \notag \\
\left. + q_{ C , ( s_{ C } , f_{ C } , s_{ D } , f_{ D } ) , 0 } \left( G_{ t + 1 } ( s_{ C } , f_{ C } + 1 , s_{ D } , f_{ D } ) \right) \right] \notag \\
+ p_{ D }^{ a_{ ( s_{ C } , f_{ C } , s_{ D } , f_{ D } ) } } \left[
q_{ D , ( s_{ C } , f_{ C } , s_{ D } , f_{ D } ) , 1 } \left( G_{ t + 1 } ( s_{ C } , f_{ C } , s_{ D } + 1 , f_{ D } ) \right) \right. \notag \\
\left. + q_{ D , ( s_{ C } , f_{ C } , s_{ D } , f_{ D } ) , 0 } \left( G_{ t + 1 } ( s_{ C } , f_{ C } , s_{ D } , f_{ D } + 1 ) \right) \right] \label{G_t}
\end{align}

\begin{lemma}
The recursion \eqref{G_T}--\eqref{G_t} evaluates the design equivalently to the recursion \eqref{F_T}--\eqref{F_t}, and in particular, $ G_{ 0 } ( 0 , 0 ,
0 , 0 ) = F_{ 0 } ( 0 , 0 , 0 , 0 ) $.
\end{lemma}

\begin{proof}
We will prove a more general result that for all $ t \in \setT \cup \{ T \} $ and for all states satisfying $ s_{ C } + f_{ C } + s_{ D } + f_{ D } = t $,
\begin{align}
G_{ t } ( s_{ C } , f_{ C } , s_{ D } , f_{ D } ) = F_{ t } ( s_{ C } , f_{ C } , s_{ D } , f_{ D } ) + s_{ C } + s_{ D } .
\end{align}

From \eqref{F_T} and \eqref{G_T} it is clear that and for all states satisfying $ s_{ C } + f_{ C } + s_{ D } + f_{ D } = T $,
\begin{align}
G_{ T } ( s_{ C } , f_{ C } , s_{ D } , f_{ D } ) = F_{ T } ( s_{ C } , f_{ C } , s_{ D } , f_{ D } ) + s_{ C } + s_{ D } .
\end{align}
It remains to prove that for all $ t \in \setT $ and for all states satisfying $ s_{ C } + f_{ C } + s_{ D } + f_{ D } = t $,
\begin{align}
G_{ t } ( s_{ C } , f_{ C } , s_{ D } , f_{ D } ) = F_{ t } ( s_{ C } , f_{ C } , s_{ D } , f_{ D } ) + s_{ C } + s_{ D }
\end{align}

Using induction assumption for $ t + 1 $
\begin{align}
G_{ t + 1 } ( s_{ C } , f_{ C } , s_{ D } , f_{ D } ) = F_{ t + 1 } ( s_{ C } , f_{ C } , s_{ D } , f_{ D } ) + s_{ C } + s_{ D }
\end{align}
and plugging it into \eqref{G_t}, we have
\begin{align}
G_{ t } ( s_{ C } , f_{ C } , s_{ D } , f_{ D } ) = p_{ C }^{ a_{ ( s_{ C } , f_{ C } , s_{ D } , f_{ D } ) } } \left[
q_{ C , ( s_{ C } , f_{ C } , s_{ D } , f_{ D } ) , 1 } \left( F_{ t + 1 } ( s_{ C } + 1 , f_{ C } , s_{ D } , f_{ D } ) + s_{ C } + s_{ D } + 1 \right) \right. \notag \\
\left. + q_{ C , ( s_{ C } , f_{ C } , s_{ D } , f_{ D } ) , 0 } \left( F_{ t + 1 } ( s_{ C } , f_{ C } + 1 , s_{ D } , f_{ D } ) + s_{ C } + s_{ D } \right) \right] \notag \\
+ p_{ D }^{ a_{ ( s_{ C } , f_{ C } , s_{ D } , f_{ D } ) } } \left[
q_{ D , ( s_{ C } , f_{ C } , s_{ D } , f_{ D } ) , 1 } \left( F_{ t + 1 } ( s_{ C } , f_{ C } , s_{ D } + 1 , f_{ D } ) + s_{ C } + s_{ D } + 1 \right) \right. \notag \\
\left. + q_{ D , ( s_{ C } , f_{ C } , s_{ D } , f_{ D } ) , 0 } \left( F_{ t + 1 } ( s_{ C } , f_{ C } , s_{ D } , f_{ D } + 1 ) + s_{ C } + s_{ D } \right) \right] \notag
\end{align}
which after repetitively using \eqref{F_t} yields
\begin{align}
G_{ t } ( s_{ C } , f_{ C } , s_{ D } , f_{ D } ) = F_{ t } ( s_{ C } , f_{ C } , s_{ D } , f_{ D } ) + p_{ C }^{ a_{ ( s_{ C } , f_{ C } , s_{ D } , f_{ D } ) } } \left[
q_{ C , ( s_{ C } , f_{ C } , s_{ D } , f_{ D } ) , 1 } \left( s_{ C } + s_{ D } \right) \right. \notag \\
\left. + q_{ C , ( s_{ C } , f_{ C } , s_{ D } , f_{ D } ) , 0 } \left( s_{ C } + s_{ D } \right) \right] \notag \\
+ p_{ D }^{ a_{ ( s_{ C } , f_{ C } , s_{ D } , f_{ D } ) } } \left[
q_{ D , ( s_{ C } , f_{ C } , s_{ D } , f_{ D } ) , 1 } \left( s_{ C } + s_{ D } \right) \right. \notag \\
\left. + q_{ D , ( s_{ C } , f_{ C } , s_{ D } , f_{ D } ) , 0 } \left( s_{ C } + s_{ D } \right) \right] \notag
\end{align}
which gives $ G_{ t } ( s_{ C } , f_{ C } , s_{ D } , f_{ D } ) = F_{ t } ( s_{ C } , f_{ C } , s_{ D } , f_{ D } ) + s_{ C } + s_{ D } $
because $ \sum_{ o } q_{ k , ( \vecx , \veci ) , o } = 1 $ for all $ k $ and $ \sum_{ k } p_{ k }^{ a } = 1 $ for all $ a $.
\end{proof}

Note that these two theorems do not hold in the model with (non-uniform) geometric discounting of future rewards, because there the order of observations
matters.

\section{Designs Performance --- Small Trial Sizes}
\label{section:performance_continued}

\autoref{fig:PS_240} and \autoref{fig:regret_240} present the performance results for small trial sizes, analogously to \autoref{fig:PS_2400} and
\autoref{fig:regret_2400}, respectively.

\begin{figure}[tbp]
\centering
    \begin{subfigure}[b]{0.48\textwidth}
        \includegraphics[trim=0pt 0pt 0pt 0pt, clip=true, width=\textwidth]{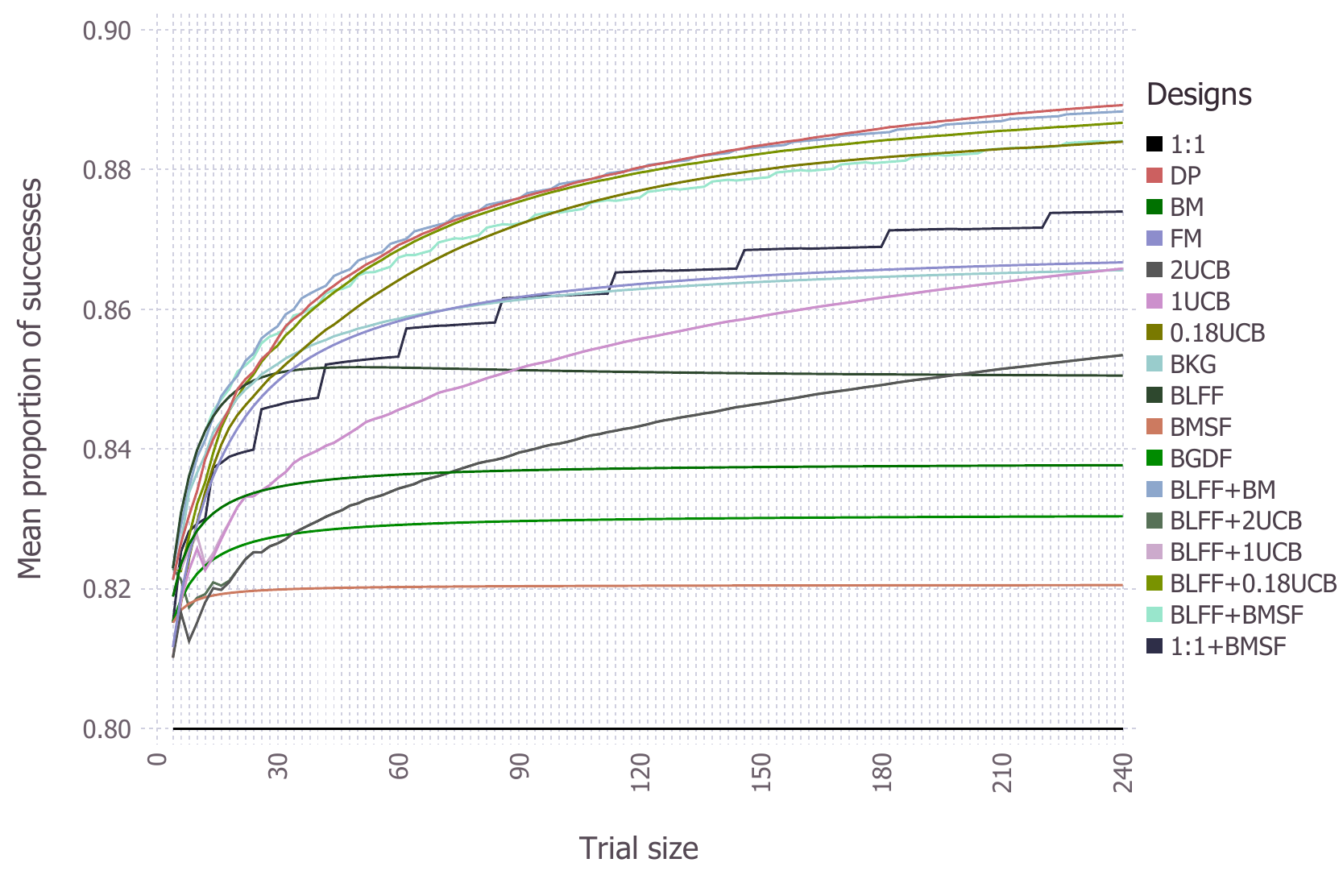}
    \end{subfigure}\hfill%
    \begin{subfigure}[b]{0.48\textwidth}
        \includegraphics[trim=0pt 0pt 0pt 0pt, clip=true, width=\textwidth]{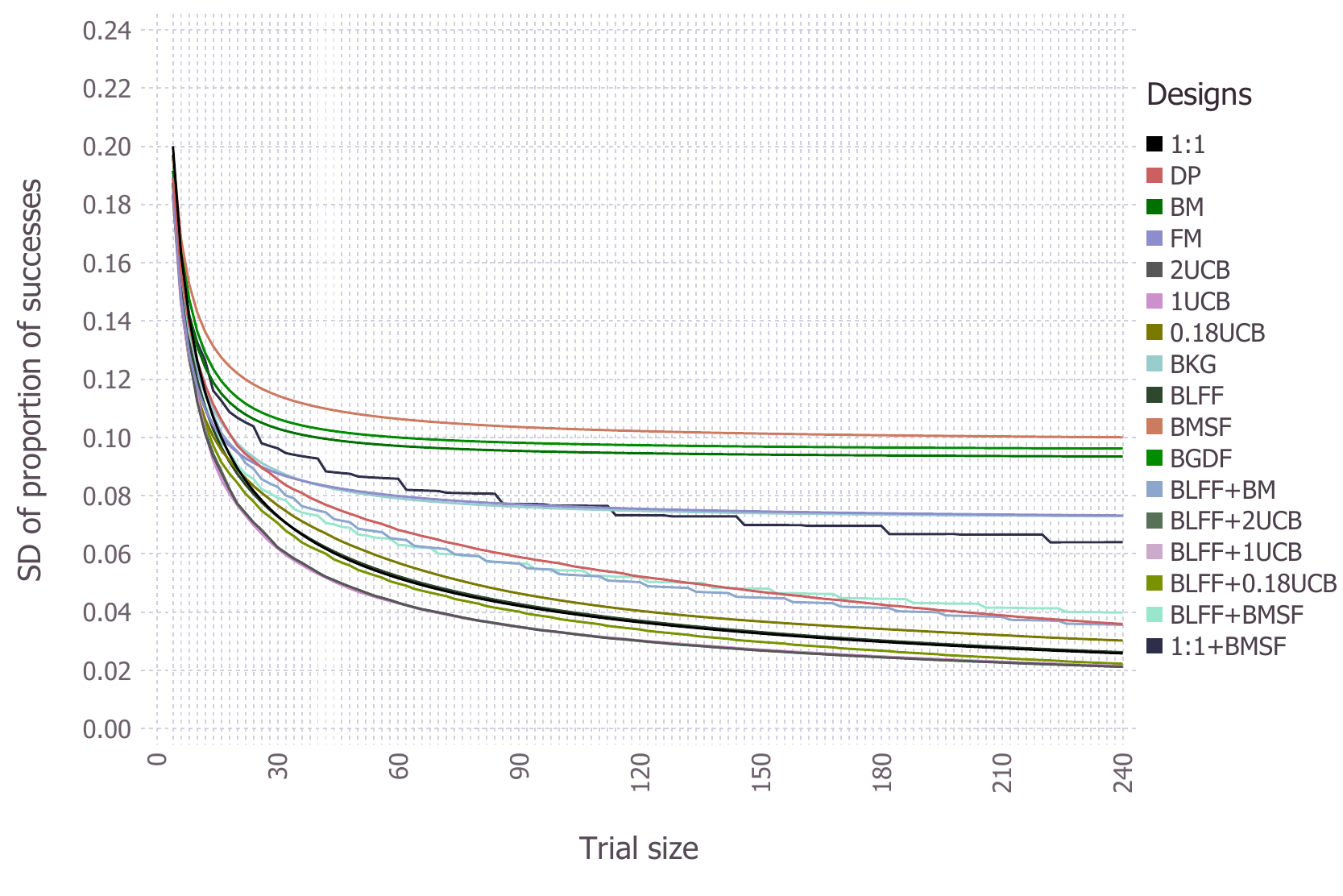}
    \end{subfigure}%

    \begin{subfigure}[b]{0.48\textwidth}
        \includegraphics[trim=0pt 0pt 0pt 0pt, clip=true, width=\textwidth]{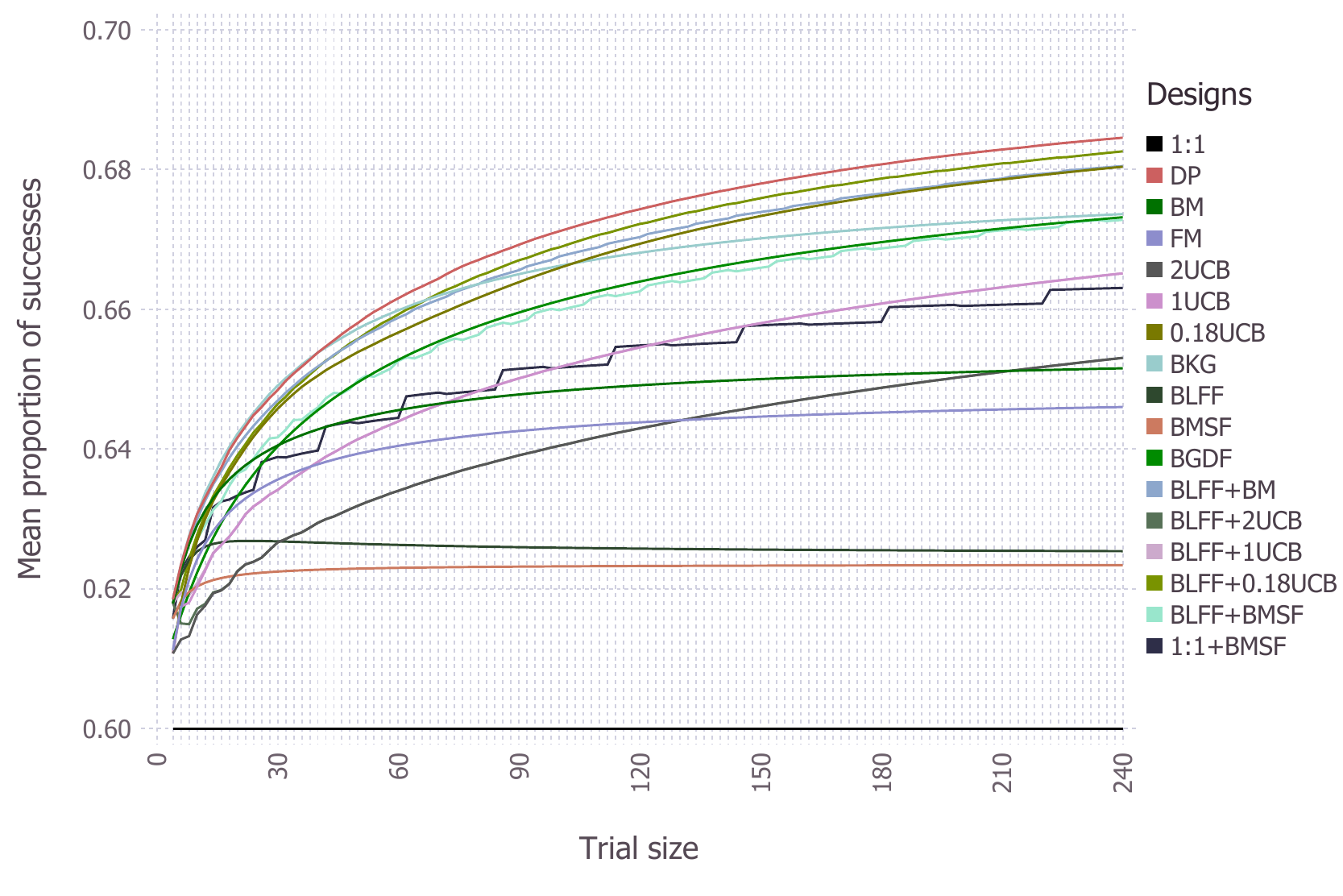}
    \end{subfigure}\hfill%
    \begin{subfigure}[b]{0.48\textwidth}
        \includegraphics[trim=0pt 0pt 0pt 0pt, clip=true, width=\textwidth]{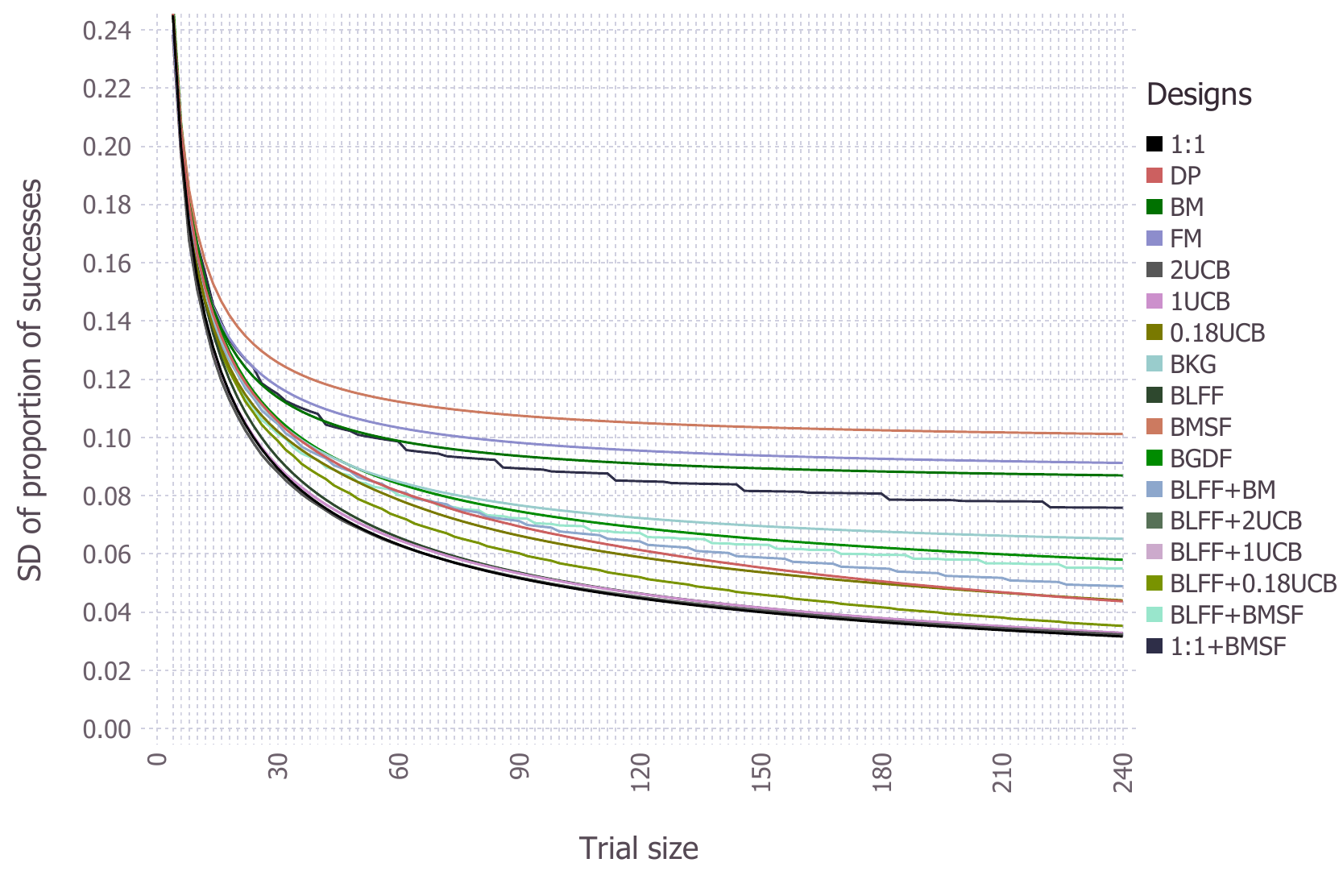}
    \end{subfigure}%

    \begin{subfigure}[b]{0.48\textwidth}
        \includegraphics[trim=0pt 0pt 0pt 0pt, clip=true, width=\textwidth]{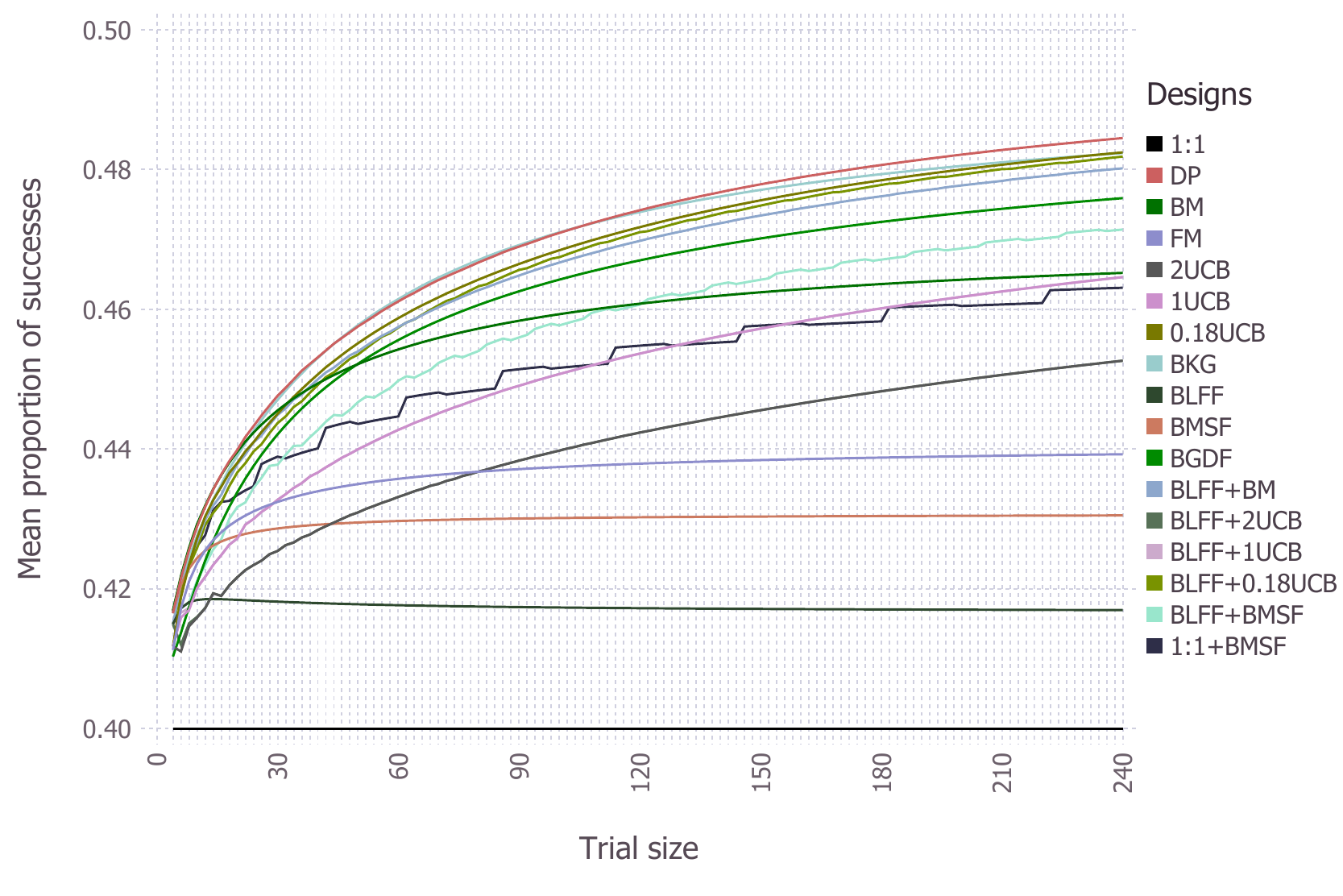}
    \end{subfigure}\hfill%
    \begin{subfigure}[b]{0.48\textwidth}
        \includegraphics[trim=0pt 0pt 0pt 0pt, clip=true, width=\textwidth]{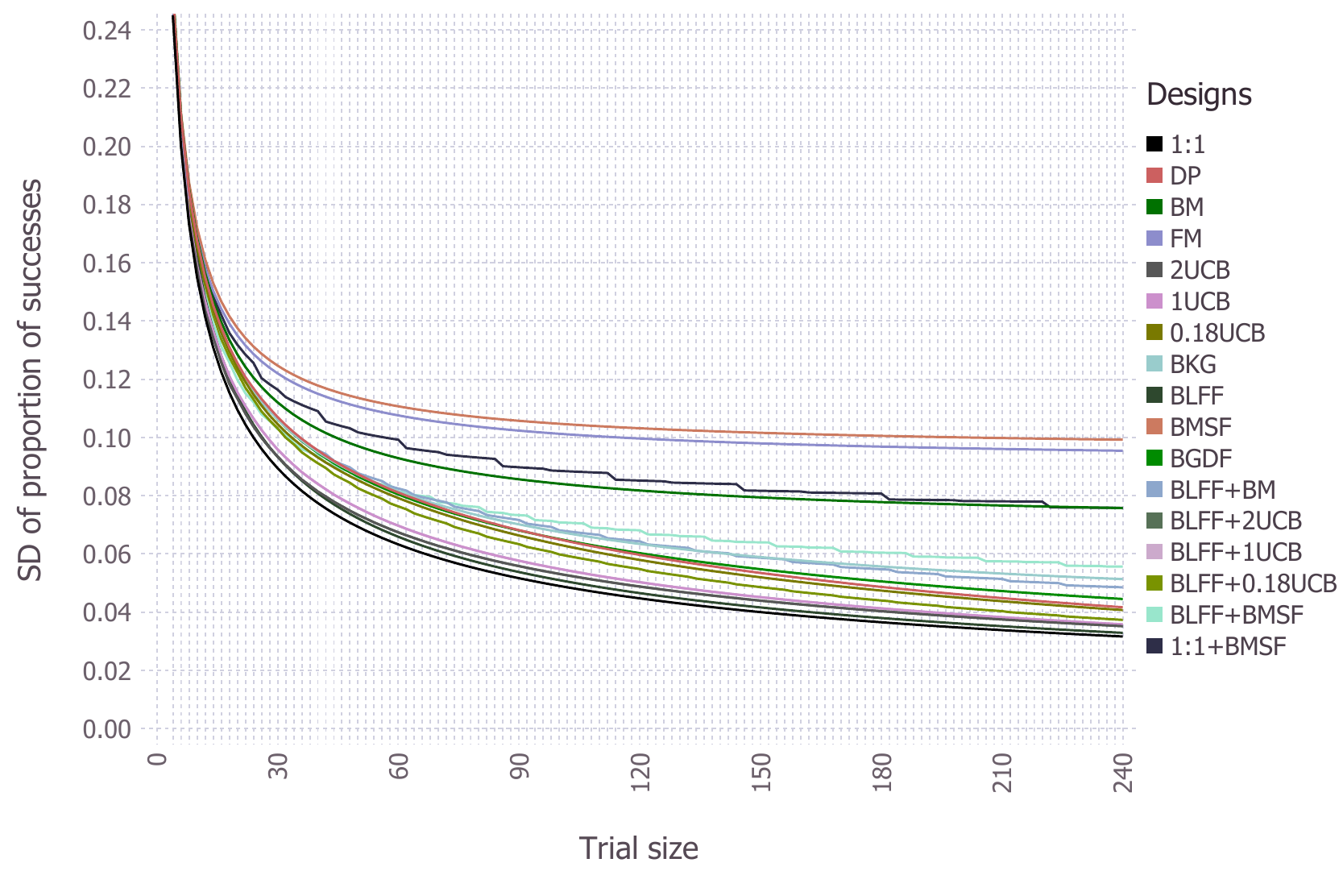}
    \end{subfigure}%

    \begin{subfigure}[b]{0.48\textwidth}
        \includegraphics[trim=0pt 0pt 0pt 0pt, clip=true, width=\textwidth]{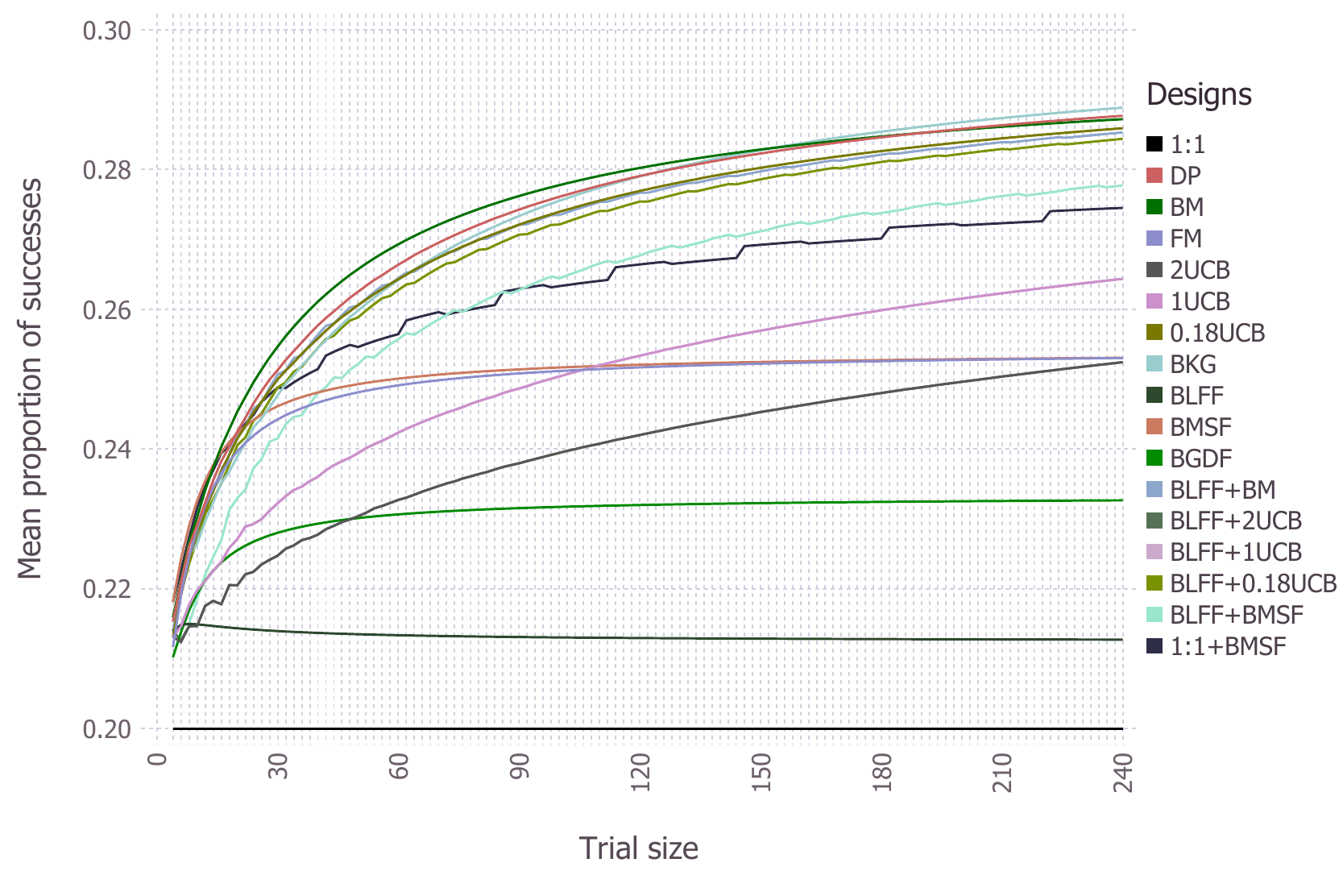}
    \end{subfigure}\hfill%
    \begin{subfigure}[b]{0.48\textwidth}
        \includegraphics[trim=0pt 0pt 0pt 0pt, clip=true, width=\textwidth]{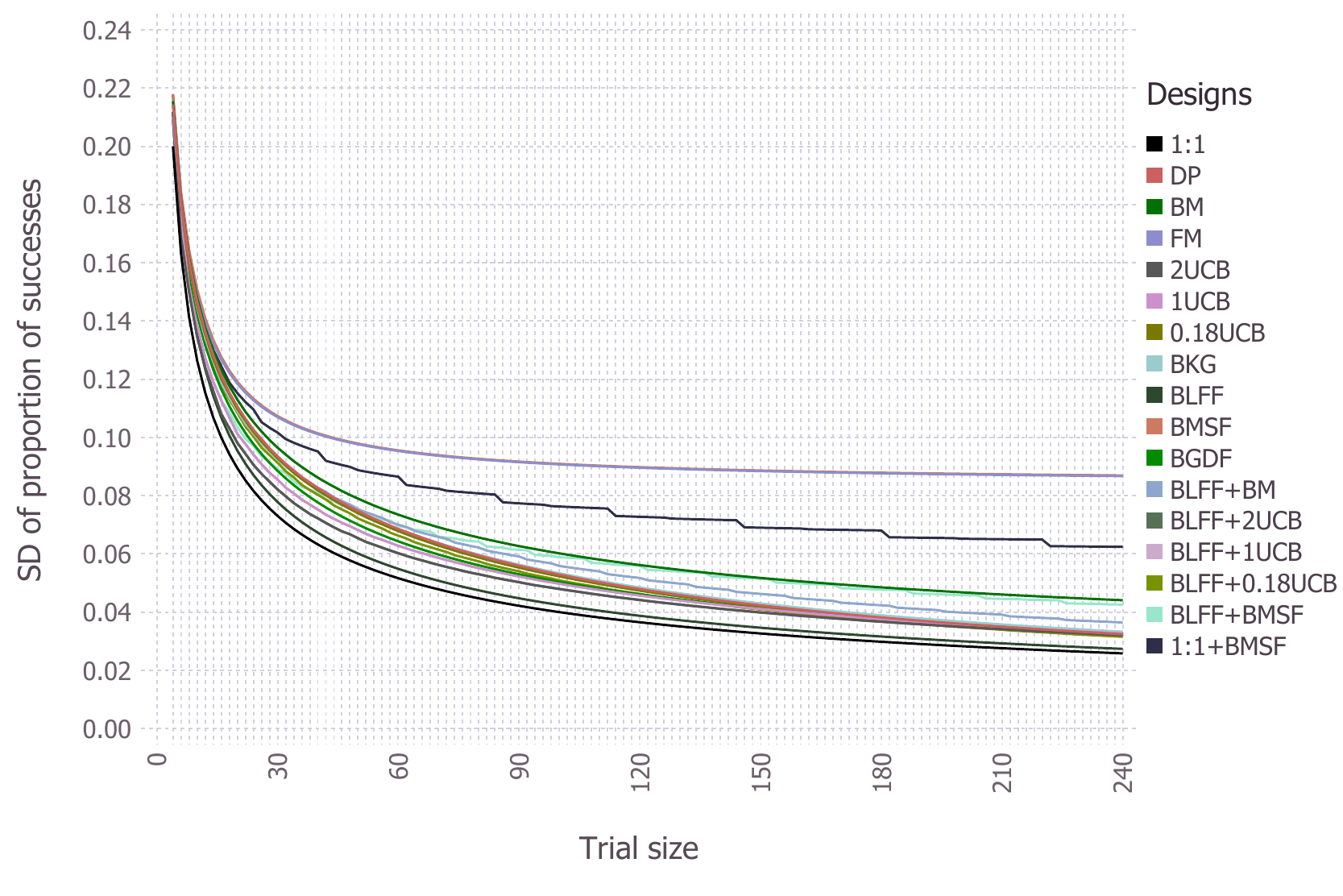}
    \end{subfigure}%
\caption{An illustration of performance (mean on the left, standard deviation on the right) in terms of the expected proportion of successes evaluated for
deterministic designs over a range of small trial sizes, for $ ( \theta_{ C } , \theta_{ D } ) = ( 0.7 , 0.9 ) $ in the first row, $ ( 0.5 , 0.7 ) $ in the
second row, $ ( 0.3 , 0.5 ) $ in the third row, $ ( 0.1 , 0.3 ) $ in the fourth row.}\label{fig:PS_240}
\end{figure}

\begin{figure}[tbp]
\centering
    \begin{subfigure}[b]{0.48\textwidth}
        \includegraphics[trim=0pt 0pt 0pt 0pt, clip=true, width=\textwidth]{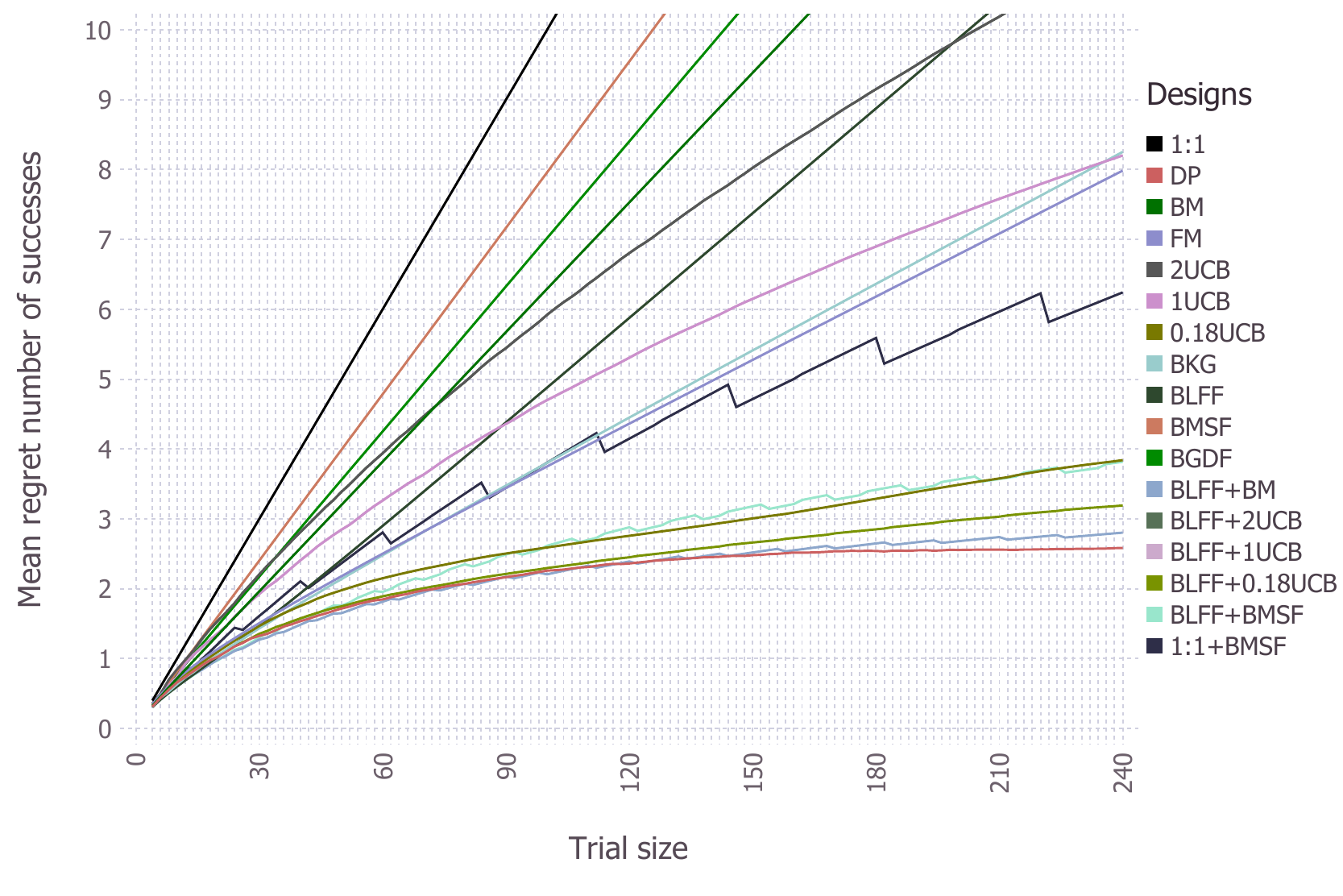}
    \end{subfigure}\hfill%
    \begin{subfigure}[b]{0.48\textwidth}
        \includegraphics[trim=0pt 0pt 0pt 0pt, clip=true, width=\textwidth]{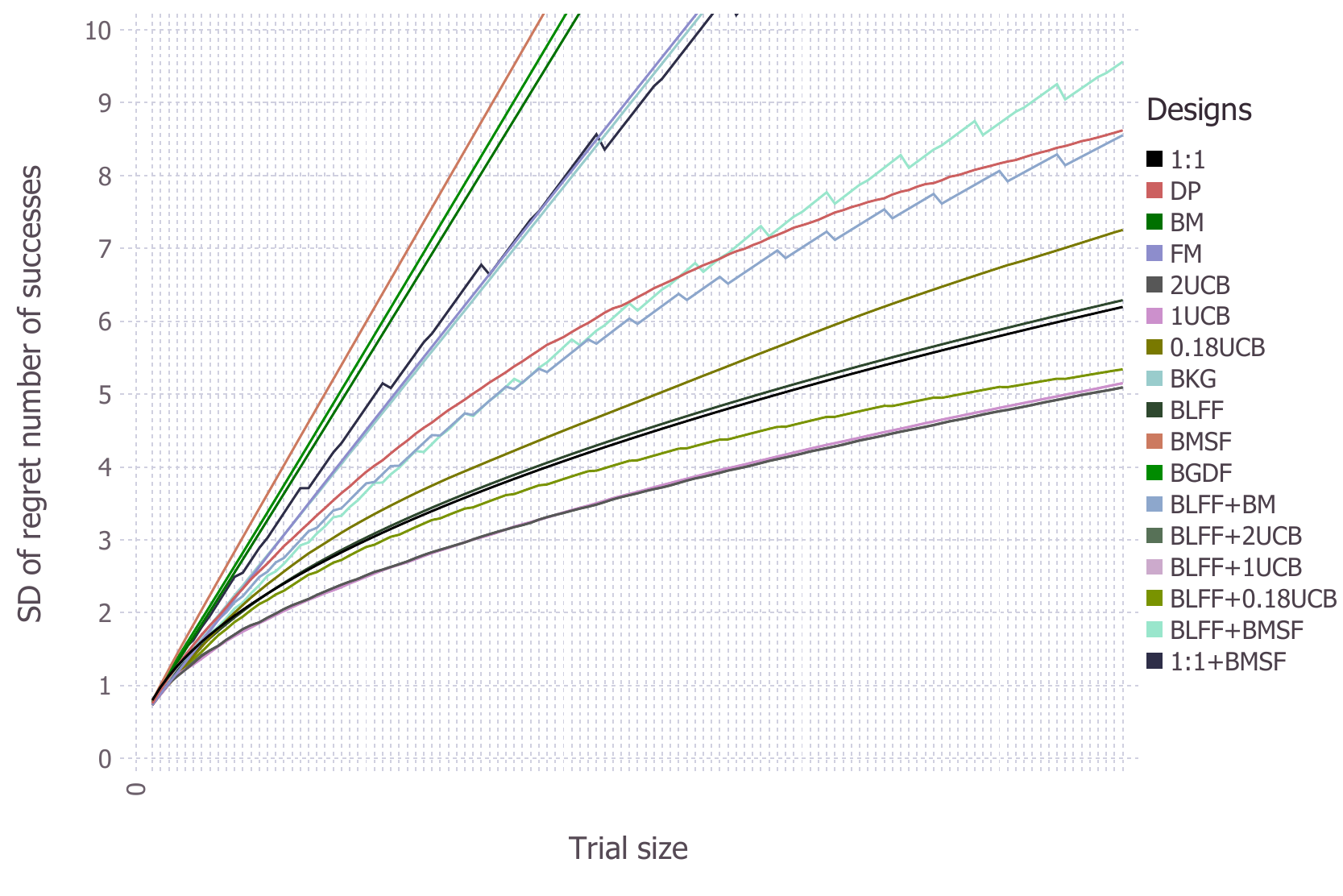}
    \end{subfigure}%

    \begin{subfigure}[b]{0.48\textwidth}
        \includegraphics[trim=0pt 0pt 0pt 0pt, clip=true, width=\textwidth]{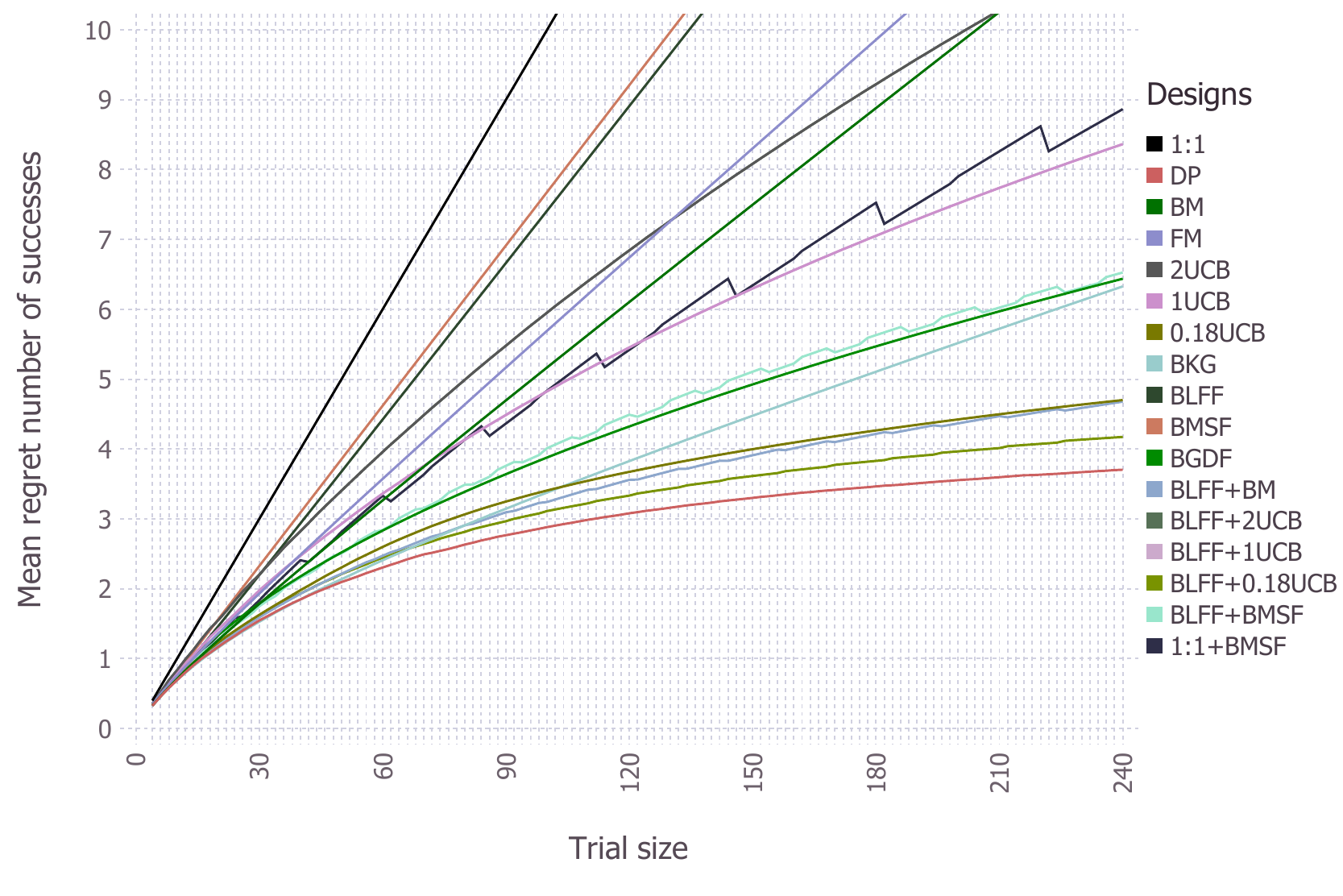}
    \end{subfigure}\hfill%
    \begin{subfigure}[b]{0.48\textwidth}
        \includegraphics[trim=0pt 0pt 0pt 0pt, clip=true, width=\textwidth]{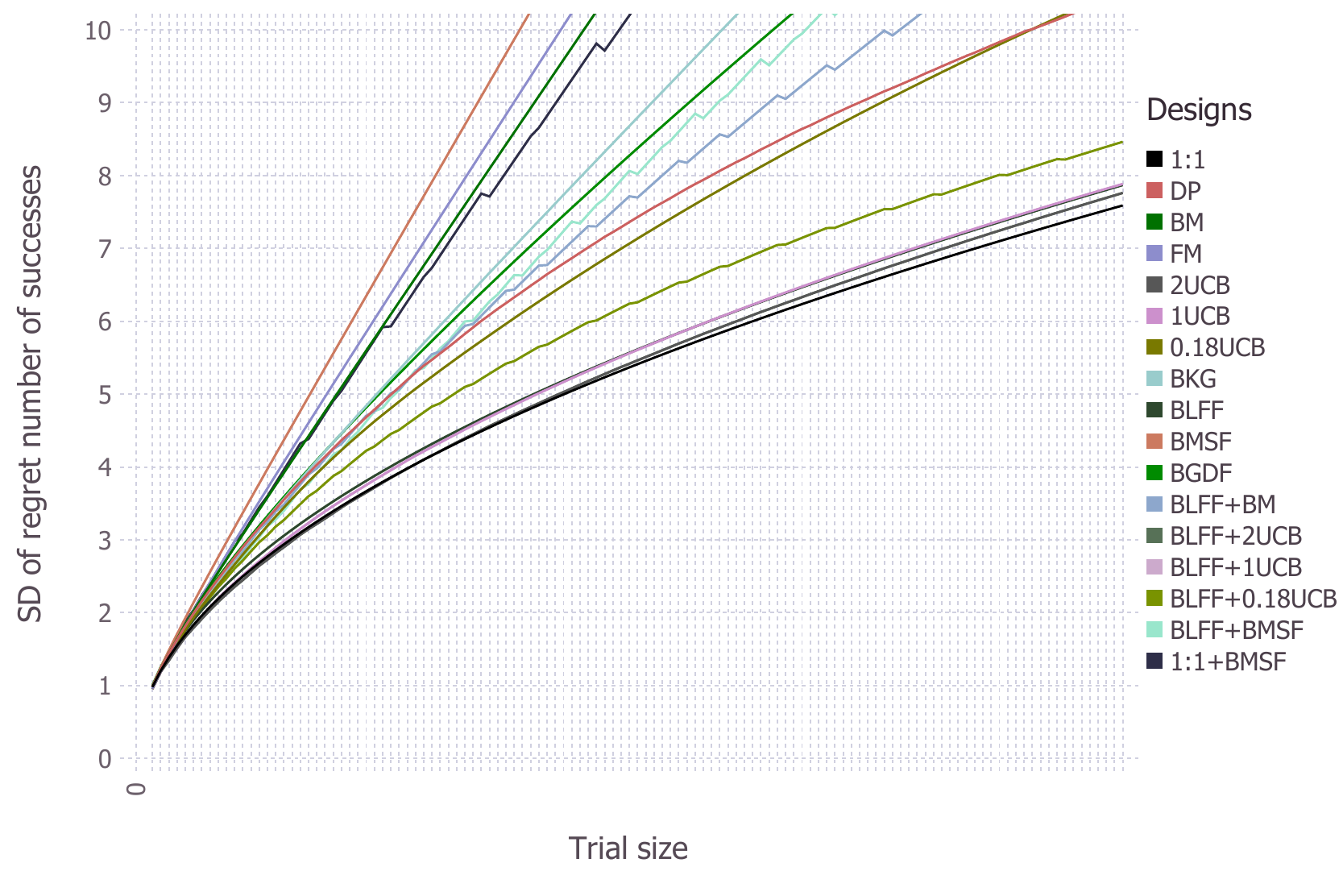}
    \end{subfigure}%

    \begin{subfigure}[b]{0.48\textwidth}
        \includegraphics[trim=0pt 0pt 0pt 0pt, clip=true, width=\textwidth]{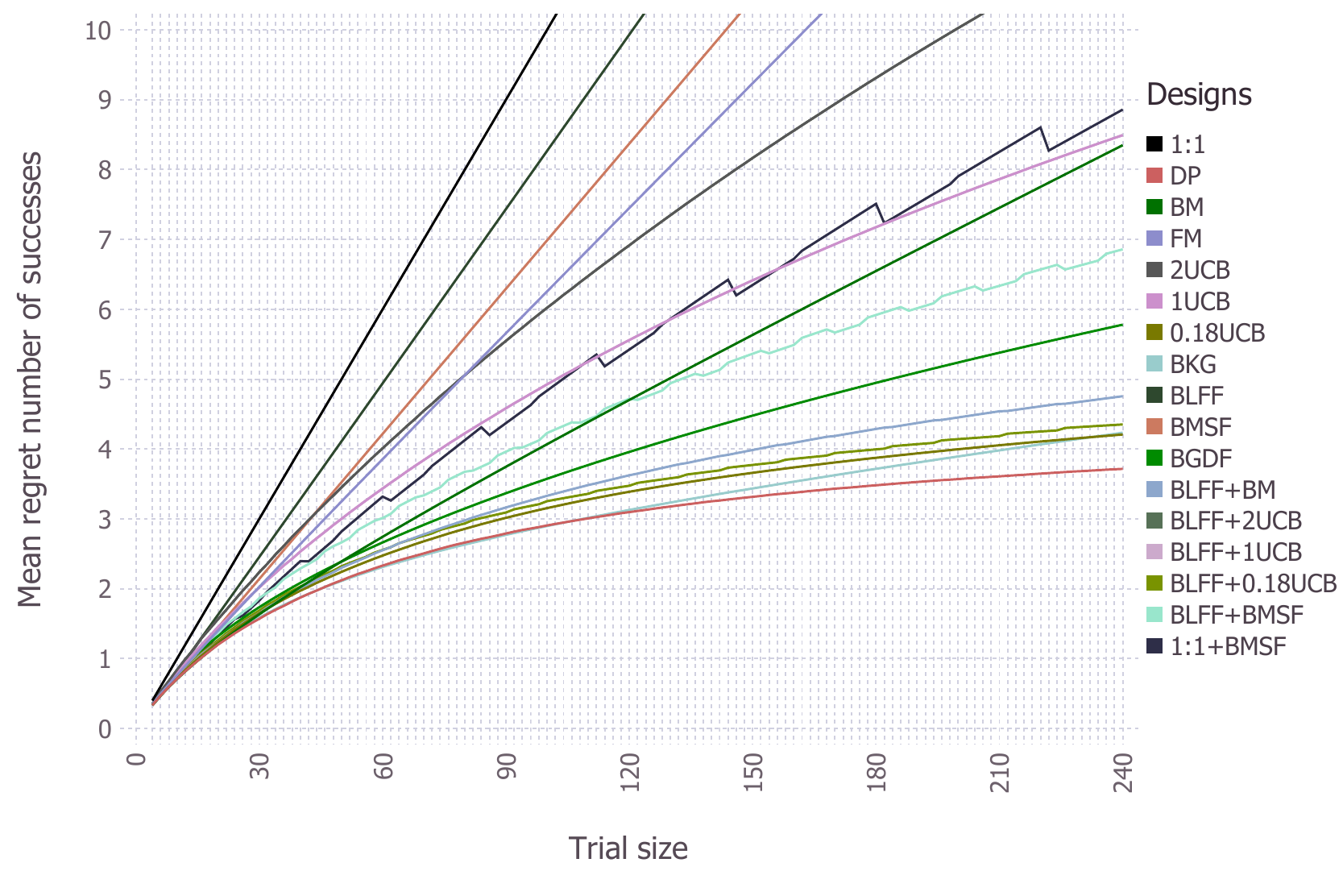}
    \end{subfigure}\hfill%
    \begin{subfigure}[b]{0.48\textwidth}
        \includegraphics[trim=0pt 0pt 0pt 0pt, clip=true, width=\textwidth]{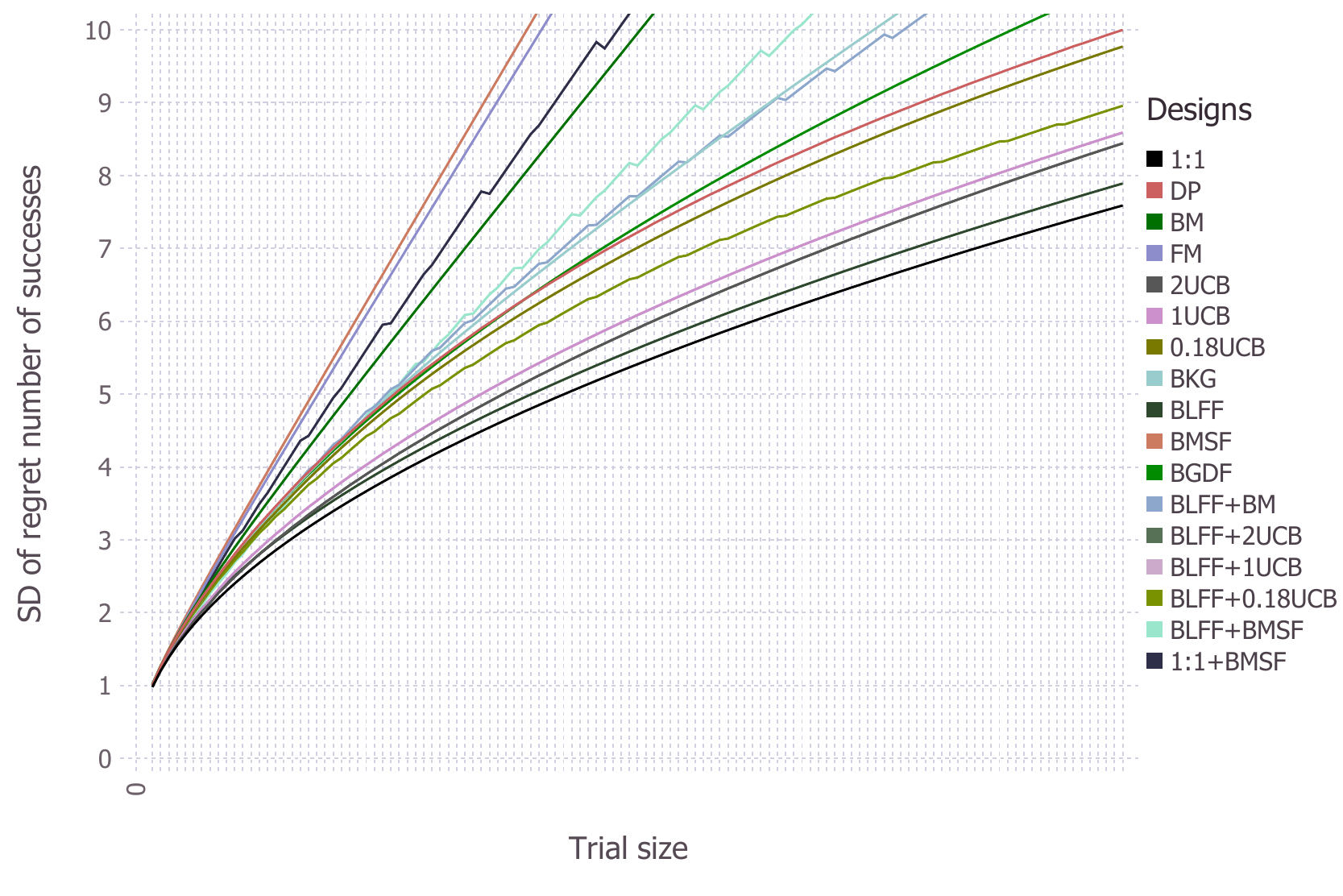}
    \end{subfigure}%

    \begin{subfigure}[b]{0.48\textwidth}
        \includegraphics[trim=0pt 0pt 0pt 0pt, clip=true, width=\textwidth]{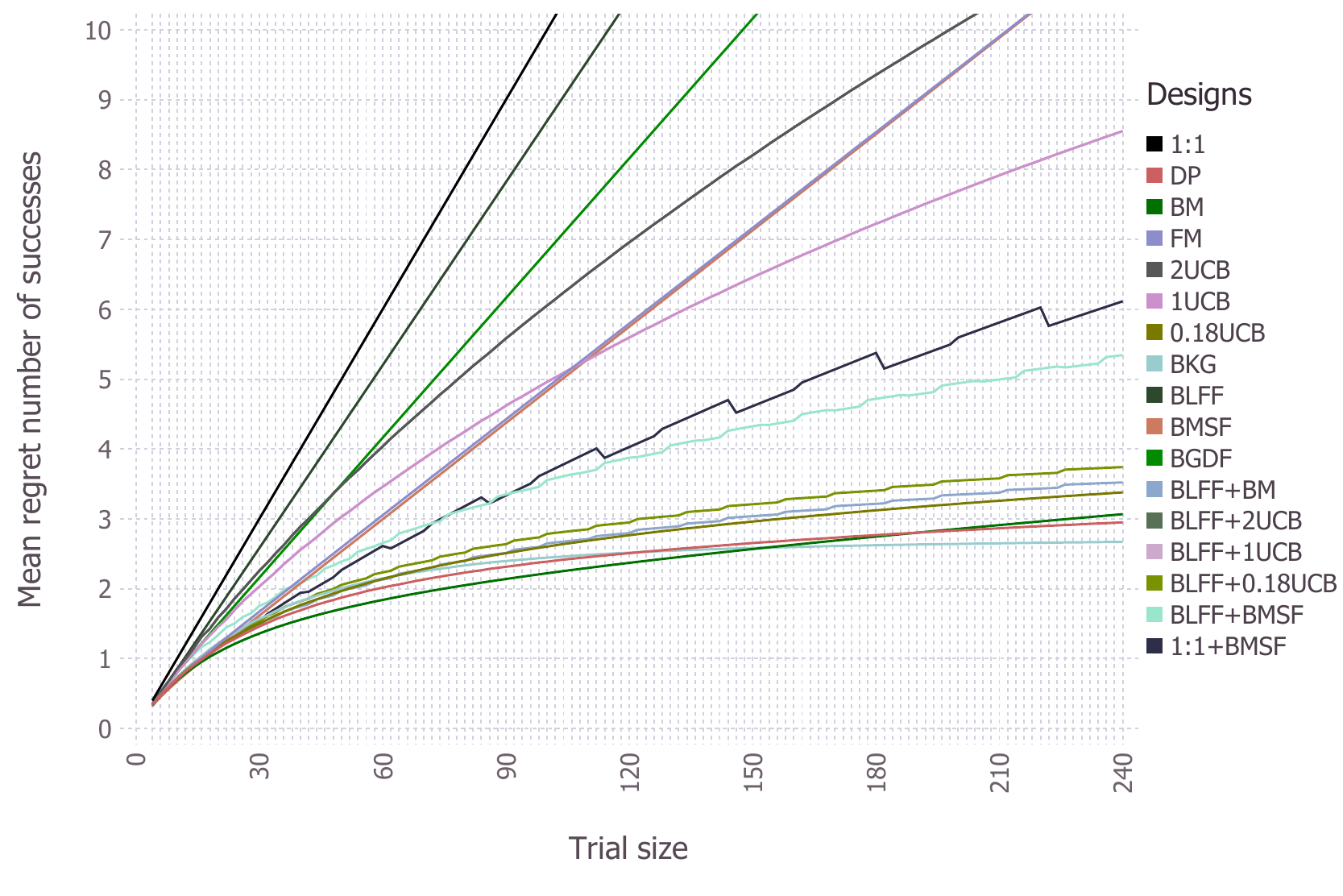}
    \end{subfigure}\hfill%
    \begin{subfigure}[b]{0.48\textwidth}
        \includegraphics[trim=0pt 0pt 0pt 0pt, clip=true, width=\textwidth]{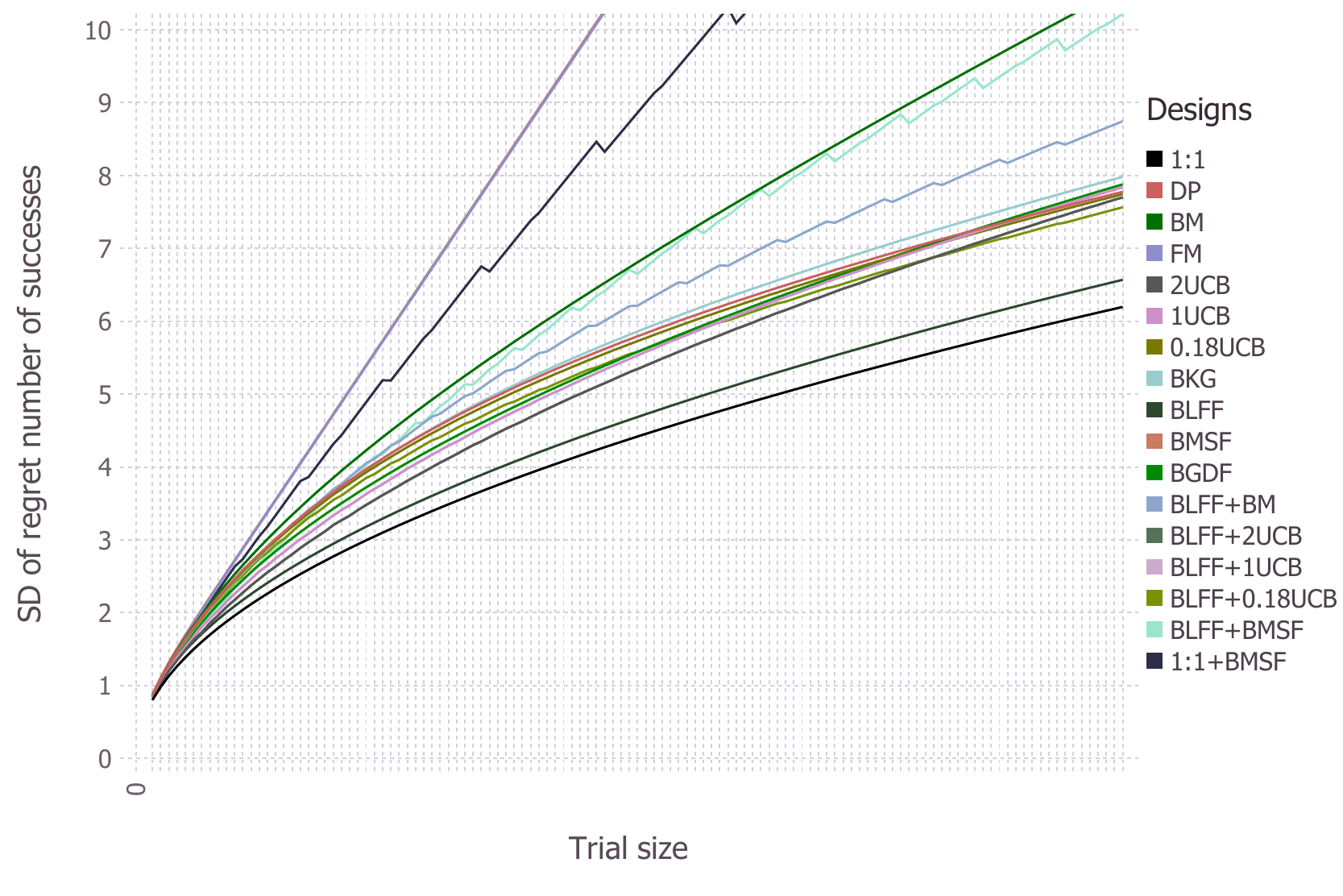}
    \end{subfigure}%
\caption{An illustration of performance (mean on the left, standard deviation on the right) in terms of the proportion of successes evaluated for
deterministic designs over a range of small trial sizes, for $ ( \theta_{ C } , \theta_{ D } ) = ( 0.7 , 0.9 ) $ in the first row, $ ( 0.5 , 0.7 ) $ in the
second row, $ ( 0.3 , 0.5 ) $ in the third row, $ ( 0.1 , 0.3 ) $ in the fourth row.}\label{fig:regret_240}
\end{figure}

\section{Designs Performance --- Tables}
\label{section:performance_tables}

In \autoref{table:performance79}, \autoref{table:performance57}, \autoref{table:performance35}, \autoref{table:performance13} we present values of the mean
regret number of successes (rounded to $ 3 $ significant digits) for trial sizes $ 60 : 60 : 1200 $, which is a subset of those presented in the figures in
\autoref{section:performance} and \autoref{section:performance_continued}. In bold font are all the values which are equal or better than that of
\design{DP} for each trial size. We further highlight with dark grey background (light grey background) the values that are lower than or equal to $ 10\% $
(greater than $ 10\% $ and lower than or equal to $ 20 \% $) of the difference between \design{DP} and \design{2UCB}.

\begin{sidewaystable}[tbp]
\footnotesize
\setlength{\tabcolsep}{5pt}
\centering
\begin{tabular}{l|cccccccccccccccccccc}
\toprule
Design $ \backslash $ Trial size & 60 & 120 & 180 & 240 & 300 & 360 & 420 & 480 & 540 & 600 & 660 & 720 & 780 & 840 & 900 & 960 & 1020 & 1080 & 1140 & 1200 \\
\midrule
1:1 & \myboxyyyyy{6.0} & \myboxyyyyy{12.0} & \myboxyyyyy{18.0} & \myboxyyyyy{24.0} & \myboxyyyyy{30.0} & \myboxyyyyy{36.0} & \myboxyyyyy{42.0} & \myboxyyyyy{48.0} & \myboxyyyyy{54.0} & \myboxyyyyy{60.0} & \myboxyyyyy{66.0} & \myboxyyyyy{72.0} & \myboxyyyyy{78.0} & \myboxyyyyy{84.0} & \myboxyyyyy{90.0} & \myboxyyyyy{96.0} & \myboxyyyyy{102} & \myboxyyyyy{108} & \myboxyyyyy{114} & \myboxyyyyy{120} \\
DP & \myboxx{1.85} & \myboxx{2.36} & \myboxx{2.54} & \myboxx{2.58} & \myboxx{2.57} & \myboxx{2.57} & \myboxx{2.58} & \myboxx{2.61} & \myboxx{2.62} & \myboxx{2.63} & \myboxx{2.64} & \myboxx{2.63} & \myboxx{2.62} & \myboxx{2.61} & \myboxx{2.6} & \myboxx{2.58} & \myboxx{2.57} & \myboxx{2.55} & \myboxx{2.55} & \myboxx{2.54} \\
BM & \myboxyyyyy{3.82} & \myboxyyyyy{7.53} & \myboxyyyyy{11.2} & \myboxyyyyy{15.0} & \myboxyyyyy{18.7} & \myboxyyyyy{22.4} & \myboxyyyyy{26.1} & \myboxyyyyy{29.8} & \myboxyyyyy{33.5} & \myboxyyyyy{37.2} & \myboxyyyyy{40.9} & \myboxyyyyy{44.6} & \myboxyyyyy{48.4} & \myboxyyyyy{52.1} & \myboxyyyyy{55.8} & \myboxyyyyy{59.5} & \myboxyyyyy{63.2} & \myboxyyyyy{66.9} & \myboxyyyyy{70.6} & \myboxyyyyy{74.3} \\
FM & \myboxyyyyy{2.5} & \myboxyyyyy{4.36} & \myboxyyyyy{6.18} & \myboxyyyyy{7.99} & \myboxyyyyy{9.78} & \myboxyyyyy{11.6} & \myboxyyyyy{13.4} & \myboxyyyyy{15.2} & \myboxyyyyy{16.9} & \myboxyyyyy{18.7} & \myboxyyyyy{20.5} & \myboxyyyyy{22.3} & \myboxyyyyy{24.1} & \myboxyyyyy{25.9} & \myboxyyyyy{27.6} & \myboxyyyyy{29.4} & \myboxyyyyy{31.2} & \myboxyyyyy{33.0} & \myboxyyyyy{34.8} & \myboxyyyyy{36.6} \\
2UCB & \myboxyyyyy{3.94} & \myboxyyyyy{6.81} & \myboxyyyyy{9.15} & \myboxyyyyy{11.2} & \myboxyyyyy{13.0} & \myboxyyyyy{14.6} & \myboxyyyyy{16.0} & \myboxyyyyy{17.4} & \myboxyyyyy{18.6} & \myboxyyyyy{19.8} & \myboxyyyyy{20.9} & \myboxyyyyy{21.9} & \myboxyyyyy{22.9} & \myboxyyyyy{23.8} & \myboxyyyyy{24.7} & \myboxyyyyy{25.5} & \myboxyyyyy{26.3} & \myboxyyyyy{27.0} & \myboxyyyyy{27.8} & \myboxyyyyy{28.5} \\
1UCB & \myboxyyyyy{3.26} & \myboxyyyyy{5.31} & \myboxyyyyy{6.9} & \myboxyyyyy{8.2} & \myboxyyyyy{9.33} & \myboxyyyyy{10.3} & \myboxyyyyy{11.2} & \myboxyyyyy{12.0} & \myboxyyyyy{12.7} & \myboxyyyyy{13.4} & \myboxyyyyy{14.0} & \myboxyyyyy{14.5} & \myboxyyyyy{15.1} & \myboxyyyyy{15.6} & \myboxyyyyy{16.0} & \myboxyyyyy{16.5} & \myboxyyyyy{16.9} & \myboxyyyyy{17.3} & \myboxyyyyy{17.7} & \myboxyyyyy{18.1} \\
0.18UCB & \myboxyyy{2.15} & \myboxy{2.76} & \myboxyyy{3.29} & \myboxyyy{3.84} & \myboxyyy{4.17} & \myboxyyy{4.43} & \myboxyyy{4.65} & \myboxyyy{4.86} & \myboxyyy{5.05} & \myboxyyy{5.21} & \myboxyyy{5.35} & \myboxyyy{5.46} & \myboxyyy{5.56} & \myboxyyy{5.64} & \myboxyyy{5.72} & \myboxyyy{5.79} & \myboxyyy{5.85} & \myboxyyy{5.92} & \myboxyyy{5.98} & \myboxyyy{6.04} \\
BKG & \myboxyyyyy{2.48} & \myboxyyyyy{4.45} & \myboxyyyyy{6.36} & \myboxyyyyy{8.25} & \myboxyyyyy{10.1} & \myboxyyyyy{12.0} & \myboxyyyyy{13.9} & \myboxyyyyy{15.7} & \myboxyyyyy{17.6} & \myboxyyyyy{19.4} & \myboxyyyyy{21.3} & \myboxyyyyy{23.1} & \myboxyyyyy{25.0} & \myboxyyyyy{26.8} & \myboxyyyyy{28.7} & \myboxyyyyy{30.5} & \myboxyyyyy{32.4} & \myboxyyyyy{34.2} & \myboxyyyyy{36.1} & \myboxyyyyy{37.9} \\
BLFF & \myboxyyyyy{2.9} & \myboxyyyyy{5.88} & \myboxyyyyy{8.88} & \myboxyyyyy{11.9} & \myboxyyyyy{14.9} & \myboxyyyyy{17.9} & \myboxyyyyy{20.9} & \myboxyyyyy{23.9} & \myboxyyyyy{26.9} & \myboxyyyyy{29.9} & \myboxyyyyy{32.9} & \myboxyyyyy{35.9} & \myboxyyyyy{38.9} & \myboxyyyyy{41.9} & \myboxyyyyy{44.9} & \myboxyyyyy{47.9} & \myboxyyyyy{50.9} & \myboxyyyyy{53.9} & \myboxyyyyy{56.9} & \myboxyyyyy{59.9} \\
BMSF & \myboxyyyyy{4.78} & \myboxyyyyy{9.55} & \myboxyyyyy{14.3} & \myboxyyyyy{19.1} & \myboxyyyyy{23.8} & \myboxyyyyy{28.6} & \myboxyyyyy{33.4} & \myboxyyyyy{38.1} & \myboxyyyyy{42.9} & \myboxyyyyy{47.7} & \myboxyyyyy{52.4} & \myboxyyyyy{57.2} & \myboxyyyyy{61.9} & \myboxyyyyy{66.7} & \myboxyyyyy{71.5} & \myboxyyyyy{76.2} & \myboxyyyyy{81.0} & \myboxyyyyy{85.8} & \myboxyyyyy{90.5} & \myboxyyyyy{95.3} \\
BGDF & \myboxyyyyy{4.25} & \myboxyyyyy{8.4} & \myboxyyyyy{12.6} & \myboxyyyyy{16.7} & \myboxyyyyy{20.9} & \myboxyyyyy{25.0} & \myboxyyyyy{29.2} & \myboxyyyyy{33.3} & \myboxyyyyy{37.5} & \myboxyyyyy{41.6} & \myboxyyyyy{45.8} & \myboxyyyyy{49.9} & \myboxyyyyy{54.1} & \myboxyyyyy{58.2} & \myboxyyyyy{62.4} & \myboxyyyyy{66.5} & \myboxyyyyy{70.7} & \myboxyyyyy{74.8} & \myboxyyyyy{79.0} & \myboxyyyyy{83.1} \\
BLFF+BM & \myboxx{1.82} & \myboxy{2.39} & \myboxy{2.65} & \myboxy{2.8} & \myboxy{2.88} & \myboxy{2.96} & \myboxy{3.02} & \myboxy{3.06} & \myboxy{3.11} & \myboxy{3.16} & \myboxy{3.2} & \myboxy{3.25} & \myboxy{3.3} & \myboxy{3.35} & \myboxy{3.42} & \myboxy{3.46} & \myboxy{3.51} & \myboxy{3.57} & \myboxy{3.63} & \myboxy{3.7} \\
BLFF+2UCB & \myboxyyyyy{3.94} & \myboxyyyyy{6.81} & \myboxyyyyy{9.15} & \myboxyyyyy{11.2} & \myboxyyyyy{13.0} & \myboxyyyyy{14.6} & \myboxyyyyy{16.0} & \myboxyyyyy{17.4} & \myboxyyyyy{18.6} & \myboxyyyyy{19.8} & \myboxyyyyy{20.9} & \myboxyyyyy{21.9} & \myboxyyyyy{22.9} & \myboxyyyyy{23.8} & \myboxyyyyy{24.7} & \myboxyyyyy{25.5} & \myboxyyyyy{26.3} & \myboxyyyyy{27.0} & \myboxyyyyy{27.8} & \myboxyyyyy{28.5} \\
BLFF+1UCB & \myboxyyyyy{3.26} & \myboxyyyyy{5.31} & \myboxyyyyy{6.9} & \myboxyyyyy{8.2} & \myboxyyyyy{9.33} & \myboxyyyyy{10.3} & \myboxyyyyy{11.2} & \myboxyyyyy{12.0} & \myboxyyyyy{12.7} & \myboxyyyyy{13.4} & \myboxyyyyy{14.0} & \myboxyyyyy{14.5} & \myboxyyyyy{15.1} & \myboxyyyyy{15.6} & \myboxyyyyy{16.0} & \myboxyyyyy{16.5} & \myboxyyyyy{16.9} & \myboxyyyyy{17.3} & \myboxyyyyy{17.7} & \myboxyyyyy{18.1} \\
BLFF+0.18UCB & \myboxy{1.89} & \myboxy{2.45} & \myboxy{2.85} & \myboxy{3.19} & \myboxy{3.43} & \myboxy{3.62} & \myboxy{3.8} & \myboxy{3.96} & \myboxy{4.12} & \myboxy{4.25} & \myboxy{4.39} & \myboxy{4.51} & \myboxy{4.62} & \myboxy{4.72} & \myboxyyy{4.84} & \myboxyyy{4.92} & \myboxyyy{5.01} & \myboxyyy{5.1} & \myboxyyy{5.19} & \myboxyyy{5.27} \\
BLFF+BMSF & \myboxy{1.96} & \myboxyyy{2.88} & \myboxyyy{3.42} & \myboxyyy{3.82} & \myboxyyy{4.06} & \myboxyyy{4.37} & \myboxyyy{4.49} & \myboxyyy{4.67} & \myboxyyy{4.82} & \myboxyyy{4.95} & \myboxyyy{5.08} & \myboxyyy{5.25} & \myboxyyy{5.3} & \myboxyyy{5.46} & \myboxyyy{5.55} & \myboxyyy{5.65} & \myboxyyy{5.7} & \myboxyyy{5.83} & \myboxyyy{5.88} & \myboxyyy{5.96} \\
1:1+BMSF & \myboxyyyyy{2.81} & \myboxyyyyy{4.15} & \myboxyyyyy{5.59} & \myboxyyyyy{6.24} & \myboxyyyyy{7.13} & \myboxyyyyy{7.83} & \myboxyyyyy{8.39} & \myboxyyyyy{8.82} & \myboxyyyyy{9.16} & \myboxyyyyy{9.41} & \myboxyyyyy{9.58} & \myboxyyyyy{9.66} & \myboxyyyyy{9.73} & \myboxyyyyy{10.4} & \myboxyyyyy{10.4} & \myboxyyyyy{10.3} & \myboxyyyyy{10.2} & \myboxyyyyy{10.7} & \myboxyyyyy{10.5} & \myboxyyyyy{11.0} \\
\bottomrule
\end{tabular}
\caption{The mean regret number of successes under success probabilities 0.7 and 0.9.}\label{table:performance79}
\end{sidewaystable}

\begin{sidewaystable}[tbp]
\footnotesize
\setlength{\tabcolsep}{5pt}
\centering
\begin{tabular}{l|cccccccccccccccccccc}
\toprule
Design $ \backslash $ Trial size & 60 & 120 & 180 & 240 & 300 & 360 & 420 & 480 & 540 & 600 & 660 & 720 & 780 & 840 & 900 & 960 & 1020 & 1080 & 1140 & 1200 \\
\midrule
1:1 & \myboxyyyyy{6.0} & \myboxyyyyy{12.0} & \myboxyyyyy{18.0} & \myboxyyyyy{24.0} & \myboxyyyyy{30.0} & \myboxyyyyy{36.0} & \myboxyyyyy{42.0} & \myboxyyyyy{48.0} & \myboxyyyyy{54.0} & \myboxyyyyy{60.0} & \myboxyyyyy{66.0} & \myboxyyyyy{72.0} & \myboxyyyyy{78.0} & \myboxyyyyy{84.0} & \myboxyyyyy{90.0} & \myboxyyyyy{96.0} & \myboxyyyyy{102} & \myboxyyyyy{108} & \myboxyyyyy{114} & \myboxyyyyy{120} \\
DP & \myboxx{2.3} & \myboxx{3.08} & \myboxx{3.47} & \myboxx{3.7} & \myboxx{3.86} & \myboxx{3.96} & \myboxx{4.05} & \myboxx{4.12} & \myboxx{4.19} & \myboxx{4.25} & \myboxx{4.3} & \myboxx{4.35} & \myboxx{4.39} & \myboxx{4.43} & \myboxx{4.48} & \myboxx{4.51} & \myboxx{4.55} & \myboxx{4.58} & \myboxx{4.62} & \myboxx{4.65} \\
BM & \myboxyyyyy{3.27} & \myboxyyyyy{6.1} & \myboxyyyyy{8.88} & \myboxyyyyy{11.6} & \myboxyyyyy{14.4} & \myboxyyyyy{17.1} & \myboxyyyyy{19.8} & \myboxyyyyy{22.5} & \myboxyyyyy{25.1} & \myboxyyyyy{27.8} & \myboxyyyyy{30.5} & \myboxyyyyy{33.2} & \myboxyyyyy{35.8} & \myboxyyyyy{38.5} & \myboxyyyyy{41.2} & \myboxyyyyy{43.8} & \myboxyyyyy{46.5} & \myboxyyyyy{49.1} & \myboxyyyyy{51.8} & \myboxyyyyy{54.4} \\
FM & \myboxyyyyy{3.57} & \myboxyyyyy{6.74} & \myboxyyyyy{9.86} & \myboxyyyyy{13.0} & \myboxyyyyy{16.0} & \myboxyyyyy{19.1} & \myboxyyyyy{22.2} & \myboxyyyyy{25.3} & \myboxyyyyy{28.3} & \myboxyyyyy{31.4} & \myboxyyyyy{34.4} & \myboxyyyyy{37.5} & \myboxyyyyy{40.5} & \myboxyyyyy{43.6} & \myboxyyyyy{46.6} & \myboxyyyyy{49.7} & \myboxyyyyy{52.7} & \myboxyyyyy{55.8} & \myboxyyyyy{58.8} & \myboxyyyyy{61.9} \\
2UCB & \myboxyyyyy{3.96} & \myboxyyyyy{6.84} & \myboxyyyyy{9.22} & \myboxyyyyy{11.3} & \myboxyyyyy{13.1} & \myboxyyyyy{14.7} & \myboxyyyyy{16.2} & \myboxyyyyy{17.5} & \myboxyyyyy{18.8} & \myboxyyyyy{19.9} & \myboxyyyyy{21.0} & \myboxyyyyy{22.0} & \myboxyyyyy{23.0} & \myboxyyyyy{23.9} & \myboxyyyyy{24.8} & \myboxyyyyy{25.6} & \myboxyyyyy{26.4} & \myboxyyyyy{27.1} & \myboxyyyyy{27.9} & \myboxyyyyy{28.6} \\
1UCB & \myboxyyyyy{3.36} & \myboxyyyyy{5.45} & \myboxyyyyy{7.05} & \myboxyyyyy{8.36} & \myboxyyyyy{9.48} & \myboxyyyyy{10.5} & \myboxyyyyy{11.3} & \myboxyyyyy{12.1} & \myboxyyyyy{12.8} & \myboxyyyyy{13.5} & \myboxyyyyy{14.1} & \myboxyyyyy{14.6} & \myboxyyyyy{15.2} & \myboxyyyyy{15.6} & \myboxyyyyy{16.1} & \myboxyyyyy{16.5} & \myboxyyyyy{17.0} & \myboxyyyyy{17.4} & \myboxyyyyy{17.7} & \myboxyyyyy{18.1} \\
0.18UCB & \myboxyyy{2.6} & \myboxyyy{3.67} & \myboxyyy{4.27} & \myboxyyy{4.7} & \myboxyyy{5.05} & \myboxyyy{5.36} & \myboxyyy{5.64} & \myboxyyy{5.9} & \myboxyyy{6.15} & \myboxyyy{6.36} & \myboxyyy{6.54} & \myboxyyy{6.69} & \myboxyyy{6.82} & \myboxyyy{6.93} & \myboxyyy{7.04} & \myboxyyy{7.13} & \myboxyyy{7.21} & \myboxyyy{7.3} & \myboxyyy{7.37} & \myboxyyy{7.45} \\
BKG & \myboxy{2.41} & \myboxyyy{3.83} & \myboxyyyyy{5.11} & \myboxyyyyy{6.33} & \myboxyyyyy{7.52} & \myboxyyyyy{8.68} & \myboxyyyyy{9.83} & \myboxyyyyy{11.0} & \myboxyyyyy{12.1} & \myboxyyyyy{13.2} & \myboxyyyyy{14.3} & \myboxyyyyy{15.4} & \myboxyyyyy{16.5} & \myboxyyyyy{17.6} & \myboxyyyyy{18.7} & \myboxyyyyy{19.8} & \myboxyyyyy{20.9} & \myboxyyyyy{22.0} & \myboxyyyyy{23.0} & \myboxyyyyy{24.1} \\
BLFF & \myboxyyyyy{4.42} & \myboxyyyyy{8.91} & \myboxyyyyy{13.4} & \myboxyyyyy{17.9} & \myboxyyyyy{22.4} & \myboxyyyyy{26.9} & \myboxyyyyy{31.4} & \myboxyyyyy{35.9} & \myboxyyyyy{40.4} & \myboxyyyyy{44.9} & \myboxyyyyy{49.4} & \myboxyyyyy{53.9} & \myboxyyyyy{58.4} & \myboxyyyyy{62.9} & \myboxyyyyy{67.4} & \myboxyyyyy{71.9} & \myboxyyyyy{76.4} & \myboxyyyyy{80.9} & \myboxyyyyy{85.4} & \myboxyyyyy{89.9} \\
BMSF & \myboxyyyyy{4.62} & \myboxyyyyy{9.21} & \myboxyyyyy{13.8} & \myboxyyyyy{18.4} & \myboxyyyyy{23.0} & \myboxyyyyy{27.6} & \myboxyyyyy{32.1} & \myboxyyyyy{36.7} & \myboxyyyyy{41.3} & \myboxyyyyy{45.9} & \myboxyyyyy{50.5} & \myboxyyyyy{55.1} & \myboxyyyyy{59.7} & \myboxyyyyy{64.3} & \myboxyyyyy{68.9} & \myboxyyyyy{73.4} & \myboxyyyyy{78.0} & \myboxyyyyy{82.6} & \myboxyyyyy{87.2} & \myboxyyyyy{91.8} \\
BGDF & \myboxyyyyy{2.83} & \myboxyyyyy{4.32} & \myboxyyyyy{5.47} & \myboxyyyyy{6.44} & \myboxyyyyy{7.29} & \myboxyyyyy{8.06} & \myboxyyyyy{8.77} & \myboxyyyyy{9.43} & \myboxyyyyy{10.1} & \myboxyyyyy{10.6} & \myboxyyyyy{11.2} & \myboxyyyyy{11.7} & \myboxyyyyy{12.2} & \myboxyyyyy{12.7} & \myboxyyyyy{13.2} & \myboxyyyyy{13.7} & \myboxyyyyy{14.1} & \myboxyyyyy{14.5} & \myboxyyyyy{15.0} & \myboxyyyyy{15.4} \\
BLFF+BM & \myboxyyy{2.47} & \myboxyyy{3.56} & \myboxyyy{4.22} & \myboxyyy{4.68} & \myboxyyy{5.0} & \myboxyyy{5.27} & \myboxyyy{5.48} & \myboxyyy{5.65} & \myboxyyy{5.78} & \myboxyyy{5.93} & \myboxyyy{6.02} & \myboxyyy{6.13} & \myboxy{6.23} & \myboxy{6.32} & \myboxy{6.39} & \myboxy{6.48} & \myboxy{6.56} & \myboxy{6.63} & \myboxy{6.7} & \myboxy{6.77} \\
BLFF+2UCB & \myboxyyyyy{3.96} & \myboxyyyyy{6.84} & \myboxyyyyy{9.22} & \myboxyyyyy{11.3} & \myboxyyyyy{13.1} & \myboxyyyyy{14.7} & \myboxyyyyy{16.2} & \myboxyyyyy{17.5} & \myboxyyyyy{18.8} & \myboxyyyyy{19.9} & \myboxyyyyy{21.0} & \myboxyyyyy{22.0} & \myboxyyyyy{23.0} & \myboxyyyyy{23.9} & \myboxyyyyy{24.8} & \myboxyyyyy{25.6} & \myboxyyyyy{26.4} & \myboxyyyyy{27.1} & \myboxyyyyy{27.9} & \myboxyyyyy{28.6} \\
BLFF+1UCB & \myboxyyyyy{3.36} & \myboxyyyyy{5.45} & \myboxyyyyy{7.05} & \myboxyyyyy{8.36} & \myboxyyyyy{9.48} & \myboxyyyyy{10.5} & \myboxyyyyy{11.3} & \myboxyyyyy{12.1} & \myboxyyyyy{12.8} & \myboxyyyyy{13.5} & \myboxyyyyy{14.1} & \myboxyyyyy{14.6} & \myboxyyyyy{15.2} & \myboxyyyyy{15.6} & \myboxyyyyy{16.1} & \myboxyyyyy{16.6} & \myboxyyyyy{17.0} & \myboxyyyyy{17.4} & \myboxyyyyy{17.7} & \myboxyyyyy{18.1} \\
BLFF+0.18UCB & \myboxy{2.44} & \myboxy{3.33} & \myboxy{3.83} & \myboxy{4.17} & \myboxy{4.45} & \myboxy{4.67} & \myboxy{4.87} & \myboxy{5.06} & \myboxy{5.24} & \myboxy{5.38} & \myboxy{5.54} & \myboxy{5.68} & \myboxy{5.8} & \myboxy{5.93} & \myboxy{6.09} & \myboxy{6.17} & \myboxy{6.29} & \myboxy{6.41} & \myboxy{6.52} & \myboxy{6.64} \\
BLFF+BMSF & \myboxyyyyy{2.84} & \myboxyyyyy{4.49} & \myboxyyyyy{5.63} & \myboxyyyyy{6.53} & \myboxyyyyy{7.17} & \myboxyyyyy{7.86} & \myboxyyyyy{8.25} & \myboxyyyyy{8.69} & \myboxyyyyy{9.07} & \myboxyyyyy{9.4} & \myboxyyyyy{9.69} & \myboxyyyyy{10.1} & \myboxyyyyy{10.2} & \myboxyyyyy{10.6} & \myboxyyyyy{10.7} & \myboxyyyyy{11.0} & \myboxyyyyy{11.0} & \myboxyyyyy{11.3} & \myboxyyyyy{11.4} & \myboxyyyyy{11.5} \\
1:1+BMSF & \myboxyyyyy{3.33} & \myboxyyyyy{5.42} & \myboxyyyyy{7.53} & \myboxyyyyy{8.87} & \myboxyyyyy{10.4} & \myboxyyyyy{11.7} & \myboxyyyyy{12.9} & \myboxyyyyy{13.9} & \myboxyyyyy{14.7} & \myboxyyyyy{15.5} & \myboxyyyyy{16.1} & \myboxyyyyy{16.7} & \myboxyyyyy{17.1} & \myboxyyyyy{18.3} & \myboxyyyyy{18.7} & \myboxyyyyy{18.9} & \myboxyyyyy{19.1} & \myboxyyyyy{20.1} & \myboxyyyyy{20.2} & \myboxyyyyy{21.2} \\
\bottomrule
\end{tabular}
\caption{The mean regret number of successes under success probabilities 0.5 and 0.7.}\label{table:performance57}
\end{sidewaystable}

\begin{sidewaystable}[tbp]
\footnotesize
\setlength{\tabcolsep}{5pt}
\centering
\begin{tabular}{l|cccccccccccccccccccc}
\toprule
Design $ \backslash $ Trial size & 60 & 120 & 180 & 240 & 300 & 360 & 420 & 480 & 540 & 600 & 660 & 720 & 780 & 840 & 900 & 960 & 1020 & 1080 & 1140 & 1200 \\
\midrule
1:1 & \myboxyyyyy{6.0} & \myboxyyyyy{12.0} & \myboxyyyyy{18.0} & \myboxyyyyy{24.0} & \myboxyyyyy{30.0} & \myboxyyyyy{36.0} & \myboxyyyyy{42.0} & \myboxyyyyy{48.0} & \myboxyyyyy{54.0} & \myboxyyyyy{60.0} & \myboxyyyyy{66.0} & \myboxyyyyy{72.0} & \myboxyyyyy{78.0} & \myboxyyyyy{84.0} & \myboxyyyyy{90.0} & \myboxyyyyy{96.0} & \myboxyyyyy{102} & \myboxyyyyy{108} & \myboxyyyyy{114} & \myboxyyyyy{120} \\
DP & \myboxx{2.33} & \myboxx{3.1} & \myboxx{3.48} & \myboxx{3.72} & \myboxx{3.89} & \myboxx{4.02} & \myboxx{4.13} & \myboxx{4.22} & \myboxx{4.3} & \myboxx{4.38} & \myboxx{4.45} & \myboxx{4.52} & \myboxx{4.58} & \myboxx{4.63} & \myboxx{4.69} & \myboxx{4.74} & \myboxx{4.79} & \myboxx{4.84} & \myboxx{4.88} & \myboxx{4.92} \\
BM & \myboxyyyyy{2.74} & \myboxyyyyy{4.7} & \myboxyyyyy{6.55} & \myboxyyyyy{8.35} & \myboxyyyyy{10.1} & \myboxyyyyy{11.9} & \myboxyyyyy{13.6} & \myboxyyyyy{15.4} & \myboxyyyyy{17.1} & \myboxyyyyy{18.9} & \myboxyyyyy{20.6} & \myboxyyyyy{22.3} & \myboxyyyyy{24.1} & \myboxyyyyy{25.8} & \myboxyyyyy{27.5} & \myboxyyyyy{29.3} & \myboxyyyyy{31.0} & \myboxyyyyy{32.7} & \myboxyyyyy{34.4} & \myboxyyyyy{36.2} \\
FM & \myboxyyyyy{3.86} & \myboxyyyyy{7.45} & \myboxyyyyy{11.0} & \myboxyyyyy{14.6} & \myboxyyyyy{18.1} & \myboxyyyyy{21.7} & \myboxyyyyy{25.2} & \myboxyyyyy{28.8} & \myboxyyyyy{32.3} & \myboxyyyyy{35.9} & \myboxyyyyy{39.4} & \myboxyyyyy{43.0} & \myboxyyyyy{46.5} & \myboxyyyyy{50.1} & \myboxyyyyy{53.6} & \myboxyyyyy{57.2} & \myboxyyyyy{60.7} & \myboxyyyyy{64.2} & \myboxyyyyy{67.8} & \myboxyyyyy{71.3} \\
2UCB & \myboxyyyyy{4.01} & \myboxyyyyy{6.92} & \myboxyyyyy{9.31} & \myboxyyyyy{11.4} & \myboxyyyyy{13.2} & \myboxyyyyy{14.8} & \myboxyyyyy{16.3} & \myboxyyyyy{17.6} & \myboxyyyyy{18.9} & \myboxyyyyy{20.0} & \myboxyyyyy{21.1} & \myboxyyyyy{22.1} & \myboxyyyyy{23.1} & \myboxyyyyy{24.0} & \myboxyyyyy{24.9} & \myboxyyyyy{25.7} & \myboxyyyyy{26.5} & \myboxyyyyy{27.3} & \myboxyyyyy{28.0} & \myboxyyyyy{28.7} \\
1UCB & \myboxyyyyy{3.44} & \myboxyyyyy{5.56} & \myboxyyyyy{7.18} & \myboxyyyyy{8.49} & \myboxyyyyy{9.61} & \myboxyyyyy{10.6} & \myboxyyyyy{11.5} & \myboxyyyyy{12.2} & \myboxyyyyy{12.9} & \myboxyyyyy{13.6} & \myboxyyyyy{14.2} & \myboxyyyyy{14.8} & \myboxyyyyy{15.3} & \myboxyyyyy{15.8} & \myboxyyyyy{16.2} & \myboxyyyyy{16.7} & \myboxyyyyy{17.1} & \myboxyyyyy{17.5} & \myboxyyyyy{17.9} & \myboxyyyyy{18.2} \\
0.18UCB & \myboxy{2.48} & \myboxy{3.39} & \myboxy{3.88} & \myboxy{4.21} & \myboxy{4.46} & \myboxy{4.65} & \myboxy{4.8} & \myboxy{4.93} & \myboxy{5.05} & \myboxy{5.15} & \myboxy{5.26} & \myboxy{5.35} & \myboxy{5.43} & \myboxy{5.51} & \myboxy{5.58} & \myboxy{5.64} & \myboxy{5.7} & \myboxy{5.75} & \myboxy{5.8} & \myboxy{5.85} \\
BKG & \myboxx{2.31} & \myboxy{3.13} & \myboxy{3.72} & \myboxy{4.23} & \myboxy{4.72} & \myboxyyy{5.18} & \myboxyyy{5.63} & \myboxyyy{6.08} & \myboxyyy{6.51} & \myboxyyy{6.95} & \myboxyyy{7.37} & \myboxyyy{7.8} & \myboxyyy{8.22} & \myboxyyyyy{8.64} & \myboxyyyyy{9.06} & \myboxyyyyy{9.48} & \myboxyyyyy{9.89} & \myboxyyyyy{10.3} & \myboxyyyyy{10.7} & \myboxyyyyy{11.1} \\
BLFF & \myboxyyyyy{4.94} & \myboxyyyyy{9.93} & \myboxyyyyy{14.9} & \myboxyyyyy{19.9} & \myboxyyyyy{24.9} & \myboxyyyyy{29.9} & \myboxyyyyy{34.9} & \myboxyyyyy{39.9} & \myboxyyyyy{44.9} & \myboxyyyyy{49.9} & \myboxyyyyy{54.9} & \myboxyyyyy{59.9} & \myboxyyyyy{64.9} & \myboxyyyyy{69.9} & \myboxyyyyy{74.9} & \myboxyyyyy{79.9} & \myboxyyyyy{84.9} & \myboxyyyyy{89.9} & \myboxyyyyy{94.9} & \myboxyyyyy{99.9} \\
BMSF & \myboxyyyyy{4.22} & \myboxyyyyy{8.37} & \myboxyyyyy{12.5} & \myboxyyyyy{16.7} & \myboxyyyyy{20.8} & \myboxyyyyy{25.0} & \myboxyyyyy{29.1} & \myboxyyyyy{33.3} & \myboxyyyyy{37.4} & \myboxyyyyy{41.6} & \myboxyyyyy{45.8} & \myboxyyyyy{49.9} & \myboxyyyyy{54.1} & \myboxyyyyy{58.2} & \myboxyyyyy{62.4} & \myboxyyyyy{66.5} & \myboxyyyyy{70.7} & \myboxyyyyy{74.8} & \myboxyyyyy{79.0} & \myboxyyyyy{83.1} \\
BGDF & \myboxyyy{2.66} & \myboxyyyyy{3.96} & \myboxyyyyy{4.95} & \myboxyyyyy{5.78} & \myboxyyyyy{6.51} & \myboxyyyyy{7.17} & \myboxyyyyy{7.78} & \myboxyyyyy{8.35} & \myboxyyyyy{8.88} & \myboxyyyyy{9.38} & \myboxyyyyy{9.86} & \myboxyyyyy{10.3} & \myboxyyyyy{10.8} & \myboxyyyyy{11.2} & \myboxyyyyy{11.6} & \myboxyyyyy{12.0} & \myboxyyyyy{12.4} & \myboxyyyyy{12.7} & \myboxyyyyy{13.1} & \myboxyyyyy{13.4} \\
BLFF+BM & \myboxyyy{2.55} & \myboxyyy{3.62} & \myboxyyy{4.29} & \myboxyyy{4.76} & \myboxyyy{5.1} & \myboxyyy{5.39} & \myboxyyy{5.62} & \myboxyyy{5.82} & \myboxyyy{5.99} & \myboxyyy{6.15} & \myboxyyy{6.29} & \myboxyyy{6.42} & \myboxyyy{6.54} & \myboxyyy{6.65} & \myboxyyy{6.77} & \myboxyyy{6.86} & \myboxyyy{6.96} & \myboxy{7.06} & \myboxy{7.16} & \myboxy{7.26} \\
BLFF+2UCB & \myboxyyyyy{4.01} & \myboxyyyyy{6.92} & \myboxyyyyy{9.31} & \myboxyyyyy{11.4} & \myboxyyyyy{13.2} & \myboxyyyyy{14.8} & \myboxyyyyy{16.3} & \myboxyyyyy{17.6} & \myboxyyyyy{18.9} & \myboxyyyyy{20.0} & \myboxyyyyy{21.1} & \myboxyyyyy{22.1} & \myboxyyyyy{23.1} & \myboxyyyyy{24.0} & \myboxyyyyy{24.9} & \myboxyyyyy{25.7} & \myboxyyyyy{26.5} & \myboxyyyyy{27.3} & \myboxyyyyy{28.0} & \myboxyyyyy{28.7} \\
BLFF+1UCB & \myboxyyyyy{3.44} & \myboxyyyyy{5.56} & \myboxyyyyy{7.18} & \myboxyyyyy{8.49} & \myboxyyyyy{9.61} & \myboxyyyyy{10.6} & \myboxyyyyy{11.5} & \myboxyyyyy{12.2} & \myboxyyyyy{12.9} & \myboxyyyyy{13.6} & \myboxyyyyy{14.2} & \myboxyyyyy{14.8} & \myboxyyyyy{15.3} & \myboxyyyyy{15.8} & \myboxyyyyy{16.2} & \myboxyyyyy{16.7} & \myboxyyyyy{17.1} & \myboxyyyyy{17.5} & \myboxyyyyy{17.9} & \myboxyyyyy{18.2} \\
BLFF+0.18UCB & \myboxyyy{2.56} & \myboxy{3.48} & \myboxy{4.0} & \myboxy{4.35} & \myboxy{4.65} & \myboxy{4.88} & \myboxy{5.09} & \myboxy{5.29} & \myboxy{5.49} & \myboxy{5.63} & \myboxy{5.82} & \myboxy{5.96} & \myboxy{6.1} & \myboxy{6.23} & \myboxy{6.42} & \myboxy{6.5} & \myboxy{6.63} & \myboxy{6.77} & \myboxy{6.9} & \myboxy{7.03} \\
BLFF+BMSF & \myboxyyyyy{3.01} & \myboxyyyyy{4.72} & \myboxyyyyy{5.92} & \myboxyyyyy{6.86} & \myboxyyyyy{7.54} & \myboxyyyyy{8.27} & \myboxyyyyy{8.7} & \myboxyyyyy{9.17} & \myboxyyyyy{9.59} & \myboxyyyyy{9.95} & \myboxyyyyy{10.3} & \myboxyyyyy{10.7} & \myboxyyyyy{10.9} & \myboxyyyyy{11.2} & \myboxyyyyy{11.4} & \myboxyyyyy{11.7} & \myboxyyyyy{11.8} & \myboxyyyyy{12.1} & \myboxyyyyy{12.1} & \myboxyyyyy{12.3} \\
1:1+BMSF & \myboxyyyyy{3.32} & \myboxyyyyy{5.42} & \myboxyyyyy{7.51} & \myboxyyyyy{8.86} & \myboxyyyyy{10.4} & \myboxyyyyy{11.7} & \myboxyyyyy{12.9} & \myboxyyyyy{13.9} & \myboxyyyyy{14.7} & \myboxyyyyy{15.5} & \myboxyyyyy{16.2} & \myboxyyyyy{16.7} & \myboxyyyyy{17.2} & \myboxyyyyy{18.3} & \myboxyyyyy{18.7} & \myboxyyyyy{18.9} & \myboxyyyyy{19.2} & \myboxyyyyy{20.2} & \myboxyyyyy{20.3} & \myboxyyyyy{21.2} \\
\bottomrule
\end{tabular}
\caption{The mean regret number of successes under success probabilities 0.3 and 0.5.}\label{table:performance35}
\end{sidewaystable}

\begin{sidewaystable}[tbp]
\footnotesize
\setlength{\tabcolsep}{5pt}
\centering
\begin{tabular}{l|cccccccccccccccccccc}
\toprule
Design $ \backslash $ Trial size & 60 & 120 & 180 & 240 & 300 & 360 & 420 & 480 & 540 & 600 & 660 & 720 & 780 & 840 & 900 & 960 & 1020 & 1080 & 1140 & 1200 \\
\midrule
1:1 & \myboxyyyyy{6.0} & \myboxyyyyy{12.0} & \myboxyyyyy{18.0} & \myboxyyyyy{24.0} & \myboxyyyyy{30.0} & \myboxyyyyy{36.0} & \myboxyyyyy{42.0} & \myboxyyyyy{48.0} & \myboxyyyyy{54.0} & \myboxyyyyy{60.0} & \myboxyyyyy{66.0} & \myboxyyyyy{72.0} & \myboxyyyyy{78.0} & \myboxyyyyy{84.0} & \myboxyyyyy{90.0} & \myboxyyyyy{96.0} & \myboxyyyyy{102} & \myboxyyyyy{108} & \myboxyyyyy{114} & \myboxyyyyy{120} \\
DP & \myboxx{2.02} & \myboxx{2.51} & \myboxx{2.77} & \myboxx{2.95} & \myboxx{3.1} & \myboxx{3.22} & \myboxx{3.33} & \myboxx{3.42} & \myboxx{3.51} & \myboxx{3.59} & \myboxx{3.66} & \myboxx{3.73} & \myboxx{3.8} & \myboxx{3.86} & \myboxx{3.91} & \myboxx{3.97} & \myboxx{4.02} & \myboxx{4.07} & \myboxx{4.11} & \myboxx{4.16} \\
BM & \myboxx{1.84} & \myboxx{2.37} & \myboxx{2.75} & \myboxy{3.07} & \myboxy{3.36} & \myboxy{3.64} & \myboxy{3.9} & \myboxy{4.16} & \myboxy{4.41} & \myboxy{4.66} & \myboxy{4.9} & \myboxy{5.14} & \myboxy{5.38} & \myboxy{5.61} & \myboxy{5.84} & \myboxy{6.07} & \myboxyyy{6.3} & \myboxyyy{6.53} & \myboxyyy{6.76} & \myboxyyy{6.98} \\
FM & \myboxyyyyy{3.05} & \myboxyyyyy{5.8} & \myboxyyyyy{8.54} & \myboxyyyyy{11.3} & \myboxyyyyy{14.0} & \myboxyyyyy{16.7} & \myboxyyyyy{19.5} & \myboxyyyyy{22.2} & \myboxyyyyy{25.0} & \myboxyyyyy{27.7} & \myboxyyyyy{30.4} & \myboxyyyyy{33.2} & \myboxyyyyy{35.9} & \myboxyyyyy{38.6} & \myboxyyyyy{41.4} & \myboxyyyyy{44.1} & \myboxyyyyy{46.8} & \myboxyyyyy{49.6} & \myboxyyyyy{52.3} & \myboxyyyyy{55.0} \\
2UCB & \myboxyyyyy{4.04} & \myboxyyyyy{6.96} & \myboxyyyyy{9.36} & \myboxyyyyy{11.4} & \myboxyyyyy{13.2} & \myboxyyyyy{14.9} & \myboxyyyyy{16.3} & \myboxyyyyy{17.7} & \myboxyyyyy{18.9} & \myboxyyyyy{20.1} & \myboxyyyyy{21.2} & \myboxyyyyy{22.2} & \myboxyyyyy{23.2} & \myboxyyyyy{24.1} & \myboxyyyyy{25.0} & \myboxyyyyy{25.8} & \myboxyyyyy{26.6} & \myboxyyyyy{27.4} & \myboxyyyyy{28.1} & \myboxyyyyy{28.8} \\
1UCB & \myboxyyyyy{3.46} & \myboxyyyyy{5.6} & \myboxyyyyy{7.23} & \myboxyyyyy{8.55} & \myboxyyyyy{9.68} & \myboxyyyyy{10.7} & \myboxyyyyy{11.5} & \myboxyyyyy{12.3} & \myboxyyyyy{13.0} & \myboxyyyyy{13.7} & \myboxyyyyy{14.3} & \myboxyyyyy{14.9} & \myboxyyyyy{15.4} & \myboxyyyyy{15.9} & \myboxyyyyy{16.4} & \myboxyyyyy{16.8} & \myboxyyyyy{17.2} & \myboxyyyyy{17.6} & \myboxyyyyy{18.0} & \myboxyyyyy{18.4} \\
0.18UCB & \myboxy{2.14} & \myboxy{2.77} & \myboxy{3.12} & \myboxy{3.38} & \myboxy{3.58} & \myboxy{3.75} & \myboxy{3.9} & \myboxy{4.03} & \myboxy{4.14} & \myboxy{4.24} & \myboxy{4.34} & \myboxy{4.43} & \myboxy{4.51} & \myboxy{4.58} & \myboxy{4.65} & \myboxy{4.72} & \myboxy{4.78} & \myboxy{4.84} & \myboxy{4.9} & \myboxy{4.95} \\
BKG & \myboxy{2.15} & \myboxy{2.52} & \myboxx{2.62} & \myboxx{2.67} & \myboxx{2.7} & \myboxx{2.72} & \myboxx{2.74} & \myboxx{2.76} & \myboxx{2.77} & \myboxx{2.78} & \myboxx{2.8} & \myboxx{2.81} & \myboxx{2.82} & \myboxx{2.83} & \myboxx{2.84} & \myboxx{2.85} & \myboxx{2.86} & \myboxx{2.87} & \myboxx{2.88} & \myboxx{2.89} \\
BLFF & \myboxyyyyy{5.2} & \myboxyyyyy{10.4} & \myboxyyyyy{15.7} & \myboxyyyyy{20.9} & \myboxyyyyy{26.2} & \myboxyyyyy{31.4} & \myboxyyyyy{36.7} & \myboxyyyyy{41.9} & \myboxyyyyy{47.2} & \myboxyyyyy{52.4} & \myboxyyyyy{57.7} & \myboxyyyyy{62.9} & \myboxyyyyy{68.2} & \myboxyyyyy{73.4} & \myboxyyyyy{78.7} & \myboxyyyyy{83.9} & \myboxyyyyy{89.2} & \myboxyyyyy{94.4} & \myboxyyyyy{99.7} & \myboxyyyyy{105} \\
BMSF & \myboxyyyyy{2.99} & \myboxyyyyy{5.75} & \myboxyyyyy{8.51} & \myboxyyyyy{11.3} & \myboxyyyyy{14.0} & \myboxyyyyy{16.8} & \myboxyyyyy{19.5} & \myboxyyyyy{22.3} & \myboxyyyyy{25.0} & \myboxyyyyy{27.8} & \myboxyyyyy{30.6} & \myboxyyyyy{33.3} & \myboxyyyyy{36.1} & \myboxyyyyy{38.8} & \myboxyyyyy{41.6} & \myboxyyyyy{44.3} & \myboxyyyyy{47.1} & \myboxyyyyy{49.9} & \myboxyyyyy{52.6} & \myboxyyyyy{55.4} \\
BGDF & \myboxyyyyy{4.16} & \myboxyyyyy{8.16} & \myboxyyyyy{12.2} & \myboxyyyyy{16.2} & \myboxyyyyy{20.2} & \myboxyyyyy{24.2} & \myboxyyyyy{28.2} & \myboxyyyyy{32.2} & \myboxyyyyy{36.2} & \myboxyyyyy{40.2} & \myboxyyyyy{44.2} & \myboxyyyyy{48.2} & \myboxyyyyy{52.2} & \myboxyyyyy{56.2} & \myboxyyyyy{60.2} & \myboxyyyyy{64.2} & \myboxyyyyy{68.2} & \myboxyyyyy{72.2} & \myboxyyyyy{76.2} & \myboxyyyyy{80.2} \\
BLFF+BM & \myboxy{2.13} & \myboxy{2.8} & \myboxy{3.22} & \myboxy{3.52} & \myboxy{3.8} & \myboxy{4.02} & \myboxy{4.23} & \myboxy{4.42} & \myboxy{4.63} & \myboxy{4.77} & \myboxy{4.97} & \myboxy{5.11} & \myboxy{5.26} & \myboxy{5.4} & \myboxy{5.61} & \myboxy{5.69} & \myboxy{5.84} & \myboxy{5.99} & \myboxy{6.14} & \myboxy{6.29} \\
BLFF+2UCB & \myboxyyyyy{4.04} & \myboxyyyyy{6.96} & \myboxyyyyy{9.36} & \myboxyyyyy{11.4} & \myboxyyyyy{13.2} & \myboxyyyyy{14.9} & \myboxyyyyy{16.3} & \myboxyyyyy{17.7} & \myboxyyyyy{18.9} & \myboxyyyyy{20.1} & \myboxyyyyy{21.2} & \myboxyyyyy{22.2} & \myboxyyyyy{23.2} & \myboxyyyyy{24.1} & \myboxyyyyy{25.0} & \myboxyyyyy{25.8} & \myboxyyyyy{26.6} & \myboxyyyyy{27.4} & \myboxyyyyy{28.1} & \myboxyyyyy{28.8} \\
BLFF+1UCB & \myboxyyyyy{3.46} & \myboxyyyyy{5.6} & \myboxyyyyy{7.23} & \myboxyyyyy{8.55} & \myboxyyyyy{9.68} & \myboxyyyyy{10.7} & \myboxyyyyy{11.5} & \myboxyyyyy{12.3} & \myboxyyyyy{13.0} & \myboxyyyyy{13.7} & \myboxyyyyy{14.3} & \myboxyyyyy{14.9} & \myboxyyyyy{15.4} & \myboxyyyyy{15.9} & \myboxyyyyy{16.4} & \myboxyyyyy{16.8} & \myboxyyyyy{17.2} & \myboxyyyyy{17.6} & \myboxyyyyy{18.0} & \myboxyyyyy{18.4} \\
BLFF+0.18UCB & \myboxyyy{2.23} & \myboxy{2.95} & \myboxy{3.4} & \myboxy{3.74} & \myboxy{4.04} & \myboxy{4.28} & \myboxy{4.51} & \myboxy{4.73} & \myboxy{4.95} & \myboxy{5.11} & \myboxy{5.31} & \myboxy{5.47} & \myboxy{5.62} & \myboxy{5.77} & \myboxy{5.97} & \myboxy{6.06} & \myboxy{6.21} & \myboxy{6.35} & \myboxy{6.5} & \myboxyyy{6.65} \\
BLFF+BMSF & \myboxyyyyy{2.66} & \myboxyyyyy{3.88} & \myboxyyyyy{4.72} & \myboxyyyyy{5.34} & \myboxyyyyy{5.79} & \myboxyyyyy{6.28} & \myboxyyyyy{6.62} & \myboxyyyyy{6.95} & \myboxyyyyy{7.26} & \myboxyyyyy{7.56} & \myboxyyyyy{7.84} & \myboxyyyyy{8.12} & \myboxyyyyy{8.31} & \myboxyyyyy{8.57} & \myboxyyyyy{8.84} & \myboxyyyyy{9.01} & \myboxyyyyy{9.18} & \myboxyyyyy{9.43} & \myboxyyyyy{9.61} & \myboxyyyyy{9.77} \\
1:1+BMSF & \myboxyyyyy{2.61} & \myboxyyyyy{4.03} & \myboxyyyyy{5.38} & \myboxyyyyy{6.12} & \myboxyyyyy{7.02} & \myboxyyyyy{7.72} & \myboxyyyyy{8.3} & \myboxyyyyy{8.78} & \myboxyyyyy{9.18} & \myboxyyyyy{9.5} & \myboxyyyyy{9.76} & \myboxyyyyy{9.89} & \myboxyyyyy{10.0} & \myboxyyyyy{10.7} & \myboxyyyyy{10.7} & \myboxyyyyy{10.7} & \myboxyyyyy{10.7} & \myboxyyyyy{11.2} & \myboxyyyyy{11.2} & \myboxyyyyy{11.6} \\
\bottomrule
\end{tabular}
\caption{The mean regret number of successes under success probabilities 0.1 and 0.3.}\label{table:performance13}
\end{sidewaystable}

\end{document}